%% file: Exact_Limit_Theorems_200818.tex
\documentclass[12pt]{amsart}
\usepackage[margin=1in]{geometry}
\numberwithin{equation}{section}

\usepackage{amssymb, amsthm, mathtools}
\usepackage{enumerate, xspace, enumitem, cases}
\usepackage{graphicx, tikz-cd}
\usepackage[colorlinks]{hyperref}
\usepackage{cleveref}
\usepackage{etoolbox}
\usepackage{changebar}
\usepackage{comment}
\usepackage[sans]{dsfont} 
\usepackage{xcolor}

\makeatletter
\renewcommand\part{%
   \if@noskipsec \leavevmode \fi
   \par
   \addvspace{4ex}%
   \@afterindentfalse
   \secdef\@part\@spart}

\def\@part[#1]#2{%
    \ifnum \c@secnumdepth >\m@ne
      \refstepcounter{part}%
      \addcontentsline{toc}{part}{\thepart\hspace{1em}#1}%
    \else
      \addcontentsline{toc}{part}{#1} 
    \fi
    {\parindent \z@ \raggedright
     \interlinepenalty \@M
     \normalfont
     \ifnum \c@secnumdepth >\m@ne
       \normalsize\scshape \partname\nobreakspace\thepart
     \fi
     \normalsize \scshape \centering \hspace{3pt} $-$ \hspace{1pt}#2%
     \par}%
    \nobreak
    \vskip 3ex
    \@afterheading}
\def\@spart#1{%
    {\parindent \z@ \raggedright
     \interlinepenalty \@M
     \normalfont
     \normalsize \bfseries \scshape #1\par}%
     \nobreak
     \vskip 3ex
     \@afterheading}
\makeatother
\renewcommand\thepart{\Roman{part}}

\makeatletter
\def\@tocline#1#2#3#4#5#6#7{\relax
  \ifnum #1>\c@tocdepth 
  \else
    \par \addpenalty\@secpenalty\addvspace{#2}%
    \begingroup \hyphenpenalty\@M
    \@ifempty{#4}{%
      \@tempdima\csname r@tocindent\number#1\endcsname\relax
    }{%
      \@tempdima#4\relax
    }%
    \parindent\z@ \leftskip#3\relax \advance\leftskip\@tempdima\relax
    \rightskip\@pnumwidth plus4em \parfillskip-\@pnumwidth
    #5\leavevmode\hskip-\@tempdima
      \ifcase #1
       \or\or \hskip 1em \or \hskip 2em \else \hskip 1em \fi%
      #6\nobreak\relax
    \hfill\hbox to\@pnumwidth{\@tocpagenum{#7}}\par
    \nobreak
    \endgroup
  \fi}
\makeatother

\usepackage[backgroundcolor=blue!20!white, linecolor=blue!20!white,textsize=footnotesize]{todonotes}

\allowdisplaybreaks[4]
\newtheorem{thm}{Theorem}[section]
\newtheorem{lem}[thm]{Lemma}
\newtheorem{cor}[thm]{Corollary}

\newtheorem{prop}[thm]{Proposition}

\newtheorem{defin}[thm]{Definition}

\newtheorem{exmp}[thm]{Example}
\newtheorem{rem}[thm]{Remark}
\AfterEndEnvironment{rem}{\noindent\ignorespaces}

\newcommand\cB{{\mathcal B}}

\newcommand\cE{{\mathcal E}}

\newcommand\cL{{\mathcal L}}
\newcommand\cO{{\mathcal O}}
\newcommand\cM{{\mathcal M}}

\newcommand\bbN{{\mathbb N}}

\newcommand\bbT{{\mathbb T}}

\newcommand\ve{\varepsilon}
\newcommand\eps{\epsilon}

\newcommand{\dd}{\text{d}}
\newcommand{\x}{\vec{x}}

\newcommand{\wh}[1]{\widehat{#1}}
\newcommand{\wt}[1]{\widetilde{#1}}

\input{abbrev.tex}

\usepackage{subfiles}

\title[Expansions in the CLT and the MLCLT]{
Expansions in the local and the central limit theorems
for dynamical systems}

\author{Kasun Fernando and Fran\c{c}oise P\`ene}
\address{}
\address{}

\date{\today}

\begin{document}

\begin{abstract}
We study higher order expansions both in the Berry-Ess\'een estimate (Edgeworth expansions) and in the local limit theorems for 
Birkhoff sums of chaotic probability preserving dynamical systems. We establish general results under technical assumptions, discuss the verification of these assumptions and illustrate our results by different examples (subshifts of finite type, Young towers, Sinai billiards, random matrix products), including situations of 
unbounded observables with integrability order arbitrarily close to the optimal moment condition required in the i.i.d.~setting. 
\end{abstract}



\maketitle
\tableofcontents

\section*{Introduction}\label{Intro}
Given a chaotic probability preserving dynamical system (PPDS), $(f,\mathcal M,\mu)$, and a centered observable $\phi:\mathcal M\rightarrow\mathbb R$, we are interested in the 
asymptotic behaviour of the sequence of centered random variables $(S_n:=\sum_{k=0}^{n-1}\phi\circ f^k)_{n\ge 1}$ as $n\to \infty$. More precisely, we are interested in establishing expansions in the central limit theorem (CLT) and in the mixing local central limit theorem (MLCLT) for $(S_n)_{n\ge 1}$ in the context of hyperbolic dynamical systems.

Let us recall that $(S_n)_{n\ge 1}$ is said to satisfy a nondegenerate CLT if $(S_n/\sqrt{n})_{n\ge 1}$ converges in distribution to a centered Gaussian random variable $Z$ of variance $\sigma^2>0$ with distribution function $\fN$, that is if
\[
\forall x\in\mathbb R,\quad \lim_{n\rightarrow +\infty}\mu\left(\frac{S_n}{\sqrt{n}}\le x\right)=\fN( x)\, .
\]
The MLCLT is a generalization of the local central limit theorem (LCLT) and has been used in \cite{PS2010} and \cite{DN} to prove limit theorems for flows. 
The MLCLT has the following form
\[
\EXP_\mu\left(\psi\, g(S_n)\, \xi\circ f^n\right)= \frac{\fN'(0)}{\sqrt{n}} I(g)\, \EXP_\mu(\psi)\EXP_\mu(\xi)+o(n^{-\frac 12})\, ,\quad\mbox{as} \ n\rightarrow +\infty\,  ,
\]
where $I(g):=\int_{\mathbb R }g(x)\, dx$ if $\phi$ is nonlattice and $I(g):=\sum_{k\in\mathbb Z}g(k)$ if $\phi$ is $\mathbb Z$-valued. When $ 
\xi\equiv 1$ and $\psi$ is the density of a probability measure $\mathbb P$ with respect to $\mu $, the above estimate corresponds to the LCLT with respect to $\mathbb P$.

Our goal is to investigate the rate of convergence in the two preceding results, via expansions of arbitrary order. 
We focus on expansions of the form 
\begin{equation}\label{EdgeExp}
\mathbb P \left(\frac{S_n}{\sqrt{n}}\leq x\right) =  \fN(x) \mathbb + \sum_{k=1}^{
r} \frac{\mathcal R_k(x)}{n^{k/2}}+o(n^{-r/2})\, ,\quad\mbox{as} \ n\rightarrow +\infty\, ,
\end{equation}
(corresponding to expansions in the CLT beyond Berry-Ess\'een estimates, such expansions are called Edgeworth expansions) and expansions of the form
\begin{equation}\label{expLLT}
\EXP_\mu\left(\psi\, g(S_n)\, \xi\circ f^n\right) =\sum_{k=0}^{\lfloor r/2\rfloor} \frac{\mathfrak a_j(g,\psi,\xi)}{n^{\frac 12+k}}+o(n^{-\frac {r+1}2})\, ,
\end{equation}
with $\mathfrak a_0(g,\psi,\xi)=\fN'(0)I(g)\, \mathbb E_\mu(\psi)\mathbb E_\mu(\xi)$ (corresponding to expansions in the MLCLT), under assumptions analogous to those of the classical case of sums of independent identically distributed random variables.

We recall that in the case when $(S_n)_{n\geq 1}$ is a sum of independent identically distributed random variables (the so-called i.i.d.~setting), \eqref{EdgeExp} and \eqref{expLLT} hold true as soon as these random variables admit a moment of order $r+2$ (together with another assumption for \eqref{EdgeExp} implying the fact that $S_1$ is far from being lattice). Here we obtain such results in a general dynamical context under assumptions close to the optimal condition in the i.i.d.~setting. In particular, we exhibit a family of examples of $\phi$ on expanding Young towers such that for every $\eta>0$, one can find $\phi\in L^{r+2}(\mu)\setminus L^{r+2+\eta}(\mu)$ for which \eqref{expLLT} holds true (see Theorem~\ref{EdgeExpforExpTowers} and the remark afterwards). The construction of such examples is based on operators acting on a chain of Banach spaces.

Estimates of the form \eqref{EdgeExp} have been established in \cite{CP} for one-sided subshifts of finite type. More recently, both \eqref{EdgeExp} and \eqref{expLLT} with $\psi\equiv \xi\equiv 1$ have been proved in \cite{FL} for general expanding dynamical systems, and independently, motivated by expansions in mixing for $\mathbb Z^d$-extensions of chaotic dynamical systems, analogues of \eqref{expLLT} have been shown for particular class of observables in \cite{PN,DNP}.
In all of these, expansions have been obtained for chaotic dynamical systems and bounded observables. Our goal here is to extend the results of \cite{FL} to hyperbolic systems modeled by Young towers with exponential tails and to the case of unbounded observables.

We point out the fact that, in the context of dynamical systems, the study of expansions for MLCLT and CLT is not just a curiosity from probability theory. There are important applications of these expansions to dynamical systems. For example, if $f$ is a map of a compact manifold $\cM$ preserving a measure $\mu$, $g_t$ is a flow on a compact manifold $Y$ preserving a measure $\mu_0$, and $\phi:X\to \mathbb{R}$ is a bounded zero mean observable, then the skew product $F(x,y)=(f(x),g_{\phi(x)}(y))$ preserving $\mu \times \mu_0$ exhibits decay of correlations provided that 
\begin{itemize}
\item the base map $f$ has decay of correlations 
\item $f$ admits a higher order expansion in the MLCLT for $\phi$ 
\end{itemize}
and a few mild assumptions on tail probabilities \cite[Section 6]{ACDN}. Even though we will not state the precise formulation here, this shows that there is a new mechanism to establish decay of correlations for dynamical systems via the expansions we study. 

Moreover, these expansions imply moderate deviation principles and local limit theorems for $S_n$. We refer to \cite[Section 5]{FL} for a detailed discussion of these applications. Edgeworth expansions are also used in statistics to improve the accuracy of bootstrap in sampling when the underlying process is Markov. See, for example, \cite{DM}. So proving the existence of these expansions may be considered as the first step of extending the bootstrap from the Markovian situation to deterministic dynamical systems. This is part of an on-going project with Nan Zou, and has also been independently considered in the recent preprint \cite{JWZ} where, in addition, a criterion to characterize the existence of the first order Edgeworth expansion is presented. 

Therefore, we not only introduce new classes  of weakly dependent random variables for which these expansions hold but also pave the way to establish interesting results about dynamical systems. Other interesting results can be obtained by considering the asymptotics for the large deviation principle as in \cite{FH}. However, to keep the exposition as concise as possible, we focus only on the CLT regime. 

This article is divided in two parts. In Part~\ref{part1}, we state expansions in a general context 
adapted (but not restricted) to a class of dynamical systems characterized by having an extension  which has an appropriate factor whose twisted transfer operators enjoy nice spectral properties. This is implemented thanks to the Nagaev-Guivarc'h perturbation method \cite{NG1,GH,HH} via the Keller-Liverani approach \cite{KL} combined with recent developments from \cite{D1, D2,D3, BV, AGY, IM, IM18}. In Part~\ref{part2}, we start by a detailed discussion of  the verification of our assumptions (in Section~\ref{verif}) and illustrate our general results by several examples: mixing subshifts of finite type (SFTs) with Lipschitz observables $\phi$ (in Section~\ref{SFT}), systems modeled by Young towers  including Sinai billiard with unbounded observables $\phi$ (in Section~\ref{Towers} completed with Appendix~\ref{appendYoung}),
and random matrix products (in Section~\ref{RWalks}).

\part{Edgeworth Expansions under general assumptions}\label{part1}

In this part of the paper, we state asymptotic expansions in a general context prove their existence of asymptotic expansions in that setting. \Cref{se:assum} is dedicated to the statement of the general assumptions about random variables, the resulting theorems, and our choice of the broad class of dynamical systems. In Section~\ref{CharFun}, we state
a key result about the asymptotic expansions of the characteristic functions of $S_n$. These expansions are of independent interest in probability theory (see, for example, \cite[Chapter 2]{BR1}). We end this part with \Cref{proofs} where we show how expansions of characteristic functions lead to expansions in the CLT and in the MLCLT for $S_n$, and hence, prove our general theorems. 

\section{General setting and results}\label{se:assum}

Let $(S_n)_{n\ge 1}$
be a sequence of $\mathbb X$-valued random variables with $\mathbb X=\mathbb R$ or $\mathbb Z$ defined on a probability space $(\cM,\mu)$.
We consider 
a double sequence of
real valued
random variables $(\psi_n,\xi_n)_{n\geq 1}$ on  $(\cM,\mu)$. We are interested in asymptotic expansions for $\mathbb E_\mu\left(\psi_n g(S_n)\xi_n\right)$ and $(\psi_n\mu)(S_n\le x\sqrt{n})$ (for the latter, assuming that $\psi_n$ is a probability density and that $\xi_n\equiv 1$).
Our proofs are based on Fourier transforms, and thus, will involve the quantity $\mathbb E_{\mu}
\left(\psi_n e^{isS_n} \xi_n\right)$. We set $\mathbb X^*=\mathbb R$ if $\mathbb X=\mathbb R$ 
and $\mathbb X^*=[-\pi,\pi]$ if $\mathbb X=\mathbb Z$.

We write $\mathcal{SP}$ (resp.~$\mathcal{LP}$) for the set of sequences of real numbers $(a_n)_{n\ge 1}$ (resp.~$(b_n)_{n\ge 1}$) converging to 0 super-polynomially fast (resp.~dominated by any positive power) such that, for all $p>0$, $a_n=o( n^{-p})$ (resp.~$b_n=o( n^{p})$).

\noindent 
{\bf Assumption $(\alpha)[r]$:}
Let $\delta>0$ and $n_0\ge 1$.
The function $s\mapsto \mathbb E_\mu\left(\psi_n e^{isS_n} \xi_n\right)$
 is $C^{r+2}$ on $[-\delta,\delta]$ and there exist $(b_n)_{n\ge 1}\in \mathcal{LP}$, $(a_n)_{n\ge 1}\in \mathcal{SP}$, a $C^{r+2}$-smooth complex valued function $s\mapsto \lambda(is)$ on $[-\delta,\delta]$ and constants $B_j, j=0,1,\dots, r+2$, such that for all $n\ge n_0$ and all $|s|<\delta$,
\[
\sup_{|s|<\delta}\left|H_n^{(j)}(0)
- B_j\right|=\cO(a_{n}
)\quad\mbox{and}\quad |\lambda(is)^n H_n^{(j)}(s)|\le b_ne^{-\frac {\sigma^2 s^2}8}+a_n\, ,
\]
where $H_n(s):=\lambda(is)^{-n}\mathbb E_{\mu}
\left(\psi_n e^{isS_n} \xi_n
\right)$ and with $\lambda(is)=1-\frac{\sigma
^2s^2}{2}+o(s^2)$, with $\sigma
^2>0$.

\noindent{\bf Assumption $(\beta)$:}
For any compact $K$ of $\mathbb X^*\setminus \{0\}$,  there exists $(a_n)_{n\ge 1}\in \mathcal{SP}$ such that
\[
\quad \sup_{s\in K}\left|\mathbb E_{\mu}
\left(\psi_n e^{isS_n} \xi_n
\right)\right|\le a_n
\, ,
\]

\noindent {\bf Assumption $(\gamma)$:}
Either $\mathbb X=\mathbb Z$, or there exists $K>0$ such that  there exist $(a_n)_{n\ge 1}\in \mathcal{SP}$,
there exist three positive constants $K_1,\alpha,\alpha_1,\hat\delta$ 
such that
\[
\forall |s|>K,\quad \left|
\mathbb E_{\mu}
\left(\psi_n e^{isS_n} \xi_n
\right)\right|\le K_1\left(
a_n
+|s|^{1+\alpha} e^{-n^{\alpha_1}\hat\delta|s|^{-\alpha}}\right) \, .
\]

\noindent {\bf Assumption $(\delta)[r]$:}
$\mathbb X=\mathbb R$ and for any $B>0$, there exists $K>0$ such that
\[
\int_{K<|s|<B n^{\frac {r-1}2}}\frac{
\left|
\mathbb E_{\mu}
\left(\psi_n e^{isS_n}
\right)\right|
}{|s|}\, ds=o(n^{-r/2})\, .
\]

Assumption $(\alpha)[r]$ is related to the existence of moments of $S_n$ up to and including the order $r+2$. In particular, in the i.i.d.~setting, if $S_n=\sum_{k=1}^{n}X_k$ with $(X_k)_{k\ge 1}$ a sequence of i.i.d.~random variables and if $\psi_n=\xi_n=1$, since $\mathbb E_\mu(e^{isS_n})=\lambda_{is}^n$ where $\lambda_{is}=\mathbb E_\mu(e^{isX_1})$ is the characteristic function of $X_1$, Assumption $(\alpha)[r]$ corresponds to the existence of the moment of order $r+2$ of $X_1$.


Assumption $(\beta)$ is an non-arithmeticity condition which translates in the i.i.d. to the fact that $X_1$ is not supported by a strict sublattice of $\mathbb X$.

While Assumptions $(\alpha)$ and $(\beta)$ deal with the behaviour of $\mathbb E_\mu(\psi_n e^{isS_n}\xi_n)$ for respectively small and intermediate values of $s$,
Assumptions $(\gamma)$ and $(\delta)$ deal with the behaviour of this quantity for large values, and should be compared (when $\mathbb X=\mathbb R$) with $0-$Diophantine property or equivalently, Cram\'er's continuity assumption  (see \cite[Chapter XVI]{Feller2})  in the i.i.d. setting as:
$$\limsup_{n \to \infty} |\EXP(e^{isX})| <1$$
which gives us that $|\EXP(e^{isS_n})|=|\EXP(e^{isX})|^n< \gamma^n$ for some $\gamma \in (0,1)$, and more generally, the $\alpha-$Diophantine property of supp $X$:
$$|\EXP(e^{isX})| < 1- \frac{\widehat C}{|s|^\alpha} \implies |\EXP(e^{isS_n})|<e^{-n\wh{C}|s|^{-\alpha}}$$
which guarantee the existence of Edgeworth expansions of all orders $r < \alpha^{-1}+1/2$ provided $X$ has $r+2$ moments (see \cite{DF}). 

Now, let us introduce 
the space $\fF_k^m$ of functions for which 
we prove expansions in the MLCLT. Set
\begin{equation}\label{hat g}
\widehat g(s):=\int_{\mathbb X}e^{-isx}g(x)\, d\lambda(x)\, ,\quad s\in \mathbb X^*\, ,
\end{equation}
where $\lambda$ is the Lebesgue measure 
 if $\mathbb X=\mathbb R$
 and where $\lambda$
is the counting measure 
 if $\mathbb X=\mathbb Z$.
We also set
$$C^{m}(g):=\sup_{s\in\mathbb X^*} \frac{|\widehat g(s)|}{\min(1,|s|^{-m})}\quad\mbox{and}\quad C_k(g):=\Vert \widehat g^{(k)}\Vert_\infty\, .$$
Observe that $C_k(g)\le \max_{0\leq j \leq k} 
\int_{\mathbb X}|x|^j|g(x)|\, d\lambda(x)$ if this last quantity is finite. When $\mathbb X=\mathbb R$ and $g$ is  $m$ times continuously differentiable, $C^m(g)\le \max_{0\leq j \leq m} \|g^{(j)}\|_{\text{L}^1 (\mathbb R)}$. When $\mathbb X=\mathbb Z$, $C^m(g)\le \pi^m\sum_{n\in\mathbb Z}|g(n)|$. Define $$C^m_k(g):=C^m(g)+C_k(g).$$ If $\mathbb X=\mathbb R$, we say $g \in \fF_k^m$ if $g:\mathbb R\rightarrow\mathbb R$ is continuous, $\lambda$-integrable
and if $\hat g:\mathbb X^*\rightarrow\mathbb C$ is $k$ times continuously differentiable with $C^m_k(g) < \infty$. In particular, if $\mathbb X=\mathbb R$, 
compactly supported smooth functions
are in $\fF_k^m$ for all $k,m$. If $\mathbb X=\mathbb Z$, $\fF_k^m=\fF_k^0$ is the set of functions $g:\mathbb Z\rightarrow \mathbb C$ satisfying the following summability condition
\[
\sum_{n\in\mathbb Z}|n|^k|g(n)|<\infty\, .
\]

Under our assumptions, we set
$\mathfrak N$ for the distribution function of a centered Gaussian random variable with variance $\sigma
^2$ and  $\mathfrak n$ for the corresponding probability density function (that is $\mathfrak n$ is the derivative of $\mathfrak N$).

\subsection{Main abstract results}

Here we state the three main abstract results of this paper.



\begin{thm}[Global expansion of order $r$ in the MLCLT]\label{WeakGlobalExp}
Let $(S_n)_{n\ge 1}$,$(\psi_n)_{n\ge 1}$ and $(\xi_n)_{n\ge 1}$ be three sequences of real valued random variables defined on the same probability space $(\cM,\mu)$,
with $S_n$ taking values in $\mathbb X$. 
Let $r$ be a nonnegative integer. 

Suppose the Assumptions 
$(\alpha)[r]$, $(\beta)$ and $(\gamma)$ hold. Then there exist polynomials $R_j$ such that
$$\EXP_\mu\left(\psi_n\, g(S_n)\, \xi_n\right) =\sum_{j=0}^{r}
\frac{1}{ n^{(j-1)/2}}\int_{\mathbb X} (R_j\cdot\fn)(x/\sqrt{n}) g(x) \, d\lambda(x) + C^{q+2}(g) \cdot o(n^{-r/2})\, ,$$ for all $g \in \fF^{q+2}_{0}$ where $q>\alpha\big(1+\frac{r}{2 \alpha_1}\big)$.
\end{thm}
\begin{thm}[Local expansion of order $r+1$ in the MLCLT]\label{WeakLocalExp}
Let $(S_n)_{n\ge 1}$,$(\psi_n)_{n\ge 1}$ and $(\xi_n)_{n\ge 1}$ be three sequences of real valued random variables defined on the same probability space $(\cM,\mu)$ with $S_n$ taking values in $\mathbb X$
Let $r$ be a nonnegative integer. 

Suppose the Assumptions 
$(\alpha)[r]$, $(\beta)$ and $(\gamma)$ hold. Then there exist polynomials $Q_j$ such that
$$\sqrt{n}\EXP_\mu\left(\psi_n\, g(S_n)\, \xi_n\right) =\sum_{j=
0}^{\lfloor r/2 \rfloor} \frac{1}{n^{j}}\int_{\mathbb X}g(x )Q_j(x) \, d\lambda(x)+ C^{q+2}_{r+1}(g) \cdot o(n^{-r/2})\, ,$$
for all $g \in \fF^{q+2}_{r+1}$ where $q>\alpha\big(1+\frac{r+1}{2 \alpha_1}\big) $.
\end{thm}
\vspace{5pt}

\begin{rem} When $\bbX=\integers$, the two theorems above are true for $g\in \fF^{0}_0$ and $g\in \fF^{0}_{r+1}$, respectively. Later when we discuss these results in the setting of a specific example $($see \Cref{StongExpOrd1SFT}, \Cref{EdgeExpforExpTowers}, \Cref{EdgeExpforTowers} and \Cref{Ord1RMP}$)$, we only mention the condition on $q$ corresponding to $\bbX=\reals$. For $\bbX=\integers$, it is understood that $q+2=0$. 
\end{rem}

\begin{rem}
Note that the second result is {\it local} because we consider rapidly decaying $g$ and hence, the contribution away from the origin is negligible whereas in the first for large $n$ even values of $g$ further away from the origin contributes significantly (and hence, {\it global}). Moreover, in both the cases we have precise control over the error in terms of $g$. 
\end{rem}

\begin{rem}
Observe that assumptions of the our two first above results are closed to the optimal
moment assumptions in the i.i.d.~setting. Indeed, the $C^{r+2}$ smoothness coming from Assumption $(\alpha)[r]$ is the spectral equivalent of the existence of a moment of order $r+2$ in the i.i.d.~setting.
\end{rem}

The third and the last main theorem is on Edgeworth expansions which provide a uniform control over the error term in the CLT for $S_n$ when it is non-lattice.

\begin{thm}\label{StongExp}
Let $(S_n)_{n\ge 1}$,$(\psi_n)_{n\ge 1}$ and $(\xi_n\equiv 1)_{n\ge 1}$ be three sequences of real valued random variables defined on a same probability space $(\cM,\mu)$. 
Let $r'$ be a positive integer and $r\ge 1$ be a real number. Let $\mathbb P_n$ be the probability measure on $\mathcal M$ admitting the density $\psi_n$ with respect to $\mu$.

Suppose the Assumptions 
$(\alpha)[r']$,
$(\beta)$ and $(\delta)[r]$ hold. Then there exist polynomials $P_k$ such that
$$\Prob_n\left(\frac{S_n}{\sqrt{n}}\leq x\right) = \fN(x) + \fn(x)\sum_{k=1}^{
\min(r',\lfloor r \rfloor)} \frac{P_k(x)}{n^{k/2}}+o(n^{-\min(r,r')/2})\, ,$$
uniformly in $x$.
\end{thm}
For completeness, let us indicate that the $R_j$'s, $Q_j$'s and $P_k$'s appearing in Theorems~\ref{WeakGlobalExp}, \ref{WeakLocalExp} and~\ref{StongExp} are given respectively by~formulas \eqref{StronExpPolyDer}, \eqref{Qm} and \eqref{EdgePolyRel}.

Finally, we state two Corollaries about the first order Edgeworth expansions. We note that for the first order expansion with an error of $o(n^{-1/2})$, only the assumptions $(\alpha)[1]$ and $(\beta)$ are required. 

\begin{cor}[Order 1]\label{StongExpOrd1}
Let $(S_n)_{n\ge 1}$,$(\psi_n)_{n\ge 1}$ and $(\xi_n=1)_{n\ge 1}$ be three sequences of real valued random variables defined on a same probability space $(\cM,\mu)$. 
Let $\mathbb P_n$ be the probability measure on $\mathcal M$ that is absolutely continuous with respect to $\mu$ with probability density function $\psi_n$.

Suppose the Assumptions 
$(\alpha)[1]$ 
and $(\beta)$ hold with $\bbX = \reals$. Then 
$$\mathbb P_n \left(\frac{S_n}{\sqrt{n}}\leq x\right) = \fN(x) +  \frac{P_1(x)}{n^{1/2}}\fn(x)+o(n^{-1/2})\, ,$$
uniformly in $x$. 
\end{cor}

If we have slightly better control, that is $(\alpha)[2]$, $(\beta)$ and $(\delta)[r]$ with $r\in (1,2)$,  then the error in the expansion improves to $o(n^{-r/2})$ but could not be better in general because the second term in the expansion is $\cO(n^{-1})$. 

\begin{cor}\label{BetterExpOrd1}
Let $(S_n)_{n\ge 1}$,$(\psi_n)_{n\ge 1}$ and $(\xi_n=1)_{n\ge 1}$ be three sequences of real valued random variables defined on a same probability space $(\cM,\mu)$ 
Let $\mathbb P_n$ be the probability measure on $\mathcal M$ that is absolutely continuous with respect to $\mu$ with probability density function $\psi_n$.

Suppose the Assumptions 
$(\alpha)[2]$,
$(\beta)$ and $(\delta)[r]$ hold for some real number $r \in( 1,2)$, 
Then
$$ \Prob_n \left(\frac{S_n}{\sqrt{n}}\leq x\right) =\fN(x)+\frac{P_1(x)}{\sqrt{n}} \fn(x)+
o\left(n^{- r/2 }\right),$$
uniformly in $x$. 
\end{cor}

Finally, we recall from \cite[Appendix A]{FL} that there is a hierarchy of expansions. Suppose $r$ and $q$ are positive integers. Then the Figure~\ref{hierarchy} shows implications among different expansions. 

\begin{figure}[h!]
\centering
\begin{tikzcd}[/tikz/column 1/.append style={anchor=base east}]
\text{Edgeworth expansions:} & \text{order}\,\, r \arrow[d, Leftrightarrow] & \\
\text{global expansion in LCLT:} & \text{order}\,\, r\,\, \text{for}\,\, g \in \fF^{1}_0 \arrow[d, Rightarrow]  \rar[shorten <= 0.5em, shorten >= 0.5em] & \text{order}\,\, r\,\, \text{for}\,\, g \in \fF^{q}_{r} \arrow[d, Rightarrow] \\
\text{local expansion in LCLT:} & 
\text{order}\,\, r\,\, \text{for}\,\, g \in \fF^{1}_{r} \ar[r, rightarrow, shorten <= 0.5em, shorten >= 0.5em]  &
 \text{order}\,\, r\,\, \text{for}\,\, g \in \fF^{q}_{r} 
\end{tikzcd}
\caption{Hierarchy of expansions.}
\label{hierarchy}
\end{figure}
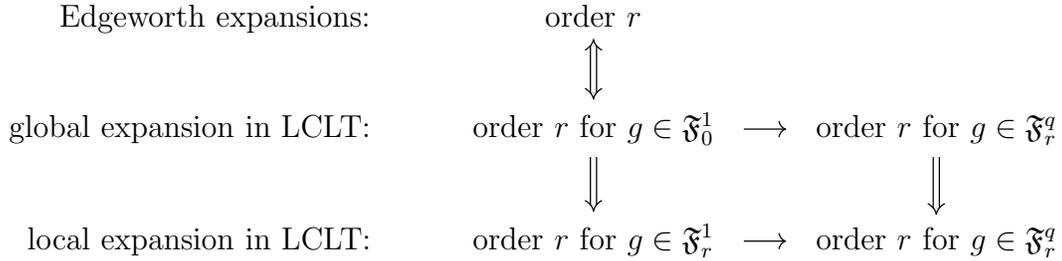

Here $\rightarrow$ indicates that the implication is obvious due to the inclusion of spaces $\fF^q_{r} \subseteq \fF^{q'}_{r'}$ if $r \geq r'$ and $q \geq q'$. One cannot expect expansions in the MLCLT to come from Edgeworth expansions because the former keeps track of both $S_n$ and $S_{n}-S_{n-1}$ whereas the latter keeps track of $S_n$ only. However, expansions in the MLCLT imply expansion in the LCLT in the obvious way. 
Even in the LCLT case, there are elementary examples where expansions in the LCLT of all orders exist but Edgeworth expansions of higher order fail to exist. We refer the reader to \cite{FL} for more details about this. 

As expected stronger control over the decay of $\EXP_\mu(\psi_n e^{isS_n})$
than the one provided by Assumption $(\delta)$ leads to stronger results. In fact, if we know that $\EXP_\mu(\psi_n e^{isS_n}) \leq Cs^{-\beta}$ for $|s|>n^{\ell}$ with $\beta > \frac{r+1}{2\ell}$ -- for example, when $S_n$ is close to a Gaussian and $\psi_n$ bounded -- then \Cref{WeakGlobalExp} and \Cref{WeakLocalExp} hold for $g \in \fF^1_{0}$ and $g \in \fF^1_{r+1}$, respectively. We refer the reader to \cite{FL} for a proof of this fact in the case $\psi_n \equiv \xi_n \equiv 1$.

\subsection{Technical assumptions in a dynamical context}\label{sec:hypo}

In this section, we state assumptions tailored for hyperbolic/dispersive dynamical systems and implying our previous assumptions $(\alpha)$, $(\beta)$, $(\gamma)$ and $(\delta)$
(see Propositions~\ref{lemmaABCgreek} and \ref{lemmaABCgreek1}). Recall that our goal is to study the case of 
Birkhoff sums $S_n=\sum_{k=0}^{n-1}\phi\circ f^k$ for some $\phi:\cM\rightarrow\mathbb R$ and for a PPDS $(f,\cM,\mu)$. So here we take $\psi_n=\psi$ and $\xi_n=\xi\circ f^n$.

In order to prove our general Assumptions $(\alpha)$, $(\beta)$, $(\gamma)$ and $(\delta)$, we use a natural and efficient strategy based on transfer operators \cite{NG1,GH}.
The key idea is to approximate $$\mathbb E_\mu(\psi e^{isS_n}\xi\circ f^n)$$ by
$$\mathbb E_{\bar\nu}(\bar\psi_{n,s} e^{is\bar S_{m}}\bar\xi_{n,s}\circ \bar F^{m})$$
where $m=m(n)\sim n$ and $\bar S_m=\sum_{k=0}^{m-1}\bar\phi\circ \bar F^k$
is a Birkhoff sum for a PPDS $(\bar F,\bar\Delta,\bar\nu)$ (which may be different from the initial on $(f,\cM,\mu)$) of which the transfer operator $\cL$ enjoys nice spectral properties. Recall that $\cL$ satisfies $\mathbb E_{\bar\nu}(g.h\circ \bar F)=\mathbb E_{\bar\nu}(h\cL(g))$. This implies the following key formula
\[
\mathbb E_{\bar\nu}(\bar\psi_{n,s} e^{is\bar S_{m}}\bar\xi_{n,s}\circ \bar F^{m})
=\mathbb E_{\bar\nu}(\bar\xi_{n,s} \cL_{is}^m(\bar\psi_{n,s}n))\, ,
\]
with $\cL_{is}(h)=\cL(e^{is\bar\phi}h)$.
Thus our strategy to prove $(\alpha)$, $(\beta)$, $(\gamma)$ and $(\delta)$
will be to prove Assumptions involving $\cL_{is}$.

Note that this classical approach has its equivalent in the case of additive functional of Markov processes (see \cite{HH,HP} and our application to random matrix products in Section~\ref{RWalks}). Indeed, if $S_n=\sum_{k=1}^{n}h_0(X_k)$ where $X_k$ is a Markov chain on $(\cM,\mu)$, then
\[
\mathbb E_{\mu}(h_1(X_0) e^{is S_{n}}h_2(X_n))
=\mathbb E_{\bar\nu}(h_1 \cL_{is}^n(h_2))\, ,\,\,
\mbox{with}\ \cL_{is}(h)=\mathbb E_\mu (e^{ish_0(X_1)}h(X_1)|X_0)\, .
\]

This approach has already been used for expansions in the CLT and in the LLT in \cite{FL} in the case when 
$(f,\cM,\mu)=(\bar F,\bar\Delta,\bar\nu)$ is an expanding PPDS. We generalize it here in two directions: first, our assumptions below are tailored to study hyperbolic systems, and second, we weaken the regularity assumptions on $\cL_{is}$
in order to treat also the case of functions $\phi$ not admitting moments of all orders.

In our series of assumptions below (and more precisely in assumptions $(A)$ and $(C)$), if $(f,\cM,\mu)=(\bar F,\bar\Delta,\bar\nu)$, we assume that $k=\vartheta^k=0$. Otherwise the assumptions have to hold for any $k\in\mathbb N$.

The first assumption describes the abstract model we work on. Those who are familiar with towers in \cite{Y}, while reading, may keep in mind the tower construction $(F,\Delta)$ associated to
a mostly hyperbolic map $(f,\cM)$ and the subsequent quotienting along stable directions to obtain an expanding tower $(\bar{F},\bar{\Delta})$.
\vspace{10pt}

\noindent 
{\bf Assumption $(0)$:}

Let $(f,\mathcal M,\mu)$, $(F, \Delta, \nu)$ and $(\bar{F}, \bar{\Delta}, \bar{\nu})$ be three PPDS such that $(F, \Delta, \nu)$ is an extension of the two others by $\mathfrak p:\Delta\rightarrow\mathcal M$ and $\bar{\mathfrak p} : \Delta \to \bar\Delta$, respectively (see the Figure~\ref{setup}). 

\begin{figure}[h!]
\centering
\begin{tikzcd}
             & (F, \Delta, \nu) \arrow[ld, "\bar{\mathfrak p}" ']  \arrow[dr, "{\mathfrak p}"] & \\
(\bar{F}, \bar{\Delta}, \bar{\nu}) &                         & (f, \mathcal M, \mu)
\end{tikzcd}
\caption{Associated dynamical systems.}
\label{setup}
\end{figure}
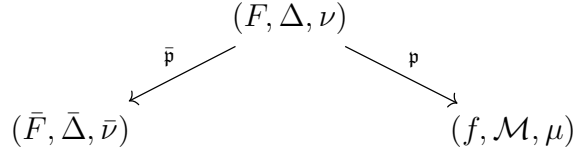

Let $\phi:\cM\rightarrow\mathbb X$ with $\mathbb X=\mathbb R$ or $\mathbb X=\mathbb Z$ be a centered observable, i.e., $\mathbb E_\mu(\phi)=0$. We further assume that $\phi$ is not a coboundary  in $L^2(\cM,\mu)$, i.e. $\phi\ne h-h\circ f$ for all $h\in L^2(\mu)$. Define $S_n:=\sum_{k=0}^{n-1}\phi\circ f^k$ and let $\psi,\xi:\mathcal M\rightarrow\mathbb R$ be two observables. Take $\cL$ to be the transfer operator of $\bar{F}$ with respect to $\bar{\nu}$.  For any complex Banach space $\cB \hookrightarrow L^{1}(\bar{\nu})$, we define $\|\cdot\|_{\cB'}$ by
$\|g\|_{\cB'}:=\sup_{\|h\|_{\cB} \leq 1}\big|\EXP_{\bar \nu}(g  h)\big|$. Here $\hookrightarrow$ denotes continuous embedding of spaces, i.e., $\cB \subset L^1(\bar\nu)$ and there exists $\mathfrak c >0$ such that $\|\cdot\|_{L^1(\bar\nu)} \leq \mathfrak c \|\cdot\|_{\cB}$.
\vspace{10pt}

The next assumption states conditions to ensure that $(\bar{F},\bar{\Delta})$ and $\cL$ retain sufficient information about $S_n$ upto a controlled error. Combined with the favourable properties of $(\bar{F},\bar{\Delta})$ and $\cL$, this assumption would lead to the expansions we seek. The introduction of a ``double chain of spaces" is crucial for our study of unbounded functions. For the study of bounded functions, we can work with a single Banach space $\cB$ and assume that $\cX_a=\cX_a^{(+)}=\cB$.
\vspace{10pt}


\noindent 
{\bf Assumption $(A)[r]$:} 

There exist $\delta>0$, $p_0\ge 1$ and a double chain $(\mathcal X_a,\mathcal X^{(+)}_{a})_{a=0,...,r+2}$ of complex Banach spaces containing ${\mathbf 1}_{\bar\Delta}$ and satisfying
\[
\forall a=0,...,r+1,\quad
\mathcal X_0\hookrightarrow\cX_a\hookrightarrow\cX^{(+)}_a\hookrightarrow \cX_{a+1}\hookrightarrow\mathcal X_{r+2}\hookrightarrow \cX_{r+2}^{(+)}
\hookrightarrow L^{p_0}(\bar\nu)
\]
and three non negative real numbers $r_0,q(\psi),q(\xi)$ with $r_0\ge r$,  $q(\xi)+q(\psi)\le r_0-r$, $\xi\in L^{\frac{r_0+2}{q(\xi)}}(\bar\nu)$ and $\psi\in L^{\frac{r_0+2}{q(\psi)}}(\bar\nu)$ such that the following holds true.

\begin{enumerate}\setlength\itemsep{12pt}
\item 
There exist  $\mathbb X$-valued functions $\chi\in  L^{r_0+2}(\Delta,\nu)$ and $\bar\phi \in L^{r_0+2}(\bar\Delta,\bar\nu)$ such that $\phi \circ \mathfrak{p} = \bar{\phi} \circ \bar{\mathfrak{p}} + \chi - \chi \circ F$.
\item Set $$\bar S_n:=\sum_{l=0}^{n-1}\bar\phi\circ\bar F^l,\,\,\, \cL_{is}(\cdot):=\cL(e^{is\bar\phi} \cdot )\,\,\,\text{for}\,\,\, s \in \reals,$$ and $$h_{k,s,H}:=(H\circ{\mathfrak p}\, e^{is\chi})\circ F^k e^{-is\bar S_k\circ{\bar{\mathfrak{p}}}}\,\,\,\text{for}\,\,\,H\in\{\psi,\xi\}.$$
Denote the $j$th derivative of a function with respect to $s$ by the superscript $(j)$.\\

\noindent
 There exist $\vartheta \in (0,1)$ and
$\bar h_{k,s,H}
:\bar\Delta\rightarrow\mathbb C$ where $H\in\{\psi,\xi\}$
 that are $C^{r+2}$ in $s$ such that 
for all $j=0, 1,...,r+2$ and for  $H\in\{\psi,\xi\}$,

\begin{align}
\Vert \bar h
_{k,s,H}\Vert_{L^{\frac {r_0+2}{j+q(H)}}(\bar\nu)}&\le
C_0
\, ,\label{hk0}
 \\
\Vert h^{(j)}_{k,s,H} -\bar h^{(j)}_{k,s,H}\circ\bar{\mathfrak p}\Vert_{L^{\frac{r_0+2}{j+q(H)}}(\nu)} &\le C_0 \vartheta^k (1+|s|)(1+k)^j,\, \label{hk1}\\
\left\Vert (\mathcal L_{is}^{2k}{\bar h_{k,s,\psi}})^{(j)}\right\Vert_{\cX_j}+\Vert \bar h^{(j)}_{k,s,\xi}\Vert_{(\cX_{r+2-j}^{(+)})'} &\le C_0(1+|s|)(1+ k)^j\, ,\label{hk2}
\end{align}

\item For any $a=0,...,r+2$:
\begin{itemize}
\item The operators  $\cL_{is}(\cdot)=\cL(e^{is\bar\phi} \cdot ),\ s \in \reals$ are bounded operators on $\cX_a$ and $\cX_a^{(+)}$,
\item The map $s \mapsto \cL_{is}\in\mathcal L(\cX_a,\cX_{a}^{(+)})$ is continuous,
\item For any integer $j=1,...,r+2-a$, the map $s \mapsto \cL_{is}\in\mathcal L(\cX_a^{(+)},\cX_{a+j})$ is $C^{j}$ on $(-\delta,\delta)$ 
with $j$-th derivative 
$(\cL^n_{is})^{(j)}:=\cL^n_{is}((i\bar S_n)^j\,\cdot\, )\in\mathcal L(\cX^{(+)}_a,\cX_{a+j}).$
\end{itemize}
\item Either all the sets $\cX_a,\cX_a^{(+)}$ are equal or 
there exist $\widetilde C>0$ and $\widetilde\kappa_1\in (0,1)$ such that, for every $\cX=\cX_a$ or $\cX=\cX_a^{(+)}$
\[
\forall h\in\cX,\quad \sup_{|s|<\delta}\Vert \cL_{is}^nh\Vert_{\cX}\le
\widetilde C\left(\widetilde\kappa_1^n\Vert h\Vert_{\cX}+\Vert h\Vert_{L^{p_0}(\bar\nu)}\right)\, .
\] 
\end{enumerate}
\vspace{10pt}

In addition to allowing $\phi$ unbounded, the Assumption $(A)[r]$ with $r_0> r$ allows us to consider unbounded test functions $\psi$ and $\xi$. Also, to make a link with the notations $\bar H_{n,s}=\bar\xi_{n,s},\bar\psi_{n,s}$ introduced at the beginning of Section~\ref{sec:hypo}, let us indicate that $\bar H_{n,s}=\bar h_{0,s,H}$ if $(f,\cM,\mu)=(\bar F,\bar\Delta,\bar\nu)$ and $\bar H_{n,s}=\bar h_{\lfloor (\log n)^2\rfloor,s,H}$ otherwise.

In the case of SFTs, Assumption $(A)(1)$ above is reminiscent of the well-known theorem due to Sinai that any H\"older function on the two-sided shift space is a function on the one-sided shift space upto a coboundary and upto some loss of regularity. It is, in fact, what allows us to compare $S_n$ with $\bar{S}_n$ and hence, make use of $\cL_{is}$. Assumption $(A)(2)$ states that the error made in this comparison is under control. When $(f,\cM,\mu)=(\bar{F},\bar\Delta,\bar\nu)$,
Assumptions $(A)(1)$ and $(A)(2)$ (except \eqref{hk2} for $j=0$) are vacuous because $\mathfrak p = \bar {\mathfrak p} = \text{Id}$, $\chi\equiv 0$, $\vartheta^k=0$ and we take $k=0$. 


Assumption $(A)(3)$ when the $\cX_a$ and $\cX_a^{(+)}$ are all equal is the standard assumption to implement the classical Nagaev-Guivarc'h perturbation method of bounded linear operators as in \cite{HH, Kato}. The uniform Doeblin-Fortet estimate contained in Assumption $(A)(4)$ will allow us to apply this perturbation method via the Keller Liverani theorem when the $\cX_a$ and the $\cX_a^{(+)}$ are not all equal. This approach has been used in \cite{HP}, in a Markovian context, to establish various limit theorems under moment assumptions very close to the optimal assumptions in the i.i.d.~setting.


Moreover, some favourable spectral properties of twisted transfer operators $\cL_{is}$ are assumed in order to use the Nagaev-Guivarc'h approach, \cite{NG1, GH}. However, we would require more control over the spectra because we seek higher order terms in the central limit theorem. This is our next assumption. 
\vspace{10pt}

\noindent 
{\bf Assumption $(B)$:}
\begin{enumerate}\setlength\itemsep{12pt}
\item The operator $\cL$ acting on  each $\cX_a$ and $\cX_a^{(+)}$ has an isolated and simple eigenvalue $1$, the rest of its spectrum is contained inside the disk of radius smaller than $1$ (spectral gap).
\item  For all $s\in \mathbb X^*\setminus\{0\}$, the spectrum of the operator $\cL_{is}$ acting on either $\cX_0$ or acting on $\cX_0^{(+)}$ is contained in $\{z\in \complex\ |\ |z|<1\}.$
\item $\sum_{n\ge 0}\Vert \cL^n\bar\phi\Vert_{L^2(\bar\nu)}<\infty$.
\end{enumerate} 
\vspace{10pt}

Observe that Assumption $(B)(3)$ is automatic as soon as $\cX_1\hookrightarrow L^2(\bar\nu)$ and Assumptions $(B)(1)$ and $(A)(3)$ are satisfied. Note that $(A)(3)$ implies that $\cL(\bar\phi)\in\cX_1$. This will be the case in most of our examples. Moreover, Assumptions $(A)[r]$ and $(B)$ will imply Assumptions $(\alpha)[r]$ and $(\beta)$ for $\psi_n=\psi$ and $\xi_n=\xi\circ f^n$.

Next two assumptions state how much control over $\cL_{is}$ is required for large values of $s$ in order to guarantee the existence of expansions.
\vspace{10pt}

\noindent 
{\bf Assumption $(C)$:} 

Either $\mathbb X=\mathbb Z$ or there exists $K>0$ such that:\vspace{-2pt}
\begin{enumerate}\setlength\itemsep{12pt}
\item There exist two complex Banach spaces $\cB_1\hookrightarrow\cB_2\hookrightarrow L^1(\bar\nu)$ both containing $\mathbf 1_{\bar\Delta}$, and real numbers
 $\alpha\ge 0$, $\alpha_1 \in (0, 1]$ and $\hat{\delta}>0$ 
and $n_1$ such that for every $\delta>0$, every
$|s|>K$ and every $n\ge n_1$,
$$\left\Vert \mathcal L_{is}^n\right\Vert_{\mathcal B_1 \rightarrow \mathcal B_2}\le C |s|^\alpha e^{-n^{\alpha_1}\hat\delta |s|^{-\alpha}}\, .$$
\item There exists $C'_0>0$ such that, for every $|s|>K$,
\begin{equation}
\label{hk2var}
\left\Vert \mathcal L_{is}^{2k}{\bar h_{k,s,\psi}}\right\Vert_{\cB_1} \le C'_0|s|\quad\mbox{and}\quad
\Vert {\bar h_{k,s,\xi}}\Vert_{\cB_2'}\le C'_0\, .
\end{equation}
\end{enumerate}
\vspace{15pt}


\noindent 
{\bf Assumption $(D)[r]$:} 

$\mathbb X=\mathbb R$ and there exist two complex Banach spaces $\cB_1\hookrightarrow\cB_2\hookrightarrow L^1(\bar\nu)$ and $d_1, d_2 \geq 0$ with $d_1+d_2=1$ such that for all $B>0$, there exists $K>0$ such that
\begin{enumerate}
\item $$
\int_{K<|s|<Bn^{(r-1)/2}}\frac{ \left\Vert \mathcal L_{is}^n\right\Vert_{\mathcal B_1 \rightarrow \cB_2} }{|s|^{d_1}}\, ds=  o(n^{-\frac r2})\,, 
$$
\item $$\sup_{K<|s|<Bn^{\frac{r-1}2},\ k\ge 1} |s|^{-d_2} \Vert \bar h_{k,-s,1}\Vert_{\cB_2^{\prime}}\, \left\Vert \mathcal L_{is}^{2k}(\bar h_{k,s,\psi})\right\Vert_{\mathcal B_1}<\infty.$$
\end{enumerate}
\vspace{10pt}

Assumption $(C)$ combined with Assumption $(A)[r]$ imply Assumption $(\gamma)$. In addition, Assumption $(D)[r]$ along with Assumption $(A)[r]$ imply $(\delta)[r]$ with $\psi_n=\psi$. Even though the Assumptions $(C)$ and Assumption $(D)$ seem technical, Assumption $(C)(1)$ and sufficient conditions for Assumption $(D)(1)$ appear naturally in the study of decay of correlation for hyperbolic flows and are implied from Diophantine conditions on periodic orbits of these flows (see \cite{D1, D2, D3, IM}). These ideas are discussed in greater detail in \Cref{C_D}. 

When $d_1=0$, $(C)(2)$ is sufficient for $(D)(2)$, and $(C)(2)$ readily follows from the way $\bar{h}_{k,s,H}$ is defined in our examples. In the case of $(f,M,\mu)=(\bar{F},\bar{\Delta},\bar{\nu})$, Assumption $(D)(2)$ is vacuous. In this case, $(D)(1)$ with $d_1=1$ and $\cB_1=\cB_2(=\cB)$ is already used in \cite{FL} to establish expansions for bounded observables. In particular, if there are some constants $\epsilon>0$ and $\ell>0$ such that 
\begin{equation}\label{StrngD2}
\Vert \psi\Vert_{\cB'}<\infty\quad\mbox{and}\quad
\sup_{|s|\in(K,\,Bn^{(r-1+\epsilon)/2})}\left\Vert \mathcal L_{is}^n\right\Vert_{\cB} = \cO( n^{-\ell})\,
\end{equation}
$(D)(1)$ holds for all $r\geq 1$ with $d_1=1$. 

Now we prove how our assumptions $(A)-(D)$ imply our previous assumptions $(\alpha)-(\delta)$. In what follows, $\cL(\cB_1,\cB_2)$ denotes the space of bounded linear operators from a Banach space $\cB_1$ to a Banach space $\cB_2$. When $\cB_1=\cB_2$, we write $\cL(\cB_1,\cB_1)$ as $\cL(\cB_1)$. We will also continue to use $\cL$ to denote the transfer operator of $(\bar F, \bar \Delta)$. The implied meaning of $\cL$ will be clear from the context.


\begin{prop}[follows from \cite{HP}]\label{propHP}
Suppose Assumptions $(0)$,
$(A)[r](1,3, 4)$ and $(B)(1,3)$ hold. Then
there exist $\kappa_1\in (0,1)$ and a family $(\lambda(is),\Pi_{is},\Lambda_{is})_{s\in(-\delta,\delta)}$ which is $C^{m}$-smooth as functions from $(-\delta,\delta)$ to $\mathbb C\times \mathcal L(\cX_j,\cX_{j+m}^{(+)})\times \mathcal L(\cX_j,\cX_{j+m}^{(+)})$
for any $0\le j\le j+m\le r+2$ such that
\begin{equation}\label{DecompOp3'}
\mathcal L_{is}^n=\lambda(is)^n\Pi_{is}+\Lambda_{is}^n\,\,\,\text{in}\,\,\,\bigcap_{a=0}^{r+2}\left(\mathcal L(\cX_a)\cap\cL(\cX_a^{(+)})\right)\, ,
\end{equation}
where 
\begin{equation}\label{DecompOp3'Compl}
\max_{a=0,...,r+2}\max_{j=0,...,r+2-a}\sup_{s\in[-\delta,\delta]} \Vert [\Lambda(is)^n]^{(j)}\Vert_{\cX_a^{(+)}\rightarrow\cX_{a+j}}=\mathcal O(\kappa_1^n),\,\,\, \Pi_0=\mathbb E_{\bar\nu}[\,\cdot\,] \mathbf 1_{\bar{\Delta}}
\end{equation}
and 
\begin{equation}\label{devlambdamodif}
\lambda(is)=1-\frac{\sigma^2_\phi}2s^2+o(s^2),\,\,\, \text{with}\,\,\,\sigma^2_\phi=\lim_{n\rightarrow +\infty}\mathbb E_\mu(S_n^2/n)>0.
\end{equation}
\end{prop}

\begin{rem}
This allows us to write the $j^{\text{th}}$ derivative of
$\bar h_{k,s,\xi}\mathcal L_{is}^{n}(\bar h_{k,s,\psi})\in L^1(\bar\nu)$ in $s$ as follows 
\[
\sum_{j_1+j_2+j_3=j}  \bar h^{(j_1)}_{k,s,\xi}(\mathcal L_{is}^{n-2k})^{(j_2)} (\cL_{is}^{2k}(\bar h_{k,s,\psi}))^{(j_3)} \in \cX_j^{(+)}\subset\cX_{r+2}^{(+)},\,\,\, j=0,...,r+2,
\]
even if the spaces to which  $\bar h^{(j_1)}_{k,s,\xi},(\mathcal L_{is}^{n-2k})^{(j_2)}, (\cL_{is}^{2k}(\bar h_{k,s,\psi}))^{(j_3)}$ belong vary with $j_1,j_2,j_3$.
\end{rem}

\begin{proof}
Assume first that for all $a$ the spaces $\cX_a$ and $\cX_a^{(+)}$ are equal to $\cB$. Then,
due to the perturbation theory of bounded linear operators (see \cite{HH}, \cite[Chapter 7]{Kato}), for all $|s| \leq \delta$ (if required, after shrinking $\delta$), $\cL_{is}$ can be expressed in $\cL(\cB)$ as  
\begin{equation}\label{OpDecom}
 \cL_{is} =\lambda(is)\Pi_{is}+\Lambda_{is},
\end{equation}
where $\Pi_{is}$ is the eigenprojection to the top eigenspace of $\cL_{is}$, the essential spectral radius of $\Lambda_{is}$ is strictly less than $|\lambda(is)|$, and $\Lambda_{is}\Pi_{is} = \Pi_{is}\Lambda_{is} = 0$.
In particular, $\lambda_0=1$ and $\Pi_0=\mathbb E_{\bar\nu}[\cdot]\mathbf 1_{\bar{\Delta}}$.
Also, $s \mapsto (\lambda(is),\Pi_{is}, \Lambda_{is})$ is $C^{r+2}$ from
$[-\delta,\delta]$ to $\mathbb C\times(\cL(\cB))^2$. 
Iterating \eqref{OpDecom}, it follows that
\begin{equation}\label{IterOpDecom}
 \cL^n_{is} =\lambda(is)^n\Pi_{is}+\Lambda^n_{is} \, ,
\end{equation}
with $\sup_{s\in[-\delta,\delta]} \sup_{m=0,...,r+2}\Vert [\Lambda_{is}^n]^{(m)}\Vert_{\mathcal B}=\mathcal O(\vartheta_0^n)$ as $n\rightarrow +\infty$.

In the general case, we apply  \cite[Proposition A, Corollary 7.2]{HP} with $$I:=\{\cX_k,\ k=0,...,r+2\}\cup \{\cX^{(+)}_k,\ k=0,...,r+2\}$$ and the family of operators $\{\cL_{is}, s\in (-\delta,\delta)\}$. There is a slight deviation from the original notation in \cite{HP} where $I$ is the set of indices of the chain of Banach spaces. But the purpose of $T_0$ and $T_1$ defined below remain the same.

We first have to check Hypothesis $\mathcal D(r+2)$ of \cite[Appendix A]{HP}. Consider two maps $T_0, T_1:I\rightarrow I$ such that $$T_0(\cX_k)=\cX^{(+)}_k,\,\,\, T_1(\cX^{(+)}_k)=\cX_{k+1}.$$  Definition of $T_0$ and $T_1$ for other values is immaterial for us. This implies Condition {\bf (4)} of Hypothesis $\mathcal D(r+2)$ of \cite{HP}. Our assumptions on the Banach spaces imply Condition {\bf (0)} and our Assumption $(A)(3)$ implies Conditions {\bf (1)} and {\bf (2)}. It remains to prove Condition {\bf (3$^\prime$)} of  Hypothesis $\mathcal D(r+2)$ of \cite{HP}. Due to the Keller-Liverani perturbation theorem \cite{KL} recalled in \cite[Section 4]{HP}, this Condition {\bf (3$^\prime$)} comes from our Assumptions $(A)(3-4)$. 

Thus the conclusions of Proposition A and Corollary 7.2 of \cite{HP}, 
ensure the existence of $\delta>0$ (if necessary reducing the original $\delta$) such that 
and of a family $(\lambda(is),\Pi_{is},\Lambda_{is})_{s\in(-\delta,\delta)}$ which is $C^{j'}$-smooth from $(-\delta,\delta)$
to $\mathbb C\times \cL(\cV_{a_j},\cV_{b_{j+j'}})\times \cL(\cV_{a_j},\cV_{b_{j+j'}})$ satisfying \eqref{DecompOp3'}, \eqref{DecompOp3'Compl} with $\lambda(0)=1$, and the characterization of $\Pi_0$ and $\lambda(0)$ comes from our Assumption $(B)(1)$.

It remains to prove the expansion of $\lambda_{is}$.
Since $\int_{\bar\Delta}\bar\phi\, d\bar\nu = \int_{\mathcal M} \phi \, d\mu = 0$, 
it follows from \cite[Chapter 4]{HH} and from \cite[Lemmas 8.3, 8.4]{HP} that $\lambda'(0)=0$ and $$ \sigma_\phi^2:= \lambda''(0)=\underset{n \to \infty}{\lim} \EXP_{\bar{\nu}}\big(\bar{S}^2_n/n\big) =\underset{n \to \infty}{\lim} \EXP_{\mu}\big({S}^2_n/n\big)
\geq 0.$$ 
Since $\sum_{k\ge 0}\Vert \mathcal L^k(\bar\phi)\Vert_2<\infty$
(from Assumption $(B)(3)$) and since $\bar\phi$ is not a coboundary in $L^2(\bar\nu)$ (from Assumption $(0)$ and $(A)[r](1)$), $\sigma^2_\phi=\sum_{k\ge 0}\mathbb E_{\bar\nu}(\bar\phi\cL^{k}\bar\phi)>0$. Therefore,
\begin{equation}\label{devlambda}
\lambda(is)=1-\frac{\sigma^2_\phi}2s^2+o(s^2)\,\,\,\text{with}\,\,\, \sigma^2_\phi>0.
\end{equation} 
\end{proof}

\begin{prop}\label{lemmaABCgreek}
Suppose Assumptions $(0)$, $(A)[r]$ and $(B)(1,3)$ hold. Set $\psi_n=\psi$ and $\xi_n:=\xi\circ f^n$. Then
\begin{equation}\label{approxHn0}
\mathbb E_\mu(\psi e^{isS_n}\xi\circ f^n)=\EXP_{\bar\nu}\left(
\bar h_{k,-s,\xi}\mathcal L_{is}^{n-2k}(\mathcal L_{is}^{2k}\bar h_{k,s,\psi})\right) +
\mathcal O((1+|s|)\vartheta^k)\,,
\end{equation}
uniformly in $s\in\mathbb X^*$. 
Moreover,
\begin{itemize}
\item $(\alpha)[r]$ holds with $\sigma^2=\sigma^2_\phi$ and $B_0=\mathbb E_\mu(\psi)\mathbb E_\mu(\xi)$, and
\item Assumption $(C)$ implies $(\gamma)$.
\item Assumption $(D)[r]$ implies $(\delta)[r]$. 
\end{itemize}
\end{prop}

\begin{proof}
For $s\in\mathbb X^*$, we set
\begin{equation}\label{Hn}
H_n(s)=\EXP_\mu(\psi\, e^{isS_n}\, \xi\circ f^n)\lambda(is)^{-n}.
\end{equation}
The idea is now to approximate both $H_n(s)$ and $\lambda(is)^n$ by some expansions. Recall that $(S_n)^j\in L^{\frac{r_0+2}j}(\nu)$ and that $\xi\in L^{\frac{r_0+2}{q(\xi)}}(\nu)$ and $\psi\in L^{\frac{r_0+2}{q(\psi)}}(\nu)$ with $q(\xi)+q(\psi)\le r_0-r$. Note that $H_n$ is $C^{r+2}$ on $(-\delta,\delta)$. So, for $L \leq r+2$, we have 
\begin{equation}\label{formulaHn}
\left|H_n(s)-\sum_{N=0}^{L-1}\frac {H_n^{(N)}(0)}{N!}
    s^N\right|
\le  \sup_{u\in[0,1]}\left|H_n^{(L)}(us)\right|\, |s|^L\, ,
\end{equation}
where $H^{(0)}_n(0)=H_n(0)=\mathbb E_{\mu}[\psi\,\,\xi\circ f^n]$ for all $n$. 
Observe that the Assumption $(A)[r](1)$ implies that $S_n\circ\mathfrak p=\bar S_n\circ \bar {\mathfrak p}+\chi-\chi\circ F^n$. Combining this with the fact that $\mu=\mathfrak p_*\nu$ we obtain that
\begin{align}\label{Code1}
\EXP_\mu(\psi\,  e^{isS_n}\, \xi\circ f^n) 
&= \EXP_\nu(\psi\circ{\mathfrak p}\, e^{is\chi}e^{is \bar S_n\circ\bar{\mathfrak{p}}} (\xi\circ{\mathfrak p\,}\, e^{-is\chi}
)\circ F^n) \nonumber \\
&= \EXP_\nu(\psi\circ{\mathfrak p}\circ F^k e^{is\chi\circ F^k}e^{is\bar S_n\circ\bar{\mathfrak{p}}\circ F^{k}} \xi\circ{\mathfrak p\circ F^{n+k}}\,e^{-is\chi\circ F^{n+k}}) \nonumber \\
&=\EXP_\nu([\psi\circ{\mathfrak p}\circ F^ke^{is\chi\circ F^k}e^{-is\bar S_k\circ\bar{\mathfrak{p}}}]e^{is \bar S_n\circ\bar{\mathfrak{p}}} [ \xi\circ{\mathfrak p\circ F^{k}}e^{-is\chi\circ F^k}e^{is\bar S_k\circ\bar{\mathfrak{p}}}]\circ F^n) \nonumber \\
&
=\EXP_\nu(h_{k,s,\psi}e^{is \bar S_n\circ\bar{\mathfrak{p}}} h_{k,-s,\xi} \circ F^n)\, ,
\end{align}
using the notation ${h_{k,s,H}}$  given in the Assumption $(A)$.
Recall that the superscript $(j)$ denotes the $j$th derivative of a function with respect to $s$, and that the duality relation 
\begin{equation}\label{Code2}
\EXP_{\bar\nu}(g\, e^{is\bar{S}_n}\, h\circ \bar{F}^n)=\EXP_{\bar\nu}(h\,\cL^n_{is}(g))
\end{equation}
holds.

Estimating $h_{k,s,H}$ by $\bar h_{k,s,H}$ for $H\in\{\psi,\xi\}$, using $(A)[r]$ along with \eqref{Code1}, and assuming $3k\le n$, we obtain that for $N\le r+2$, \begin{align*}
H_n^{(N)}(s)&=\frac{\partial ^N}{\partial s^N}
(\EXP_\mu(\psi e^{isS_n}\xi\circ f^n)\lambda(is)^{-n})\nonumber\\
&=\sum_{m_1+m_2+m_3=N}\frac{N!}{m_1!m_2!m_3!}\EXP_\nu\left(h_{k,s,\psi}^{(m_1)}\frac{\partial^{m_2}}{\partial s^{m_2}}[e^{is \bar S_n\circ\bar{\mathfrak{p}}} \lambda(is)^{-n}]
h_{k,-s,\xi}^{(m_3)}\circ F^n\right) \nonumber\\
&=\sum_{m_1+m_2+m_3=N}\frac{N!}{m_1!m_2!m_3!}\EXP_{\bar\nu}\left(\bar h_{k,s,\psi}^{(m_1)}\frac{\partial^{m_2}}{\partial s^{m_2}}[e^{is \bar S_n} \lambda(is)^{-n}]
\bar  h_{k,-s,\xi}^{(m_3)}\circ \bar{F}^n\right)
\\
 &\phantom{\frac{N!}{m_1!m_2!m_3!}\EXP_{\bar\nu}\left(\bar h_{k,s,\psi}^{(m_1)}\frac{\partial^{m_2}}{\partial s^{m_2}}\right)aaaaaaa}+\mathcal O(n^{N}(1+|s|)(1+|s|\vartheta^k)\vartheta^k|\lambda(is)|^{-n})
 \, ,\nonumber
\end{align*}
uniformly in $s\in\mathbb X^*$
where we used
\eqref{hk1} combined with 
the fact that 
\[
\left\Vert\frac{\partial^{m_2}}{\partial s^{m_2}}[e^{is \bar S_n}
]\right\Vert_{L^\frac{r_0+2}{m_2}(\bar\nu)}=O\left(
\Vert \bar S_n\Vert^{m_2}_{L^{r_0+2}(\bar\nu)}=O(n^{m_2})
\right)\, 
\]
and that for $H=\xi$ or $H=\psi$,
\[
\Vert  h^{(j)}_{k,s,H}\Vert_{L^{\frac {r_0+2}{j+q(H)}}(\nu)}=
\Vert (\chi\circ F^k-\bar S_k\circ\bar p)^j h_{k,s,H}\Vert_{L^{\frac {r_0+2}{j+q(H)}}(\nu)}=O((1+k)^j)
\]
which with  \eqref{hk1} implies that $\Vert \bar  h^{(j)}_{k,s,H}\Vert_{L^{\frac {r_0+2}{j+q(H)}}(\bar\nu)}\le O((1+|s|\vartheta^k)(1+k)^j)$. When $N=0$, we replace the above estimate of $\Vert \bar  h^{(j)}_{k,s,H}\Vert_{L^{\frac {r_0+2}{j+q(H)}}(\bar\nu)}$ by \eqref{hk0} and obtain \eqref{approxHn0}.

We assume from now on that $|s|<\delta$. Due to \eqref{devlambdamodif}, up to decreasing (if necessary) the value of $\delta$, $s\mapsto |\lambda(is)|$ is decreasing on $(0,\delta)$ and increasing on $(-\delta,0)$, and 
\begin{equation}\label{minomajolambda}
\forall s\in(-\delta,\delta),\quad e^{-\frac{3
\sigma^2_\phi s^2}4}\le |\lambda(is)|\le e^{-\frac{
\sigma^2_\phi s^2}4}\, .
\end{equation}
Thus
\begin{align*}
&H_n^{(N)}(s)
=\sum_{n_1+n_2=N}\frac{N!}{n_1!n_2!}\EXP_{\bar\nu}\left(
\bar h_{k,-s,\xi}^{(n_1)}\frac{\partial^{n_2}}{\partial s^{n_2}}[\mathcal L_{is}^n(\bar h_{k,s,\psi})\lambda(is)^{-n}]\right) +\mathcal O(n^N\vartheta^k|\lambda(is)|^{-n})\\
&=\sum_{n_1+n_2=N}\frac{N!}{n_1!n_2!}\EXP_{\bar\nu}\left(
\bar h_{k,-s,\xi}^{(n_1)}\frac{\partial^{n_2}}{\partial s^{n_2}}[\mathcal L_{is}^{n-2k}(\mathcal L_{is}^{2k}\bar h_{k,s,\psi})\lambda(is)^{-n}]\right) +\mathcal O(n^N\vartheta^k|\lambda(is)|^{-n})\\
&=\sum_{n_1+n_2=N}\frac{N!}{n_1!n_2!}\EXP_{\bar\nu}\left(
\bar h_{k,-s,\xi}^{(n_1)}\frac{\partial^{n_2}}{\partial s^{n_2}}[(\lambda(is)^{-2k}\Pi_{is}(\cL_{is}^{2k}\bar h_{k,s,\psi})+\lambda(is)^{-n}\Lambda_{is}^{n-2k}(\cL_{is}^{2k}\bar h_{k,s,\psi})]\right)\\&\hspace{31em}+\mathcal O(n^{N}\vartheta^k|\lambda(is)|^{-n})\, ,
\end{align*}
uniformly on $s\in(-\delta,\delta)$,
using \eqref{DecompOp3'}.
From this, we deduce the first formula and the following three things:
\begin{itemize}[leftmargin=*]
\item First, 
for all $N \le r+2$,
$H_n^{(N)}(0)$ converges exponentially fast as $n\to \infty$.
Let $ n<m$. 

Note that if $(f,\cM,\mu)=(\bar F,\bar\Delta,\bar\nu)$, then $H_n^{(N)}(0)$ converges to $\mathbb E_{\bar\nu}\left(\xi\Pi_{i0}^{(N)}(\psi)\right)$. Otherwise, taking 
$k:=\lfloor n/5\rfloor$ and,
using \eqref{hk2} combined with \eqref{DecompOp3'Compl},
\begin{align*}
|H_n^{(N)}(0)-H_m^{(N)}(0)|&\ll k^N \vartheta_0^{n-2k}+m^N\vartheta^k \ll \vartheta_0^{\frac n 2}+m^N\vartheta^{n/5}\, 
\end{align*} 
and if  $n\le 2^\ell n\le m<2^{\ell+1}n$, then
\begin{align*}
|H_n^{(N)}(0)-H_m^{(N)}(0)|& \le 
 |H_{2^\ell n}^{(N)}(0)-H_m^{(N)}(0)|+\sum_{j=0}^{\ell-1} |H_{2^jn}^{(N)}(0)-H_{2^{j+1}n}^{(N)}(0)|\\
&\ll \sum_{j=0}^{\ell} (\vartheta_0^{n2^{j-1}}+n^N2^{(j+1)N}\vartheta^{2^{j}n/5})\\
&\ll \vartheta_0^{n/2}+\vartheta^{n/10}\, .
\end{align*}
This ensures the existence of
$B_N:= \lim_{n \to \infty} H^{(N)}_n(0)$ and that $H^{(N)}_n(0)=B_N+
\cO(\vartheta_1^{n})$ 
for all $N \leq r+2$, where $B_0=\mathbb E_{\mu}(\psi)\mathbb E_{\mu}(\xi)$
and $\vartheta_1:=\max(\vartheta_0^{\frac 12},\vartheta^{\frac 1{10}})$. 
 
We have proved that $e_N(n):=H^{(N)}_n(0)-B_N=O(\vartheta_1^n)$ for every nonnegative integer $N\le r+2$. Note that if $\psi=1$ or $\xi=1$, then $e_0(n) \equiv 0$.
\item 
Second, 
for $s\in(-\delta,\delta)$, 
\[ \sup_{u\in[0,1]}\left|H_n^{(L)}(us)\right|
\le |\lambda(is)|^{-n}(n^{L}\vartheta^k+ k^L\vartheta_0^{n-2k})+k^L|\lambda(is)|^{-2k}\, .
\]
\item Finally, observe that when $N=0$,
\begin{align*}
H_n(0)
&=\EXP_{\bar\nu}\left(\bar h_{k,s,\psi}
\bar  h_{k,-s,\xi}\circ \bar{F}^n\right)+\mathcal O(\vartheta^k) \, ,\\
&= \EXP_{\bar\nu}\left(\bar  h_{k,-s,\xi}\cL^n(\bar h_{k,s,\psi})
\right)+\mathcal O(\vartheta^k) \, ,\\
&= \EXP_{\bar\nu}\left(\bar  h_{k,-s,\xi}\right)\EXP_{\bar\nu}\left(\bar h_{k,s,\psi}
\right)+\mathcal O( \Vert \Lambda_{0}^{n-2k}\Vert+\vartheta^k\,).
\end{align*}

\end{itemize}
This ends the proof of $(\alpha)$ with $B_0=\mathbb E_\mu(\psi)\mathbb E_\mu(\xi)$ by taking $k=\vartheta^k=0$ if $(f,\cM,\mu)=(\bar F,\bar \Delta,\bar \nu)$
and $k=\lfloor (\log n)^2\rfloor$ otherwise.

Now let us assume $(C)$ and prove $(\gamma)$. The only case to study is the case $\mathbb X=\mathbb R$. Let $\cK$ be a compact subset of $\mathbb R\setminus\{0\}$. Recall that \eqref{approxHn0}  ensures that
\[
\mathbb E_\mu(\psi e^{isS_n}\xi\circ f^n)=\EXP_{\bar\nu}\left(
\bar h_{k,-s,\xi}\mathcal L_{is}^{n-2k}(\mathcal L_{is}^{2k}\bar h_{k,s,\psi})\right) +\mathcal O((1+|s|)\vartheta^k)\, ,
\]
uniformly in $s\in\cK$.
But Assumption $(C)$ implies that
\begin{align*}
\left\vert \EXP_{\bar\nu}\left(
\bar h_{k,-s,\xi}\mathcal L_{is}^{n-2k}(\mathcal L_{is}^{2k}\bar h_{k,s,1})\right) \right\vert
&\le  \Vert {\bar h_{k,s,\xi}}\Vert_{\cB_2'}\left\Vert \mathcal L_{is}^{n-2k}\right\Vert_{\mathcal B_1 \rightarrow 
\mathcal B_2}\left\Vert \mathcal L_{is}^{2k}{\bar h_{k,s,\psi}}\right\Vert_{\cB_1}\\
&\le (C'_0)^2|s|  ( C |s|^\alpha e^{-(n-2k)^{\alpha_1}\hat\delta |s|^{-\alpha}}  )
\end{align*}
and we conclude by taking $k=0$ if $(f,\cM,\mu)=(\bar F,\bar \Delta,\bar \nu)$
and $ k=\lfloor (\log n)^2\rfloor$ otherwise.

Finally, assume $(D)[r]$. Let $K>0$. Then for the $B$ given by Assumption $D[r]$,
using \eqref{approxHn0}, we compute, 
\begin{align*}
\int_{K<|s|<Bn^{\frac{r-1}2}}&\frac{|\EXP_\mu(\psi e^{isS_n})|}{|s|}\, ds \\ &=\int_{K<|s|<Bn^{\frac{r-1}2}}\frac{\left|\EXP_{\bar\nu}\left(
\bar h_{k,-s,1}\mathcal L_{is}^{n-2k}(\mathcal L_{is}^{2k}\bar h_{k,s,\psi})\right)\right|}{|s|}\, ds + \mathcal O(n^{(r-1)/2}\vartheta^k)\, ,
\end{align*}
and 
\begin{align*}
&\int_{K<|s|<Bn^{\frac{r-1}2}}\frac{\left|\EXP_{\bar\nu}\left(
\bar h_{k,-s,1}\mathcal L_{is}^{n-2k}(\mathcal L_{is}^{2k}\bar h_{k,s,\psi})\right)\right|}{|s|}\, ds \\ &\leq \int_{K<|s|<Bn^{\frac{r-1}2}}\frac{
\|\bar h_{k,-s,1}\|_{\cB_2^{\prime}}\|\mathcal L_{is}^{n-2k}\|_{\cB_1\to\cB_2}\|(\mathcal L_{is}^{2k}(\bar h_{k,s,\psi})\|_{\cB_1}}{|s|}\, ds \lesssim \int_{K<|s|<Bn^{\frac{r-1}2}}\frac{
\|\mathcal L_{is}^{n-2k}\|_{\cB_1\to\cB_2}}{|s|^{d_1}}\, ds
\end{align*}
provided $$\sup_{K<|s|<Bn^{\frac{r-1}2},\ k\ge 1} |s|^{-d_2} \Vert \bar h_{k,-s,1}\Vert_{\cB_2^{\prime}}\, \left\Vert \mathcal L_{is}^{2k}(\bar h_{k,s,\psi})\right\Vert_{\mathcal B_1}<\infty.$$
Then Assumption $(\delta)[r]$ follows 
by taking again $k=0$ if $(f,\cM,\mu)=(\bar F,\bar \Delta,\bar \nu)$
and $ k=\lfloor (\log n)^2\rfloor$ otherwise.

\end{proof}
The next result translates Assumption $(B)(2)$ in terms of uniform estimates of operator norms and proves that $(\beta)$ is a consequence of Assumptions $(A)[0]$ and $(B)$.
\begin{prop}\label{lemmaABCgreek1}
Suppose Assumptions $(0)$ and $(B)(2)$ hold, $s\mapsto \cL_{is}$ is continuous as a function from $\mathbb X^*$ to $\cL(\cX_0,\cX_{0}^{(+)})$ and $\cL_{is}\in\cL(\cX_0^{(+)})$ for every $s\in\mathbb X^*$. Set $\psi_n=\psi$ and $\xi_n=\xi\circ f^n$. 

Then, for all compact subset $\cK$ of $\mathbb X^*\setminus\{0\}$, there exists $\gamma\in(0,1)$ such that $$\sup_{s\in\cK}\Vert \cL_{is}^n\Vert_{\cX_0\rightarrow\cX_{0}^{(+)}}=\cO(\gamma^n).$$
Moreover, if Assumptions $(A)[0]$ and $(B)(1,3)$ hold true, then $(\beta)$ holds true.
\end{prop}
\begin{proof}
Let $0<\delta<K$ (with $K=\pi$ if $\mathbb X=\mathbb Z$). Let us prove the existence of $\gamma\in(0,1)$ such that
\[
\sup_{\delta\le |s|\le K}\Vert\cL_{is}^n\Vert_{\cX_0\rightarrow \cX_{0}^{(+)}}=\cO(\gamma^n)\, .
\]
This follows from the fact that, for every $s\in[\delta,K]$, there exists $n_s$ such that $$\Vert\mathcal L_{is}^{n_s}\Vert_{\cX_0\rightarrow \cX_0^{(+)}}<1$$ and we conclude by continuity of $s\mapsto\mathcal L_{is}^n$ from $\mathbb X^*$ to $\cL(\cX_0,\cX_0^{(+)})$ and by compactness of $[\delta, K]$. The continuity of $\cL_{is}^n$ comes from the following computation combined with our assumptions
\[
\cL_{is}^n-\cL_{it}^n=\sum_{k=0}^{n-1}\cL_{is}^{n-k-1}(\cL_{is}-\cL_{it})\cL_{it}^k\, .
\]
Let us assume Assumptions $(A)[0]$ and $(B)(1-3)$ and prove $(\beta)$.
Due to \Cref{lemmaABCgreek} and the above, for all $\delta<|s|<K$, we have
\begin{align*}
\left|\mathbb E_\mu(\psi e^{isS_n}\xi\circ f^n)\right|
&=\left|\EXP_{\bar\nu}\left(\bar  h_{k,-s,\xi}\cL_{is}^{n-2k}(\cL_{is}^{2k}\bar h_{k,s,\psi})
\right)\right|+\mathcal O(\vartheta^k) \, ,\\
&\le \Vert \bar  h_{k,-s,\xi}\Vert_{(\cX_0^{(+)})'} \cO(\gamma^{n-2k}) \Vert\bar h_{k,s,\psi}\Vert_{\cX_0}+\mathcal O(\vartheta^k) \, ,\\
&= \cO\left(\gamma^{n-2k})+\mathcal O(\vartheta^k\right)\, ,
\end{align*}
uniformly in $s$, and we conclude by taking $k=\vartheta^k=0$ if $(f,\cM,\mu)=(\bar F,\bar\Delta,\bar\nu)$
and $k=\lfloor (\log n)^2\rfloor$ otherwise.
\end{proof}

\section{Expansions of characteristic functions}\label{CharFun}

Let us show that our assumptions guarantee the existence of 
asymptotic expansions for $$\EXP_\mu(\psi_n e^{\frac{isS_n}{\sqrt{n}}} \xi_n)$$ as $n\to \infty$. In the next section, we use these expansions to establish our main results about higher order asymptotics for the CLT. 

\begin{lem}\label{lem0}
Assume $(\alpha)[r]$.
Then, up to a reduced value of $\delta$ if necessary, there exist polynomials $A_j$'s of the form $\sum_{l=0}^j a_{l,j}s^{2l+j}$ 
such that
\begin{equation}\label{CharAsympExp}
\EXP_\mu(\psi_n e^{\frac{isS_n}{\sqrt{n}}} \xi_n)=e^{-\frac{\sigma^2
s^2}{2}}\sum_{j=0}^{r}\frac{A_j(s)}{n^{j/2}}+ \mathfrak r_n(s),\ |s|<\delta\, ,
\end{equation}
where
\begin{align}
\mathfrak r_n(s)&=\mathcal O\left(
e^{-\frac{\sigma^2
s^2}8}\left(\frac{|s|^{r+2}\bar\psi_0(s/\sqrt{n})}{n^{\frac r2}}+e_0(n)+\frac{|s|+|s|^{3r+3}}{n^{\frac r2+\frac 14}}\right)+\frac{|s|^{r+1}}{n^{r+\frac 32}}\right)\nonumber
\end{align}
with $\bar\psi_0$ continuous and vanishing at $0$ and $e_0(n):=\mathbb E_\mu(\psi_n.\xi_n)-B_0$ where $B_0=a_{0,0}=\lim_{n\rightarrow +\infty} \mathbb E_\mu(\psi_n\,\xi_n)$ 
\end{lem}

\begin{proof}
For $s \in (-\delta,\delta)$, 
we set again $H_n(s)=\EXP_\mu(\psi_n\, e^{isS_n}\, \xi_n)\lambda(is)^{-n}$.
It follows from $(\alpha)$ that,
uniformly for $s\in(-\delta,\delta)$.
Therefore, for $s \in (-\delta,\delta)$,
\begin{align}\label{AAA1}
\EXP_\mu(\psi_n &e^{isS_n}\xi_n)\nonumber \\
&=\lambda(is)^{n}H_n(s) \nonumber \\ &=\lambda(is)^{n}\left(\sum_{N=0}^{r}\frac {B_N}{N!}s^N+\cO(e_0(n)+a_n(|s|+|s|^r))\right)+\cO\Big(
\left(b_ne^{-\frac {\sigma^2 s^2}8}+a_n\right)
|s|^{r+1}\Big)\, ,
\end{align}
with $e_0(n):=\mathbb E[\psi_n.\xi_n]-B_0$.
Due to \eqref{devlambda}, and since $\lambda$ is $C^{r+2}$, there exists a function $\psi_0 \in C^{r+2}$ such that $\psi_0(0)=\psi_0'(0)=\psi_0''(0)=0$ and
\begin{align*}
\lambda(is) =e^{-
\frac{\sigma^2 s^2}{2} + \psi_0(is)}\, .
\end{align*}
Writing $\psi_0(is)=s^2\psi_r(is)+s^{r+2}\bar\psi_0(s)$ where $\psi_r(s)=\frac{\psi_0^{(3)}(0)s}{3!}+\dots+\frac{\psi_0^{(r+2)}(0)s^{r}}{(r+2)!}$, $\bar\psi_0(0)=0$ and $\bar\psi_0$ is continuous
and using
\[
\left|e^\alpha-\sum_{m=0}^r\frac{\beta^m}{m!}\right|
\le |e^\alpha-e^\beta|+\left| e^\beta-\sum_{m=0}^r\frac{\beta^m}{m!}\right|
\le e^{\max(\alpha,\beta)}
\left(|\alpha-\beta|+\frac{|\beta|^{r+1}}{(r+1)!}\right)\, ,
\]
with $\alpha:=n\psi_0(is/\sqrt n)$ and $\beta:=s^2\psi_r\big(\frac {is}{\sqrt{n}}\big)$, we obtain,
up to decreasing the value of $\delta$ if necessary,
\begin{align*}
\left|
\lambda\left(\frac {is}{\sqrt{n}}\right)^ne^{\frac{\sigma^2 s^2}2}-\sum_{m=0}^r\frac{\left(s^2\psi_r\big(\frac {is}{\sqrt{n}}\big)\right)^m}{m!}\right|\le
\mathcal O\left(e^{\frac {\sigma^2 s^2} 4}\left( \frac{|s|^{r+2}\bar\psi_0(
s/\sqrt{n})}{n^{\frac r2}}+ \frac{|s|^{3r+3}}{n^{\frac{r+1}2}}\right)\right)
\, ,
\end{align*}
for $|s|<\delta\sqrt{n}$.
Combining this with \eqref{AAA1}, we obtain
\begin{align*}
&\EXP_\mu(\psi_n e^{\frac{isS_n}{\sqrt{n}}}\xi_n) =\lambda\Big(\frac{is}{\sqrt{n}}\Big)^{n}H_n\Big(\frac{s}{\sqrt{n}}\Big) \\ &=e^{-\frac{\sigma^2 s^2}{2}}\sum_{m=0}^r\frac{\big(s^2\psi_r\big(\frac {is}{\sqrt{n}}\big)\big)^m}{m!}
\sum_{N=0}^{r}\frac {B_N}{N!}\left(\frac{s}{\sqrt{n}}\right)^N+\cO\left(e^{-\frac {
\sigma^2
 s^2}{4}}\left( \frac{|s|^{r+2}\bar\psi_0(s/\sqrt{n})}{n^{\frac r2}}+ \frac{|s|^{3r+3}}{n^{\frac{r+1}2}}\right)\right)\\
&\ \ \ +\cO\bigg(\frac{|s|^{r+1}}{n^{\frac{r+1}2}}\Big[
b_ne^{-\frac {\sigma^2 s^2}8}+a_n
\Big]\bigg) +\cO\bigg(\left|\lambda\left(\frac{is}{\sqrt{n}}\right)\right|^n\left(e_0(n)+\vartheta_1^n\left(\frac{|s|}{\sqrt{n}}+\frac{|s|^r}{n^{\frac r2}}\right)\right)\bigg)\, ,
\end{align*}
for $|s|<\delta\sqrt{n}$.
Collecting the individual terms in the principal part of the right hand side according to ascending powers of $\frac{1}{\sqrt{n}}$ and absorbing
higher order terms into the error term, we obtain
\begin{equation}\label{Aj}
e^{-\frac {\sigma^2 s^2}{2}}\sum_{m=0}^r\frac{\big(s^2\psi_r\big(\frac {is}{\sqrt{n}}\big)\big)^m}{m!}
\sum_{N=0}^{r}\frac {B_N}{N!}\left(\frac{s}{\sqrt{n}}\right)^N=e^{-\frac{\sigma^2 s^2}{2}}
\left(
\sum_{j=0}^{r}\frac{A_j(s)}{n^{j/2}}+ \cO\left(
\frac{|s|^{3r+1}
+|s|^{r+3}
}{n^{(r+1)/2}}\right)
\right)
\end{equation}
where $A_j$'s are polynomial of the form $\sum_{m=0}^j a_{m,j}s^{2m+j}$ where
\begin{align}\label{Coeffamj}
a_{m,j}&:=\frac {(-1)^m}{m!}\sum_{N= 0}^{j-m}\frac {B_N}{N!}\sum_{k_1,...,k_m\ge 3\, :\, k_1+...+k_m+N=2m +j}\prod_{l=1}^m\frac{i^{k_l}\psi_0^{(k_l)}(0)}{(k_l)!} \nonumber \\
&=
\frac {1}{m!}\sum_{N= 0}^{j-m}\frac {B_N\times [(-\psi_0(i\,\,\cdot\,))^m]^{(2m+j-N)}(0)}{N!
(2m+j-N)!
}\, .
\end{align}
recalling that 
$\psi_0(0)=\psi'(0)=\psi''(0)=0$.
Substituting this back and choosing $k = 
\lfloor(\log n)^2\rfloor
$, we obtain \eqref{CharAsympExp} with $A_0\equiv B_0
$. 
\end{proof}

\begin{rem}
Note that $A_j$ $($as a function$)$ has the same parity as $j$ $($as an integer$)$, and its coefficients, $a_{m,j}$'s, are explicitly expressed by \eqref{Coeffamj} where
$$B_N=\lim_{n\rightarrow +\infty}\left[\EXP_\mu(\psi_n\, e^{isS_n}\, \xi_n)\lambda(is)^{-n}\right]^{(N)}_{s=0}.$$
In the particular case of $(f,\mathcal M,\mu)=(F,\Delta,\nu)=(\bar F,\bar\Delta,\bar\nu)$, this reduces to $$B_N=\mathbb E_{\mu}(\xi \Pi^{(N)}_{i0}(\psi)).$$ 
The computation of $B_N$ in the general case is more delicate $($see \cite[Appendix A]{PN}$)$.
\end{rem}
%

\begin{rem}
In the special case of $\xi_n\equiv 1$ and $\mathbb E_\mu(\psi_n)=1$, in particular, if $\psi_n$'s are probability densities, we have that $e_0(n)=0$ for all $n$. Hence,
\begin{align}\label{CharExpErr2}
\mathfrak r_n(s)=\mathcal O\left(|s|\left[
e^{-\frac{\sigma^2 s^2}8}\left(\frac{|s|^{r+1}\bar\psi_0(s/\sqrt{n})}{n^{\frac r2}}+\frac{1+|s|^{3r+2}}{n^{\frac r2+\frac 14}}\right)+\frac{|s|^{r}}{n^{r+1}}\right]\right)\, .
\end{align}
This fact is used in the proof of \Cref{StongExp}. 
\end{rem}

\section{Proofs of the main results}\label{proofs}
In this section, we prove our main general
theorems stated in Section~\ref{se:assum}. First, we prove \Cref{WeakGlobalExp}.
\begin{proof}[Proof of Theorem~\ref{WeakGlobalExp}]
Let $g \in \fF^{q+2}_{0}$ with $q>\alpha\big(1+\frac{r}{2\alpha_1}\big)$. 
Recall $\widehat g$ has been defined in \eqref{hat g}
and that
$|\wh g(s)|\le C^{q+2}(g)\min(1,|s|^{-b})$ with $b=q+2> 2+\alpha+\frac{\alpha r}{2\alpha_1}$.
Since $g$ is continuous (if $\mathbb X=\mathbb R$), $\lambda$-integrable and $\hat g$ is integrable on $\mathbb X^*$, it follows that
\begin{align}\label{Planchrl}
\EXP_\mu\left(\psi_n\, g(S_n)\, \xi_n\right) = \frac{1}{2\pi}\int
_{\mathbb X^*} \wh{g}(s)\EXP_\mu(\psi_n\, e^{isS_n}\, \xi_n)\, ds
\, .
\end{align}
\noindent
{\bf Step 1:} 
Using Assumptions $(\beta)$, $(\gamma)$ 
and the decay of $\widehat g(s)$, it follows 
that
\begin{align}
\int_{s\in\mathbb X^*,\, 
|s|>\delta} \big|\widehat g(s)&\EXP_\mu\left(\psi_n\, e^{isS_n}\, \xi_n\right)\big|\, ds \nonumber \\
&\le \int_{s\in\mathbb X^*,\, 
|s|>K}|s|\left|\widehat g(s)\right|\mathcal O(a_n+|s|^{1+\alpha} e^{-n^{\alpha_1}\hat\delta |s|^{-\alpha}}) \, ds \, +\,  C^{q+2}(g)\cdot\cO(a_n
)\nonumber \\
&\le C^{q+2}(g)\int_{|s|>K} |s|\mathcal O(|s|^{\alpha-b} e^{-n^{\alpha_1\hat\delta |s|^{-\alpha}}})\, ds \,+\,  C^{q+2}(g)\cdot\cO(
a_n
)\nonumber \\
&\le C^{q+2}(g)n^{\frac{\alpha_1}{\alpha}}\int_{\mathbb R} |u|n^{\frac{\alpha_1}{\alpha}}
\cO(n^{{\alpha_1}(1-\frac b\alpha)}|u|^{\alpha-b} e^{-\hat\delta|u|^{-\alpha}})
\, du \,+\,  C^{q+2}(g)\cdot\cO(n^{1-\frac{q}\alpha})\nonumber\\
&= C^{q+2}(g)\cdot \mathcal O\left(n^{\alpha_1(1-\frac{b-2}\alpha)}+n^{1-\frac{q}\alpha}\right)=C^{q+2}(g) \cdot\cO\left(n^{\alpha_1(1-\frac{q}\alpha)}\right).\label{majoint0}
\end{align}
The last equality is true since $b=q+2$.

Recall that, due to Assumption $(\alpha)$, for every $s\in(-\delta,\delta)$,
\begin{align*}
\left|\mathbb E_\mu(\psi_n e^{isS_n}\xi_n)\right|=|\lambda(is)^n H_n^{(0)}(s)|\le b_ne^{-\frac {\sigma^2 s^2}8}+a_n\, .
\end{align*}
Fix $c>\frac{4(r+2)}{\sigma^2}$.
For $\sqrt{c\log n} <|s|<\delta $, 
$
\mathbb E_\mu(\psi_n e^{isS_n}\xi_n) = \cO\left( {n^{-(r+2)/2}} \right),$
and hence,
\begin{align*}
\int_{\sqrt{\frac{c\log n}{n}}<|s|<\delta} \big|\widehat g(s)\mathbb E_\mu(\psi_n e^{isS_n}\xi_n)\big|\, ds  &\lesssim \frac{1}{n^{(r+2)/2}} \int_{
\mathbb X^*
}|\wh{g}(s)|\, ds  \lesssim \frac{C^{q+2}(g) }{n^{(r+2)/2}}.
\end{align*}
This combined with~\eqref{Planchrl}
and \eqref{majoint0} leads to
\begin{align}\label{Planchrl2}
\EXP_\mu\left(\psi_n\, g(S_n)\, \xi_n\right) 
&= \frac{1}{2\pi}\int_{
s\in\mathbb X^*,\, 
|s|<\sqrt{\frac{c\log n}{n}}} \widehat g(s)\EXP_\mu\left(\psi_n\, e^{isS_n}\, \xi_n\right)\, ds \nonumber \\ &\hspace{15em}+ C^{q+2}(g) \cdot \cO(n^{\alpha_1(1-\frac{q}\alpha)}+n^{-\frac{r+2}2})\nonumber \\
&=\frac{1}{2\pi \sqrt{n}}\int_{|s|<\sqrt{c\log n}} \widehat g\Big(\frac{s}{\sqrt{n}}\Big)\EXP_\mu\left(\psi_n\, e^{i\frac{sS_n}{\sqrt{n}}}\, \xi_n\right)\, ds \\ 
&\hspace{15em}+ C^{q+2}(g) \cdot \cO(n^{\alpha_1(1-\frac{q}\alpha)}+n^{-\frac{r+2}2}) \, .\nonumber
\end{align}
Note that our assumption on $q$ ensures that the above error term is $o(n^{-\frac r2})$. \\

\noindent
{\bf Step 2}:
Recall the polynomials $A_j$'s from Lemma~\ref{lem0}.
Note that 
$$\int_{
\mathbb X^*} A_j(s) 
\wh{g}\Big(\frac{s}{\sqrt{n}}\Big)e^{-\frac{\sigma^2 s^2}{2}}\, ds = \int_{|s|<\sqrt{c\log n
}} A_j(s) 
\wh{g}\Big(\frac{s}{\sqrt{n}}\Big)e^{-\frac{\sigma^2 s^2}{2}}\, ds + 
C^{q+2}(g) \cdot o(n^{-r/2})$$
because 
$ \big|A_j(s)e^{-\frac{\sigma^2 s^2}{4}}\big| =\cO \left( n^{-(r+1)/2}\right)$ uniformly in $|s|>\sqrt{c
\log n
}$ .

Also, because  
$$(is)^k e^{-\frac{\sigma^2 s^2}{2}}=(is)^k \widehat{\fn} (s)=\widehat{\fn^{(k)}}(s)\, ,$$
where $\fn(s)=\frac{e^{-\frac{s^2}{2\sigma^2 }}}{\sqrt{2\pi \sigma^2}}$, defining the polynomials $R_j$'s via the relation
\begin{equation}\label{StronExpPolyDer}
R_j(s)\fn(s)=\frac{1}{\sqrt{2\pi \sigma^2
}}A_j\left(-i\frac{d}{ds}\right)\fn(s)\, ,
\end{equation}
we conclude that $A_j(s)e^{-\frac{\sigma^2 s^2}{2}}=\widehat{R_j \cdot \fn}(s),$ and so
$$\int_\reals \sqrt{n}g\left(s\sqrt{n}\right)R_j(s)\fn(s) \, dt = \frac{1}{2\pi}\int_{\reals}\widehat{g}\Big(\frac{s}{\sqrt{n}}\Big) A_j(s) e^{-\frac{\sigma^2 s^2}{2}} \, ds\, .$$
Substituting from \eqref{CharAsympExp}, integrating, and using the above observations, we obtain,
\begin{align*}
\frac{1}{2\pi\sqrt{n}}&\int_{|s|<\sqrt{c
\log n}} \widehat g\Big(\frac{s}{\sqrt{n}}\Big)\, \EXP_\mu\left(\psi_n\, e^{i\frac{sS_n}{\sqrt{n}}}\, \xi_n\right)\, ds \\ &=  \sum_{j=0}^{r}\frac{1}{ n^{j/2}}\frac{1}{2\pi \sqrt{n}}\int_{|s|<\sqrt{
c\log n
}} A_j(s) \wh{g}\Big(\frac{s}{\sqrt{n}}\Big)e^{-\frac{\sigma^2 s^2}{2}}\, ds + 
C^{q+2}(g)\cdot o(n^{-r/2}) \\ 
&=\sum_{j=0}^{r}\frac{1}{ n^{j/2}}\frac{1}{2\pi \sqrt{n}}\int
_{
\sqrt{n}\, \mathbb X^*} A_j(s) \wh{g}\Big(\frac{s}{\sqrt{n}}\Big)e^{-\frac{\sigma^2 s^2}{2}}\, ds + 
C^{q+2}(g) \cdot o(n^{-r/2}) \\
&
=\sum_{j=0}^{r}\frac{1}{ n^{j/2}}\frac{1}{2\pi}\int
_{
\mathbb X^*}
\widehat{R_j \cdot \fn}(s\sqrt{n}) \wh{g}\left(s\right)
\, ds + 
C^{q+2}(g) \cdot o(n^{-r/2}) \\
&
=\sum_{j=0}^{r}\frac{1}{ n^{(j-1)/2}}\int_{\mathbb X} (R_j\cdot \fn)\left(\frac x{\sqrt{n}}\right) g(x) \, d\lambda(x) + 
C^{q+2}(g) \cdot o(n^{-r/2})\, .
\end{align*}

In fact, when $\mathbb X=\mathbb R$, the last equality is due to the inverse Fourier formula, and when $\mathbb X=\mathbb Z$, it is due to the Plancherel formula and the Poisson formula which ensures that 
$$\sum_{m\in\mathbb Z}\widehat{R_j \cdot \fn}\left((s+2\pi m)\sqrt{n}\right)=\widehat a(s)$$
where $ a:\mathbb Z\rightarrow \mathbb C$ is given by $$a(m)=\frac 1{2\pi\sqrt{n}}(R_j \cdot \fn)\Big(\frac{m}{\sqrt{n}}\Big).$$

Finally, substituting this in \eqref{Planchrl2}, and combining error terms, we obtain the required expansion:
\begin{align*}
\EXP_\mu\left(\psi_n\, g(S_n)\, \xi_n\right)  &=\sum_{j=0}^{r}
\frac{1}{ n^{(j-1)/2}}
\int
_{\mathbb X} (R_j\cdot\fn)(x/\sqrt{n}) 
g
(x) \, d\lambda(x) + 
C^{q+2}(g)
\cdot o(n^{-r/2})\, .
\end{align*}\end{proof}

Next, we prove \Cref{WeakLocalExp}.

\begin{proof}[Proof of Theorem \ref{WeakLocalExp}]
First part of the proof is the same as Step 1 in the previous proof. Construction of the actual expansion (Step 2 above) is different. 

Note that because of the slightly stronger assumption of $q>\alpha\big(1+\frac{r+1}{2\alpha_1}\big)$, applying the Step 1 of the proof of Theorem~\ref{WeakGlobalExp}, we obtain \eqref{Planchrl2} and so
\begin{align*}
\EXP_\mu\left(\psi_n\, g(S_n)\, \xi_n\right) &= \frac{1}{2\pi}\int
_{|s|<\sqrt{\frac{c'
\log n}{
n}}} \widehat g(s)\EXP_\mu\left(\psi_n\, e^{isS_n}\, \xi_n\right)\, ds+ C^{q+2}(g) \cdot o(n^{-\frac{r+1}{2}})\, ,
\end{align*}
since $1-\frac q\alpha<-\frac {r+1}{2\alpha_1}$ and $\alpha_1 \in (0,1]$, 
with  $c'>\frac{4(r+3)}{\sigma^2}$.

Because $g \in \fF^{q+2}_{r+1}$, we can replace $\wh{g}$ with its order $r$ Taylor expansion,
$$\wh{g}(s)=\sum_{j=0}^{r}\wh{g}^{(j)}(0)\frac{s^j}{j!}+\frac{s^{r+1}}{r!}\int_0^1 (1-x)^r\wh{g}^{(r+1)}(xs)\, dx$$
on $\left\{|s|<\sqrt{\frac{c'\log n}{n}}\right\}$. Here $|\wh{g}^{(r+1)}(x)|\leq C_{r+1}(g)$. 

Along with this, we use the expansion of $\EXP_\mu\left(\psi_n\, e^{i\frac{sS_n}{\sqrt{n}}}\, \xi_n\right)$ provided by \eqref{CharAsympExp} to get 
\begin{align*}
\sqrt{n}\EXP_\mu\big(\psi_n\, g(S_n)\, \xi_n\big) 
&= \frac{1}{2\pi}\int_{|s|<\sqrt{c'\log n
}} \widehat g\Big(\frac{is}{\sqrt{n}}\Big)\EXP_\mu\left(\psi_n\, e^{i\frac{sS_n}{\sqrt{n}}}\, \xi_n\right)\, ds+ C^{q+2}_{r+1}(g) \cdot o(n^{-\frac{r}{2}}) \\ 
&= \frac{1}{2\pi}\sum_{l=0}^{r}\sum_{j=0}^r\frac{\wh{g}^{(l)}(0)}{n^{(l+j)/2}l!}\int_{|s|<\sqrt{c'\log n
}} s^l A_j(s) e^{-\frac{\sigma^2 s^2}{2}}\, ds+ C^{q+2}_{r+1}(g) \cdot o(n^{-\frac{r}{2}}).
\end{align*}

As in the previous proof, we can replace the integrals with integrals over $\reals$ due to our choice of $c$. In addition, because $A_j$ has the parity of $j$, when $l+j$ is odd, $$\int_{
|s|<\sqrt{c'\log n}} s^l A_j(s) e^{-\frac{\sigma^2 s^2}{2}}\, ds = 0\, .$$ Therefore, 
\begin{align*}
&\sqrt{n}\EXP_\mu\big(\psi_n\, g(S_n)\, \xi_n\big)  \\ 
&= \frac{1}{2\pi}\sum_{m=0}^{r}\sum_{l+j=2m}\frac{\wh{g}^{(l)}(0)}{n^{(l+j)/2}l!}\int_{\reals} s^l A_j(s) e^{-\frac{\sigma^2 s^2}{2}}ds+ C^{q+2}_{r+1}(g)\cdot o(n^{-\frac{r}{2}}) \\
&= \frac{1}{2\pi}\sum_{m=0}^{r}\frac{1}{n^m}\sum_{l+j=2m}  
\int_{\mathbb X} (-ix)^lg(x) \, d\lambda(x)
\frac{1}{l!}\int_{\reals} s^l A_j(s) e^{-\frac{\sigma^2 s^2}{2}}ds+ C^{q+2}_{r+1}(g) \cdot o(n^{-\frac{r}{2}})\\ 
&=\sum_{m=0}^{r} \frac{1}{n^m}\int _{\mathbb X}g(x)Q_m(x)
\, d\lambda(x)\left[\sum_{l+j=2m} \left( \int_{\reals} s^l A_j(s) e^{-\frac{\lambda''(0)s^2}{2}}ds \right) \frac{(-is)^l}{l!} \right]
 + C^{q+2}_{r+1}(g) \cdot o(n^{-\frac{r}{2}}).
\end{align*}
where
\begin{equation}\label{Qm}
Q_m(x)=\frac 1{2\pi}\sum_{l+j=2m} \left( \int_{\reals} u^l A_j(u) e^{-\frac{\sigma^2 u^2}{2}}du \right) \frac{(-ix)^l}{l!}\, ,
\end{equation} 
with $A_j$'s defined as in \eqref{Aj}.
Then, absorbing the higher order terms into the error we obtain the required expansion. 
\begin{align*}
\sqrt{n}\EXP_\mu\big(\psi_n\, g(S_n)\,\xi_n\big)=
\sum_{m=0}^{[r/2]}  \frac{1}{n^m} \int_{\mathbb R} g(s) Q_m(s) \, ds + C^{q+2}_{r+1}(g) \cdot o(n^{-\frac{r}{2}}).
\end{align*}
\end{proof}

Now, we prove \Cref{StongExp}. 
\begin{proof}[Proof of Theorem~\ref{StongExp}]
Recall the polynomials $R_j$'s defined by \eqref{StronExpPolyDer} with 
$ \xi_n\equiv 1$ and $\psi_n=\psi$. 
Define $P_j$'s to be the polynomials satisfying
\begin{equation}\label{EdgePolyRel}
\fn(x) R_j(x)=\frac{d}{dx}\Big[\fn(x)P_j(x)\Big].
\end{equation}
$P_j$'s are candidates for the polynomials appearing in the expansion. 
\\

\noindent 
{\bf Step 1:} From the Berry-Ess\'een inequality 
\begin{equation}\label{BerryEsseen1}
\left|F_{n}(x)-\cE_{r',n}(x)\right|\le \frac 1\pi\int_{-T}^{T}\left|\frac{\EXP_{\mu}(\psi_n\,
e^{\frac{isS_n}{\sqrt{n}}})-\widehat \cE_{r',n}(s)}{s}\right|\, ds+\frac{C_0}{T},
\end{equation}
where $$F_n(x)=\mathbb P \left(\frac{S_n}{\sqrt{n}}\leq x\right),\,\,\, \cE_{r',n}(x)=\fN(x)+\sum_{j=1}^{r'} \frac{P_j(x)}{n^{j/2}} \fn(x),$$
$$\widehat \cE_{r',n}(s) = \int_{\mathbb R} e^{-isx}\, d\cE_{r',n}(x)= e^{-\frac{\sigma^2 s^2}{2}}\sum_{j=0}^{r'}\frac{A_j(s)}{n^{j/2}}, $$ and $C_0$ is independent of $T$. 

We refer the reader to \cite[Chapter XVI.3,4]{Feller2} for a detailed discussion of the Berry-Ess\'{e}en inequality and its utility in establishing Edgeworth expansions for i.i.d.~sequences of random variables. Here we adapt those techniques to the non-i.i.d.~setting we deal with. \\

\noindent
{\bf Step 2:}  Now, we estimate the RHS of \eqref{BerryEsseen1} for an appropriate choice of $T$:

Given $\ve >0$, choose $B> \frac{C_0}{\ve}$. Then 
\begin{align}\label{BEOrdr}
|F_{n}(x)-\cE_{r',n}(x)| &\leq \frac 1\pi\int_{-Bn^{r/2}}^{Bn^{r/2}}\left|\frac{\EXP_{\mu}(\psi_n\, e^{is\frac{S_n}{\sqrt{n}}})-\widehat{\cE}_{r',n}(s)}{s}\right|\, ds+\frac{C_0}{Bn^{r/2}} 
 \\ 
& \leq I_1 + I_2+ I_3 + \frac{\ve}{n^{r/2}} \nonumber
\, .
\end{align}
where
\begin{align*}
I_1 &=\frac 1\pi\int_{|s|<\delta\sqrt{n}}\left|\frac{\EXP_{\mu}(\psi_n\,
e^{is\frac{S_n}{\sqrt{n}}})-\widehat{\cE}_{r',n}(s)}{s}\right|\, ds\, ,  \\
I_2 &=\frac 1\pi\int_{\delta\sqrt{n}<|s|<Bn^{r/2}}\left|\frac{\EXP_{\mu}(\psi_n\, e^{\frac{is \bar S_n}{\sqrt{n}}})}{s}\right|\, ds\, , \\
I_3 &=\frac 1\pi\int_{|s|>\delta\sqrt{n}} \left|\frac{\widehat{\cE}_{r',n}(s)}{s}\right|\, ds.
\end{align*}

Using \eqref{CharExpErr2},
\begin{align}\label{Near0Est}
 I_1&
= \cO\left(n^{-r'/2}\int_{|s|<\delta \sqrt{n}}\left( \bar\psi_0\Big(\frac{s}{\sqrt{n}}\Big)|s|^{r'+1}+\frac{1+|s|^{3r'+2}}{n^{\frac 14}}\right) e^{-\frac{\sigma^2 s^2}{8}}+\frac{|s|^{r'}}{n^{\frac {r'}2+1}}\, ds \right)\\ &= o(n^{-r'/2}) \, .\nonumber
\end{align}
Also, note that there exists a polynomial 
$$P(s)=\sum_{j=0}^r\sum_{\ell=0}^j|a_{\ell,j}|s^{2\ell+j}$$ such that $\big|\widehat{\cE}_{r,n}(s)\big|\le e^{-\frac{\sigma^2 s^2}{2}}P(|s|)$.
Therefore,
$$I_3=\frac 1\pi\int_{|s|>\delta\sqrt{n}} e^{-\frac{\sigma^2 s^2}{2}}\frac{P(|s|)}{|s|}\, ds = \cO(e^{-c'n})\, ,$$
for some $c'>0$. 

Because our choice of $\ve >0$ is arbitrary, if $I_2=o(n^{-r/2})$, then the proof is complete. To show this, we split $I_2$ to two integrals:
\begin{align}\label{decompI2}
I_2 &=\int_{\delta<|s|<K}\left|\frac{\EXP_{ \mu}(\psi_n\, e^{is S_n})}{s}\right|\, ds\, + \int_{K<|s|<Bn^{(r-1)/2}}\left|\frac{\EXP_{\mu}(\psi_n\,e^{is S_n})}{s} \right| \, ds\, ,
\end{align}
where $K$ as in the assumption $(\delta)[r]$ for our choice of $B$.
From the assumptions $(\beta)$ and $(\delta)[r]$, it follows that 
$$\int_{\delta<|s|<Bn^{(r-1)/2}}\left|\frac{\EXP_{\mu}(\psi_n e^{is S_n})}{s} \right| \, ds = o(n^{-r/2}).$$
This gives the required asymptotics for the right hand side of \eqref{decompI2}. 
\end{proof} 

Finally, we prove Corollaries~\ref{StongExpOrd1} and~\ref{BetterExpOrd1}.
\begin{proof}[Proof of Corollary~\ref{StongExpOrd1}]
We apply Theorem~\ref{StongExp} with $r'=r=1$, noticing that $(\delta)[1]$ holds true with $K=B$.
\end{proof}

\begin{proof}[Proof of Corollary~\ref{BetterExpOrd1}]
We apply Theorem~\ref{StongExp} with $r'=2$. Since $r\in (1,2)$, $\min\{r',[r]\}=\min\{r',1\}=1$, so there is only one term in the expansion. The error is $o(n^{-\min\{r,r'\}/2})=o(n^{-r/2})$ as $r \leq 2=r'$. 
\end{proof}

\part{Examples}\label{part2}

\section{Verification of Assumptions for Dynamical Systems}\label{verif}

In this section, we describe how to verify the assumptions stated in the abstract setting. The two key examples we focus on are 
\begin{itemize}[leftmargin=*]
\item mixing subshifts of finite type (SFTs):\\
We will consider the case when $(F,\Delta,\nu)=(\sigma,\Sigma_A,m)$
is an SFT on $k-$symbols $\cA=\{1,\dots, k\}$, with incidence matrix $A\in \{0,1\}^{\cA\times\cA}$ endowed with the Gibbs measure
associated to a Lipschitz continuous potential. Recall that $\sigma$ is the restriction of the shift map $\sigma((x_n)_{n\in \integers })=(x_{n+1})_{n\in \integers }$ to the subset $\Sigma_A$ of $\cA^{\mathbb Z}$ made of words $\vec x=(x_n)_{n\in\mathbb Z}$ such that
$A(x_n,x_{n+1})=1$ for every $n$. A will be assumed to be irreducible and aperiodic.
A more detailed description is provided in \Cref{SFT}. 
For the rest of the paper, we will write $
\vec x
$ instead of $(x_n)_{n\in \integers}$.\vspace{5pt}

\item tower systems (with exponential tails) constructed by Young in \cite{Y}:\\
We will also consider the case when $(F,\Delta,\nu)$ is the Young tower tailored to some hyperbolic system $(f,\mathcal M,\mu)$ as in \cite{Y}. Recall that this corresponds to a space $$\Delta=\bigcup_{i\ge 0}\left(\Lambda_i\times\{0,...,r_i-1\}\right)$$ for some nice sets $\Lambda_i\subset\mathcal M$ and positive integers $r_i$ corresponding to some nice return time constant on $\Lambda_i$ of the initial map $f$ to the set $\Lambda:=\bigcup_{i\ge 0}\Lambda_i$ along with dynamics given by $F(x_0,l)=(x_0,l+1)$ and $F(x_0,r_i-1)=(f^{r_i}(x_0),0)$ if $x_0\in\Lambda_i$ and $l<r_i-1$. More details are in \Cref{Towers}.
\end{itemize} 
In particular, our setting allows us to obtain more precise CLTs for axiom A attractors, Sinai billiards and H\'enon-type attractors. 
\subsection{Assumptions $(A)(1)$ and $(A)(2)$ in a hyperbolic framework}\label{sechyp}

Any dynamical system that satisfies the two assumptions $(H0)$ and $(H1)$ given below, satisfies $(A)(1)$ and $(A)(2)$. In particular, we will see in the next section that our two key examples satisfy these assumptions. \\

\noindent
{\bf Assumption $(H0)$:}
\begin{itemize}[leftmargin=*]\setlength\itemsep{10pt}
\item {\bf Separation time.}
We consider a dynamical system $(F,\Delta,\nu)$ endowed with a separation time $\hat s(\cdot,\cdot):\Delta\times\Delta\rightarrow\mathbb N_0\cup\{\infty\}$ satisfying $\hat s(x,y)= n+\hat s(F^nx,F^ny)$ whenever $\hat s(x,y)> n$.\\ \\
\underline{For 
an SFT}: $\hat s(\vec x
,\vec y
):=\inf\{k\ge 0\, :\, x_k\ne y_k\}$.\\ 
\underline{For a 
Young tower}: 
Consider the numerable partition of $\Delta$ in $\Delta_{l,j}$ introduced in \cite{Y}. Recall that each $\Delta_{l,j}$ is contained in the $l$-th level $\Delta_l$ of the tower $\Delta$. We take for $\hat s(x,y)$
the infimum of the integers $n\ge 0$ such that $F^n(x)$
and $F^n(y)$ do not belong to a same atom of the partition $\{\Delta_{l,j},\ l,j\}$.

\item {\bf Hyperbolicity.}
We assume that $(F,\Delta,\nu)$ is hyperbolic in the following sense.
There exist a family $\Gamma^s$ of measurable subsets $\gamma^s$ of $\Delta$ and a family
$\Gamma^u$ of measurable subsets $\gamma^u$ of $\Delta$
such that
\begin{itemize}[leftmargin=*]
\item there exist a unique $\gamma^s(x)\in\Gamma^s$ and a unique $\gamma^u(x)\in\Gamma^u$ both containing $x$;
\item for every $x,y\in\Delta$, $\hat s(x,y)=\infty$ if and only if $\gamma^s(x)=\gamma^s(y)$;
\item for every $x,y,z\in \Delta$ such that $\gamma^s(x)=\gamma^s(y)$, $\hat s(x,z)=\hat s(y,z)$;
\item For all $x,y\in \Delta$ such that $\hat s(x,y)>n$, $\gamma^s(x)=\gamma^s(y)$ if and only if $\gamma^s(F^n(x))=\gamma^s(F^n(y))$ and $\gamma^u(x)=\gamma^u(y)$ if and only if $\gamma^u(F^n(x))=\gamma^u(F^n(y))$. 
\end{itemize}
\vspace{10pt}
\underline{For 
an SFT}: $\gamma^s(\vec x
)=\{\vec y
\, :\ \forall m\ge 0,\ y_{m}=x_{m}\}$ and $\gamma^u(\vec x
)=\{\vec y
\, :\ \forall m\ge 0,\ y_{-m}=x_{-m}\}$).\\
\underline{For a 
Young tower}: Using the notations $\Lambda$, $\gamma^s(x)$ and $\gamma^u(x)$ of \cite[Section 1.1]{Y}, $\gamma^s((x,l))=(\Lambda\cap\gamma^s(x))\times\{l\}$ and $\gamma^u((x,l))=\{z\in\Lambda\cap\gamma^u(x)\, :\, \hat s(x,z)>l\}\times\{l\}$ .

\item {\bf Product structure.}
We assume that there exists an at most numerable partition of $\Delta$ in subsets of the form $\Delta^{(0)}_i:=\{y\in\Delta\, :\, \hat s(x^{(i)},y)\ge 1\}$ and that each $\Delta^{(0)}_i$ has a product structure of the following form: for every $x,y\in\Delta^{(0)}_i$, $\gamma^s(x)\cap\gamma^u(y)$ contains a single point.\\

\noindent
\underline{For an SFT}: The partition $\{\Delta^{(0)}_i\}$ corresponds to the partition in $0$-cylinders $\{\vec y 
\, :\, y_0=i\}$, $x^{(i)}$ being a fixed element of $\Delta^{(0)}_i$.\\
\underline{For a 
Young tower}: Consider the partition of $\Delta$ to $\Delta_{l,j}$'s.

\item {\bf Quotient system.}
We define $\bar\Delta:=\bigcup_i\gamma^u(x^{(i)})$ and $\bar{\mathfrak p}:\Delta\rightarrow\bar\Delta$ to be the projection along the $\gamma^s$, that is, for any $x\in\Delta^{(0)}_i$, $\bar{\mathfrak p}(x)$ is the intersection point of $\gamma^s(x)$
with $\gamma^u(x^{(i)})$. We define $\bar F:\bar\Delta\rightarrow\bar\Delta$ such that $\bar F\circ\bar{\mathfrak p}=\bar{\mathfrak p}\circ F$ and $\bar\nu=\bar{\mathfrak p}_*\nu$. This ensures that 
$(\bar F,\bar \Delta,\bar\nu)$ is a factor of $(F,\Delta,\nu)$
by $\bar{\mathfrak p}$. Also note that $\hat s$ is preserved under composition by $\bar{\mathfrak p}$.
\vspace{10pt}

\noindent 
\underline{For 
an SFT}: Define $\bar{\mathfrak p}$ on $\Delta^{(0)}_i$ by setting $\bar{\mathfrak p}(
\vec y
)=\vec z
$ with $z_n=y_n$ and $z_{-n}=x^{(i)}_{-n}$ for all $n\ge 0$, so that $(\bar F,\bar\Delta,\bar\nu)$ is the one-sided subshift associated to $(F,\Delta,\nu)$ (up to identifying $\bar{\mathfrak p}(\vec y
)=\vec z
$ with $(z_n)_{n\ge 0}=(y_n)_{n\ge 0}$).\\
\underline{For a 
Young tower}: $(\bar F,\bar\Delta,\bar\nu)$ corresponds to the quotient (expanding) Young tower.

\item {\bf Spaces of smooth functions.}
Let $\beta\in(0,1)$. As in \cite[Section 3.4]{DNP}, we define $\mathcal B^{(0)}_\beta$ as the space of functions $\psi:\Delta\rightarrow\mathbb C$ such that
the following quantity is finite
\begin{equation}\label{tildeB}
\Vert \psi\Vert^{(0)}_{\beta}:=\Vert \psi\Vert_\infty+\sup_{\gamma^u;\ x,y\in\gamma^u}
\frac{|\psi(x)-\psi(y)|}{\beta^{\hat s(x,y)}}+\sup_{n\ge 0,\ \gamma^s;\ x,y\in\gamma^s}
\frac{|\psi(F^n(x))-\psi(F^n(y))|}{\beta^{n}}\, .
\end{equation}
Due to the product structure, this space corresponds to the set of functions which are Lipschitz continuous with respect to the metric $(x,y)\mapsto\beta^{\hat s_0(x,y)}$,
with 
\begin{equation}\label{hats0}
\hat s_0(x,y):=\inf\{n\ge 0\, :\, \forall x'\in F^{-n}(x),\ \forall y'\in F^{-n}(y),\ \hat s(x',y')\le 2n\}
\end{equation}
In the case of SFTs, $\hat{s}_0$ reduces to $\hat s_0(\vec x,\vec y)=\inf\{m\ge 0\, :\, (x_m,x_{-m})\ne(y_m,y_{-m})\}$.
\vspace{5pt}

Denote by $\mathcal B_\beta$ the space of functions
$\bar\psi:\bar\Delta\rightarrow \mathbb C$ that are bounded and Lipschitz continuous with respect to the metric $(x,y)\mapsto \beta^{\hat s(x,y)}$.
\end{itemize}
\vspace{10pt}

\noindent
{\bf Assumption $(H1)$:}\\
Assume moreover that any point $x\in\bar\Delta$ has a number at most numerable of preimages by $F$ and that the transfer operator $\mathcal L$ associated to
$(\bar\Delta,\bar F,\bar\nu)$ has the following form
\begin{equation}\label{formuleOp}
\exists \bar g:\bar\Delta\rightarrow\mathbb R,\ \bar C_{\bar 
g}:=\sup_{x,y:\, \hat s(x,y)>1}\frac{|e^{\bar 
g(y)-\bar 
g(x)}-1|}{\sqrt{\beta}^{\hat s(x,y)}}<\infty\quad\mbox{and}
\quad \mathcal L\psi(x)=\sum_{z\in \bar F^{-1}(x)}e^{\bar g(z)}\psi(z)\, 
\end{equation}
and that for all $x,y\in\bar\Delta$ such that $\hat s(x,y)>0$,
\begin{equation}\label{bijpreimages}
\exists W_{x,y}:\bar F^{-1}(\{x\})\rightarrow \bar F^{-1}(\{y\})\,\,\text{injective such that}\,\,
\forall z\in \bar F^{-1}(\{x\}),\ \hat s(z,W_{x,y}(z))> 1.
\end{equation}

\noindent
\underline{For 
an SFT}:
 $\bar F^{-1}(\{\vec{x}\})=\{(z_0,x_0,x_1,\dots),\, z_0\in \mathcal A,\ A(z_0,x_0)=1\}$ and $W_{x,y}(z_0,x_0,x_1,\dots)=W_{x,y}(z_0,y_0,y_1,\dots)$.\\
\underline{For a Young tower}: if $x=(x_0,l)$ with $l\ge 1$, then $z=(x_0,l-1)$ and $W(x,y)(z)=(y_0,l-1)$; if $x=(x_0,0)$, then $\bar F$ defines a bijection between $\Lambda_i\times\{r_i-1\}$ and $\Lambda$ and given $z=(z_0,r_i-1)\in \Lambda_i\times\{r_i-1\}$ such that $\bar F(z)=x$, then $W_{x,y}(z)$ is the only element of $\Lambda_i\times\{r_i-1\}$ such that $\bar F(W_{x,y}(z))=y$.
\vspace{10pt}

The following result describes how $(H0) \implies (A)(1)$ and $(H0),(H1) \implies (A)(2)$. 

\begin{lem}\label{LEM0}
Assume $(H0)$.
Let $\phi,\psi,\xi:\Delta\rightarrow\mathbb R$ be three functions such that $\phi\in \mathcal B^{(0)}_\beta$, $\psi,\xi\in\mathcal B^{(0)}_{\sqrt{\beta}}$. Then
\begin{itemize}[leftmargin=15pt]
\item Assumption $(A)(1)$ holds true for $\mathfrak p=id$, more precisely,
there exists $\bar\phi\in\mathcal B_{\sqrt{\beta}}$ and $\chi\in\mathcal B^{(0)}_{\sqrt{\beta}}$
such that $\phi=\bar\phi\circ\bar{\mathfrak p}+\chi-\chi\circ F$.
\item If, in addition, $(H1)$ is true, then Assumption $(A)[r](2)$ holds true  for any integer $r\ge 0$, $p_0=\infty$, $\cX_a=\cX_a^{(+)}=\mathcal{B}_{\sqrt\beta} \hookrightarrow L^1(\bar{\nu})$, $r_0=r$, $q(H)=0$, $\mathfrak p=id$ and $\vartheta=\sqrt{\beta}$. 
\end{itemize}
\end{lem}

\begin{proof} Following \cite[(3.27)]{DNP}, we prove Assumption $(A)(1)$ with $$\chi:=\sum_{k\ge 0}\left(\phi\circ F^k-\phi\circ F^k \circ \bar{\mathfrak p}\right)\in\mathcal B^{(0)}_{\sqrt{\beta}}$$ and 
$$\bar\phi:=\left(\phi+\sum_{k\ge 1}(\phi\circ  F^k-\phi\circ F^{k-1}\circ\bar F)\right)\bigg|_{\bar \Delta}\in\mathcal B_{\sqrt{\beta}}.$$

First observe that since $\gamma^s(x)=\gamma^s(\bar{\mathfrak p}(x))$, $|\phi\circ F^k-\phi\circ F^k\circ\bar{\mathfrak p}|\le\Vert\phi\Vert^{(0)}_\beta\beta^k $. Analogously, since $\gamma^s(F(x))=\gamma^s(\bar F(x))$, $|\phi\circ F^k-\phi\circ F^{k-1}\circ\bar F|\le\Vert\phi\Vert^{(0)}_\beta\beta^{k-1}$. This ensures that $\chi$ and $\bar\phi$ are well defined and we have proved the identity
\[
\phi-\chi+\chi\circ F=\phi\circ\bar{\mathfrak p}+\sum_{k\ge 1}(\phi\circ  F^k\circ\bar{\mathfrak p}-\phi\circ F^{k-1}\circ\bar{\mathfrak p}\circ F)=\bar\phi\circ\bar{\mathfrak p}\, ,
\]
since $\bar{\mathfrak p}\circ F=\bar F\circ\bar{\mathfrak p}$.

Let us prove that $\bar\phi$ is in $\mathcal B_{\sqrt{\beta}}$.
Since $\phi\in\mathcal B^{(0)}_{\beta}$, 
for any $x,y\in\gamma^u\subset\bar\Delta$,
\begin{align*}
\left|\bar\phi(x)-\bar\phi(y)\right|&\le \left|\phi(x)-\phi(y)\right|
+2\sum_{k\ge\lceil\hat s(x,y)/2\rceil+1}\Vert
\phi\circ F^k-\phi\circ F^{k-1}\circ \bar {\mathfrak p}\circ F\Vert_\infty\\
&\ \ \ \ \ +\sum_{k=1}^{\lceil\hat s(x,y)/2\rceil}
\left|  \phi(F^k(x))-\phi(F^k(y))-[\phi( F^{k-1}( \bar  {\mathfrak p}( F(x))))-\phi( F^{k-1}( \bar {\mathfrak p}(F(y)))]\right|\\
&\le \Vert\phi\Vert^{(0)}_\beta \left(\beta^{\hat s(x,y)}+2\sum_{k\ge\lceil\hat s(x,y)/2\rceil+1}\beta^{k-1}+2\sum_{k=1}^{\lceil\hat s(x,y)/2\rceil}\beta^{\hat s(x,y)-k} \right)\\
&\le  \Vert\phi\Vert^{(0)}_\beta \beta^{\frac{\hat s(x,y) }2}\left(1+4 \frac{\beta^{-1}}{1-\beta}\right)\, ,
\end{align*}
since, whenever $1\le k<\hat s(x,y)$, $\hat s(F^k(x),F^k(y))=\hat s(x,y)-k$ 
and 
\begin{align*}
\hat s(x,y)=\hat s(F(x),F(y))+1 &=\hat s(\bar  {\mathfrak p}(F(x)),\bar  {\mathfrak p}(F(y)))+1\\ &=
\hat s(F^{k-1}(\bar {\mathfrak p}(F(x))),F^{k-1}(\bar  {\mathfrak p}(F(x))))+k\, .
\end{align*}
Now, let us prove that $\chi$ belongs to $\mathcal B^{(0)}_{\sqrt{\beta}}$.
Let $x,y\in\gamma^u\subset\Delta$, then
\begin{align*}
|\chi(x)-\chi(y)|&\le \sum_{k=0}^{\lceil\hat s(x,y)/2\rceil}\left|(\phi\circ F^k(x)-\phi\circ F^k(y))-(\phi\circ F^k \circ \bar{\mathfrak p}(x)-\phi\circ F^k \circ \bar{\mathfrak p}(y))\right|\\
&\ \ \ \ \ \ \ \ \ \ \ \ \ \ \ +2\sum_{k\ge \lceil\hat s(x,y)/2\rceil+1}\left\Vert\phi\circ F^k-\phi\circ F^k \circ \bar{\mathfrak p}\right\Vert_\infty\\
&\le 2\Vert\phi\Vert^{(0)}_\beta \left(\sum_{k=0}^{\lceil\hat s(x,y)/2\rceil}\beta^{\hat s(x,y)-k}
+\sum_{k\ge \lceil\hat s(x,y)/2\rceil+1}\beta^k\right)=O(\beta^{\frac{\hat s(x,y)}2})\, ,
\end{align*}
and, for any $x,y\in\gamma^s\subset\Delta$ and any integer $n\ge 0$, then $\bar {\mathfrak p}(F^n(x))=\bar {\mathfrak p}(F^n(y))$ and so
\begin{align*}
|\chi(F^n(x))-\chi(F^n(y))|&\le \sum_{k\ge 0}
\left|\phi\circ F^k(F^n(x))-\phi\circ F^k(F^n(y))\right|\\
&\le\Vert \phi\Vert^{(0)}_\beta
  \sum_{k\ge 0}   \beta^{k+n}=\frac{\Vert \phi\Vert^{(0)}_\beta\beta^n}{1-\beta}\, .
\end{align*}

Let us establish (A)[$\infty$](2). Let $H\in\{\psi,\xi\}\subset\mathcal B^{(0)}_{\sqrt{\beta}}$. Using the definition of $\bar\phi$ and $\chi$ as above, we define
\[
h_{k,s, H}:=H\circ F^ke^{is\chi \circ F^k }  e^{-is\bar{S}_k\circ \bar{\mathfrak p} },\quad\mbox{with}\quad
\bar S_n:=\sum_{k=0}^{n-1}\bar\phi\circ F^k\, ,
\]
and
\[
\quad\bar{h}_{k,s, H}(x)= e^{-is\bar S_k(x)}\EXP_{\nu}[H\circ F^k  e^{is\chi \circ F^k }\, |\,\hat s(\cdot , x) >2k],\quad x\in\bar\Delta\, .
\]

Observe that for any $p\ge 1$, $\Vert  \bar{h}
_{k,s, H}\Vert_{L^p(\bar\nu)}\le\Vert  h
_{k,s, H}\Vert_{L^p(\nu)}=\Vert  H\Vert_{L^p(\mu)}$ and so \eqref{hk0} and for all $x\in\bar\Delta$,
\begin{align*}
\bar{h}^{(j)}_{k,s, H}(x)&
=\mathbb E_\nu\left[\left.  h^{(j)}_{k,s, H}\right|\,\hat s(\cdot , x) >2k\right]\\
&= e^{-is\bar S_k(x)}\EXP_{\nu}\big[H\circ F^k\, \cdot \big(i(\chi \circ  F^k-\bar S_k(x))\big)^je^{is\chi \circ F^k }\, \big|\,\hat s(\cdot , x) >2k\big]\, .
\end{align*}
and so
\begin{align*}
\Vert\bar{h}^{(j)}_{k,s, H}&\circ\bar{\mathfrak p}-h^{(j)}_{k,s, H}\Vert_\infty \\
&\leq  
\sum_{m=0}^{j}{j\choose m}\Vert \bar S_k\Vert_\infty^{j-m}
 \sup_{x,y\in\Delta,\ \hat s(x,y)>2k}\left|(H\chi^{m}e^{is\chi})(F^k(x))-(H\chi^{m}e^{is\chi})(F^k(y))\right| \\
&\leq  
\sum_{m=0}^{j}{j\choose m}(k\Vert \bar \phi\Vert_\infty)^{j-m}
 \Vert H\chi^{m}e^{is\chi}\Vert^{(0)}_{\sqrt{\beta}}\sqrt{\beta}^k \, ,
\end{align*}
since $\hat s(x,y)>2k$ implies that $\hat s_0(F^k(x),F^k(y))>k$.
But $H,\chi\in \mathcal B^{(0)}_{\sqrt{\beta}}$ and $$\Vert e^{is\chi}\Vert^{(0)}_{\sqrt{\beta}}\le 1+\Vert s\chi\Vert^{(0)}_{\sqrt{\beta}}\, .$$
Therefore, there exists $C>0$ such that 
$$\left\Vert h^{(j)}_{k,s, H}-\bar h^{(j)}_{k,s, H}\circ\bar{\mathfrak p}\right\Vert_\infty \leq C(1+|s|)\, k^j \, \sqrt{\beta}^{k}\, ,$$
and we have proved \eqref{hk1}.
Next, we note that
\begin{equation}\label{h_Bound}
\|\bar h^{(j)}_{k,s, H}\|_\infty \leq \Vert H\Vert_\infty(\|\bar{S}_k\|_\infty+\|\chi\|_\infty)^{j} \leq \Vert H\Vert_\infty(\|\chi\|_\infty+k\|\bar\phi\|_\infty)^{j}\, ,
\end{equation}
which leads to the second part of \eqref{hk2} since $\mathcal B_{\sqrt{\beta}}
\hookrightarrow L^1(\bar\nu)$.

It remains to prove the first part of \eqref{hk2}.
Recalling that $\mathcal L^{2k}_{is}=\mathcal L^{2k}(e^{is\bar S_{2k}}\,\cdot\,)$ and
using \eqref{formuleOp}, we obtain that for every $x\in\bar\Delta$,
\[
\cL^{2k}_{is}\bar{h}_{k,s, H}(x)=\sum_{y\in F^{-2k}(\{x\})} e^{S_{2k}^{\bar{g}}(y)}e^{is\bar{S}_{k}\circ \bar F^k(y)}\EXP_{\nu}[H\circ F^k e^{is\chi \circ F^k }\, |\,\hat s( \cdot , y) >2k]\, ,
\]
with $S_{2k}^{\bar{g}}:=\sum_{m=0}^{2k-1}\bar g\circ \bar F^{m}$.
Therefore,
\begin{align*}
(\cL^{2k}_{is}&\bar{h}_{k,s, H})^{(j)}(x)\\&=\sum_{y\in F^{-2k}(\{x\})} e^{S^{(\bar g)}_{2k}(y)}\EXP_{\nu}[i^j(\bar{S}_{k}\circ\bar F^k(y)+\chi \circ \bar F^k)^j H\circ F^ke^{is(\bar{S}_{k}\circ\bar F^k(y)\chi \circ \bar F^k) }\, |\,\hat s( \cdot , y) >2k]\\ &=\cL^{2k}(\tilde{h}^{(j)}_{k,s, H})(x),\quad\mbox{with}\quad
\tilde{h}_{k,s, H}(x):= e^{is\bar S_k\circ \bar F^k(x)}\EXP_{\nu}[H\circ F^k e^{is\chi \circ F^k }\, |\,\hat s( \cdot , y) >2k]\, .
\end{align*} 
From this, we have the following estimate
\begin{align*}
\left\Vert (\cL^{2k}_{is}\bar{h}_{k,s, H})^{(j)}\right\Vert_\infty\leq \Vert H\Vert_\infty(\|\bar S_{k}\|_\infty+\|\chi\|_\infty)^j 
\Vert \cL^{2k}{{\bf 1}_{\bar \Delta}}\Vert_\infty 
\leq  \Vert H\Vert_\infty (\|\chi\|_\infty+k\|\bar{\phi}\|_\infty)^{j}\, .
\end{align*}

Finally, we need to estimate the Lipschitz constant of $(\cL^{2k}_{is}\bar{h}_{k,s,H})^{(j)}$. Due to \eqref{bijpreimages}, for any $x,y\in\bar\Delta$ such that
$\hat s(x,y)>1$, there exists a bijection $L_{x,y,r} : \bar F^{-r}(\{x\}) \to  \bar F^{-r}(\{y\}) $ such that, for all $z\in\bar F^{-r}(\{x\})$, $\hat s(z,L_{x,y,r}(z))> r$,
with $L_{x,y,r}$ being defined inductively on $r\ge 1$ by $L_{x,y,1}=W_{x,y}$ and $L_{x,y,r+1}(z)= W_{\bar F(z),L_{x,y,r}(\bar F(z))}(z)$.

Then it is immediate that, for all $z \in  \bar F^{-r}(\{x\})$,
$$\gamma_{k,s, H}(z):=\EXP_{\nu}[H\circ F^k \cdot e^{is\chi \circ F^k }\, |\,\hat s(\cdot , z) >2k] 
$$
remains unchanged if $z$ is replaced by $L_{x,y,2k}(z)$.
Therefore, writing 
$$\psi_{k,s}(z)=e^{is\bar{S}_{k}(z)},$$
we have 
\begin{align}\label{diffLh}
(\cL^{2k}_{is}&\bar{h}_{k,s,H})^{(j)}(x)-(\cL^{2k}_{is}\bar{h}_{k,s,H})^{(j)}(y) \\ &= \cL^{2k}(\tilde{h}^{(j)}_{k,s,H})(x)-\cL^{2k}(\tilde{h}^{(j)}_{k,s,H})(y) \nonumber \\ &=\sum_{r=0}^j {j\choose r}\sum_{z\in \bar F^{-2k}(\{x\})} \Big(e^{S_{2k}^{(\bar{g})}(z)}\psi_{k,s}^{(r)}\circ \bar F^k(z)-e^{S_{2k}^{(\bar{g})}(L_{x,y,2k}(z))}\psi^{(r)}_{k,s}\circ \bar F^k(L_{x,y,2k}(z))\Big)\gamma_{k,s,H}^{(j-r)}(z)\, . \nonumber
\end{align}
Next, we estimate each term in the above sum. Observe that
\begin{equation}\label{majogamma}
\|\gamma_{k,s,  H}^{(j-r)}\|_\infty\leq \Vert H\Vert_\infty (1+\|\chi\|_\infty
)^{j-r}\, .
\end{equation}
Also, 
\begin{align}
\nonumber\Big|e^{S_{2k}^{(\bar{g})}(z)}\psi_{k,s}^{(r)}\circ \bar F^k(z)&-e^{S_{2k}^{(\bar{g})}(L_{x,y,2k}(z))}\psi^{(r)}_{k,s}\circ \bar F^k(L_{x,y,2k}(z))\Big|  \\ 
\nonumber&\leq e^{S_{2k}^{(\bar{g})}(z)}|1-e^{S_{2k}^{(\bar{g})}(L_{x,y,2k}(z))-S_{2k}^{(\bar{g})}(z)}||\psi_{k,s}^{(r)}\circ \bar F^k(z)| \\ &\phantom{\leq e^{S_{2k}^{(\bar{g})}(z)}|1-e^{S_{2k}^{(\bar{g})}}}+ e^{S_{2k}^{(\bar{g})}(L_{x,y,2k}(z))}|\psi_{k,s}^{(r)}\circ \bar F^k(z)-\psi^{(r)}_{k,s}\circ \bar F^k(L_{x,y,2k}(z))|\, .\label{diffLhbis}
\end{align}

Note that the first part of \eqref{formuleOp} implies that
\[\hat s(x,y)>1\quad\Rightarrow\quad  e^{-\bar C_{\bar g}\sqrt{\beta}^{\hat s(x,y)}}\leq (1+\bar C_{\bar g}\sqrt{\beta}^{\hat s(x,y)})^{-1}\leq e^{\bar g(y)-\bar g(x)}\leq 1+\bar C_{\bar g}\sqrt{\beta}^{\hat s(x,y)}\leq e^{\bar C_{\bar g}\sqrt{\beta}^{\hat s(x,y)}}\]
and so
\[
|1-e^{S_{2k}^{(\bar{g})}(L_{x,y,2k}(z))-S_{2k}^{(\bar{g})}(z)}|\leq 
e^{\bar C_{\bar g}\sum_{m=0}^{2k-1}\sqrt{\beta}^{\hat s(x,y)+2k-m}}-1
\leq \max(1,e^{\frac{\bar C_{\bar g}}{1-\sqrt{\beta}}})\frac{
\bar C_{\bar g}
\beta^{(\hat s(x,y)+1)/2}}{1-\sqrt\beta}\, .
\]
Moreover,
\begin{align*}
|\psi_{k,s}^{(r)}\circ \bar F^k(z)-\psi^{(r)}_{k,s}\circ &\bar F^k(L_{x,y,2k}(z))| \\ &\leq \big|(\bar{S}_{k}(\bar F^k(z)))^r e^{is\bar{S}_{k}(\bar F^k(z))}  - (\bar{S}_{k}(\bar F^k (L_{x,y,2k}(z))))^r e^{is\bar{S}_{k}((\bar F^k(L_{x,y,2k}(z)))}\big| \\ 
&\leq
\sum_{m=0}^{k-1}\left|\bar\phi(\bar F^{k+m}(L_{x,y,2k}(z)))-\bar\phi(\bar F^{k+m}(z)) \right|\, |(r\|\bar S_k\|^{r-1}_\infty+ |s|\|\bar S_k\|^r_\infty) \\ 
&\leq \frac{|\bar\phi|_{\wt\cB_{\sqrt{\beta}}}\beta^{(\hat s(x,y)+1)/2}}{1-\sqrt\beta} (r(k\|\bar\phi\|_\infty)^{r-1}_\infty+ |s|(k\|\bar\phi\|_\infty)^r)\, . 
\end{align*}
Combining this with \eqref{diffLh},  \eqref{majogamma} and \eqref{diffLhbis}, we obtain 
\begin{align*}
\frac{|(\cL^{2k}_{is}\bar{h}_{k,s})^{(j)}(x)-(\cL^{2k}_{is}\bar{h}_{k,s})^{(j)}(y)|}{\sqrt{\beta}^{\hat s(x,y)}} = \cO((1+|s|)k^j)\, ,
\end{align*}
as required. This establishes 
$(A)[r](2)$ for all $r$. 
\end{proof}

\begin{rem}
In the case of SFTs$,$ $\EXP_{\nu}[g\circ F^k|\hat s(\cdot,x)>2k]=\EXP_\nu[g|C_{-k,...,k}(F^k(x))]$ where $C_{-k,...,k}(z)$ is the two-sided cylinder $\{\vec y\, :\, \forall m=-k,...,k,\ y_m=z_m\}$.

\end{rem}
\subsection{Assumptions $(B)$ and $(A)(3)$}\label{A_B}
Assumptions $(A)(3)$ and $(B)$ describe the spectrum of the transfer operator and its perturbations, and is typical for the Nagaev-Guivarc'h approach to establishing limit theorems for dynamical systems and Markov processes. However, Assumption $(B)(2)$ gives more control than the conventional approach allowing us to prove the first order strong Edgeworth expansion
as well as expansions in the local limit theorem. In this section, we describe general techniques to verify these assumptions for dynamical systems. 

For Assumption $(B)$, the standard tool available is a 
Doeblin-Fortet
inequality. This idea is summarized in the following proposition. 

\begin{prop}[\cite{ITM,Hennion}]\label{DF} Suppose $\cP$ is a bounded linear operator on a complex Banach space $(\cB,\Vert\cdot\Vert_{\mathcal B})$ and that $\Vert\cdot\Vert_{*}$ is a norm on $\mathcal B$ such that
\begin{itemize} 
\item[$(i)$] $\cB-$bounded sets are precompact in $(\cB,\Vert\cdot\Vert_{*})$,
\item[$(ii)$] There exist 
positive constants $A,B,\theta$
such that $0<\theta<1$ and 
\begin{align}
\label{DF_0} 
 \forall h \in \cB,\ \forall n, \quad \|\cP^{n} h\|_{\cB} &\leq A\theta^n\| h \|_{\cB}+ B\| h \|_{*}.
\end{align}
\end{itemize}
Then 
the essential spectral radius of $\cP: \cB \to \cB$ satisfies $r_{\text{ess}}(\cP) \leq \theta<1$.

\end{prop}

Note that \eqref{DF_0} holds true as soon as there exists a positive integer $n_0$ such that $\|\cP^{n_0} h\|_{\cB} \leq \theta^{n_0}\| h \|_{\cB}+ B_0\| h \|_{*}$ and if $\sup_{n\ge 1}\|\cP^n\|_*<\infty$. 

In our setting, we will take $\cP=\cL_{is}$. Recall that these define bounded operators on $L^1(\bar\nu)$ and that if ${\mathbf 1}_{\Delta} \in\cB$, then it is an eigenvector of $\cL$ associated to the eigenvalue 1 and if the hypothesis of \Cref{DF} is satisfied by $\cP=\cL_{is}$ then
\begin{enumerate}
\item The spectrum $\sigma(\cL_{is})$ of $\cL_{is}:\cB \to \cB$, is contained in the closed unit disc, and is the union of the essential spectrum, $\sigma_{\text{ess}}(\cL_{is})$, and  finitely many eigenvalues of finite multiplicity $\{\lambda_{s,1},\dots,\lambda_{s, k_s}\}$ with $\theta_s<|\lambda_{s,j}|\leq 1$ (quasi-compactness).
\item  If $\cL_{is}$ has eigenvalues of modulus $1$, they are semi-simple, i.e., no Jordan blocks (applying Proposition~\ref{DF}, this comes from the fact that $(i)$ implies that $\Vert\cdot\Vert_*$ is dominated by $\Vert\cdot\Vert_{\mathcal B}$ and so by $(ii)$ $\sup_{n\ge 1}\Vert \mathcal P^n\Vert_{\cB}<\infty$).
\end{enumerate}

To conclude $(B)(1,2)$, we need to understand the eigenvalues of modulus $1$ of $\cL_{is}$.
From the positivity of $\cL$, it follows that all its eigenvalues of modulus $1$ are roots of unity, and hence, the corresponding eigenfunctions yield invariant densities for $\bar F^n$. Since $\bar\nu$ is $\bar{F}-$invariant, $\cL({\bf 1}_{\bar{\Delta}})={\bf 1}_{\bar{\Delta}}$ and $\cL^*(\bar\nu)=\bar\nu$ where $\cL^*$ is the adjoint of $\cL$, $1$ is an eigenvalue of $\cL$. It follows that $1$ is simple if and only if $\bar\nu$ is ergodic. This is because ergodicity is equivalent to $\bar F-$invariant functions being constants. Also $\cL$ not having eigenvalues other than $1$ on the unit circle is equivalent to exactness of the transformation $\bar{F}$ which we have to establish through dynamical arguments.

Recall also that $\cL$ can often be written, by a change of variable, under the form 
\begin{equation}\label{TOp}
\cL h (x) = \sum_{z \in \bar{F}^{-1}(\{x\})} e^{\bar g(z)}{h(z)},
\end{equation}
where $e^{-\bar g}=J\bar{F}$ is the Jacobian of $\bar F$ with respect to $\bar \nu$.

\begin{defin}
The function $\phi:\cM\rightarrow\mathbb X$ is said to be {\bf non-arithmetic} if it is not $f-$cohomologous in $L^2(\cM,\mu)$ to a sublattice-valued
function, i.e. if there exists no triple $(a,B,\theta)$  with  $a\in\mathbb X$, $B$ a closed proper subgroup of $\mathbb X$ and $\theta \in L^2(\mu)$ such that  $\phi+\gamma-\gamma\circ f
 \in  \theta+B $  $\mu$-a.s.. 
\end{defin}

\begin{lem}\label{periphericalspect}
Suppose Assumptions $(0)$ with $\phi:\cM\rightarrow\mathbb X$ non-arithmetic and that the conclusion of Proposition~\ref{DF} holds true
for $\cL_{is}$ for all $s\in\mathbb X^*$ and that the spectral radius of $\cL_{is}$ on $\cB$ is dominated by 1.
Further assume that $(\bar F,\bar \nu)$ is exact,  \eqref{TOp} holds $\bar\nu$-a.s.~with $F^{-1}(\{x\})$ at most numerable and that Assumption $(A)(1)$ is true. 
Then
Assumption $(B)(1,2)$ is true.
If, in addition, $\bar\phi\in\cB\hookrightarrow L^2(\bar\nu)$, then 
the entirety of 
Assumption $(B)$ is true.
\end{lem}

\begin{proof}
Due to the conclusion of Proposition~\ref{DF},
we know that $\cL$
is quasicompact. By exactness of the system, we already know that $1$ is the only eigenvalue of $\cL$ and that it is simple. Thus
\[
\exists c>0,\quad \sum_{n\ge 0}\Vert \cL^n(\bar\phi)\Vert_2<c \sum_{n\ge 0}\Vert \cL^n(\bar\phi)\Vert_{\cB}<\infty\, ,
\]
since $\Vert\cL^n(\bar\phi)\Vert_{\cB}$ decreases exponentially fast as $n\rightarrow+\infty$.

Let $s\in\mathbb X^*\setminus\{0\}$. Let us show that $\cL_{is}$ does not have eigenvalues on the unit circle. To see this, assume that $e^{i\theta}$ is an eigenvalue of $\cL_{is}$ with $h\in \cB$ as a corresponding eigenfunction.  Then,  ${\cL}_{is}h=e^{i\theta}h$ in $L^1(\bar\nu)$. Observe that, for $\bar\nu$-almost every $x$, the following holds true
\begin{align*}
\cL|h|(x)=\sum_{z\in\bar F^{-1}(\{x\})} e^{\bar g(z)}|h(z)|\geq \bigg|\sum_{z\in\bar F^{-1}(\{x\})}{e^{\bar g(z)+i s \bar \phi(z)}h(z)} \bigg|=|\cL_{is}h(x)| =|e^{i\theta}h(x)| = |h(x)|\, ,
\end{align*}
and thus $\cL^n|h|\geq |h|,\,\, \bar\nu $-a.s. for all $n$. But 
$ \lim_{n \to \infty} (\cL^n|h|) = \int_{\bar\Delta} |h| \, d\bar\nu $ in $L^1(\bar\nu)$ because $\bar F$ is exact. 
So $ \int_{\bar\Delta} |h|\, d\bar{\nu} \geq |h| $.
This implies that $|h|$ is constant $\bar\nu$-a.e. Without loss of generality, we assume $|h|\equiv 1$. So $h =e^{i\gamma(\cdot)}$ for some $\gamma : \bar \Delta \to \reals$. Substituting back 
$$ \cL_{is} h(x) = \sum_{z\in\bar F^{-1}(\{x\})}e^{\bar g(z)+i(s\bar \phi(z)+\gamma(z))} = e^{i (\theta+\gamma(x))}\quad\mbox{for}\ \bar\nu\mbox{-a.e.}\, x\in\bar\Delta\, ,
$$
and so
$$\sum_{z\in\bar F^{-1}(\{x\})} e^{\bar g(z)+i(s\bar\phi(z)+\gamma(z)-\gamma(\bar F(z))-\theta)} = 1  \quad\mbox{for}\ \bar\nu\mbox{-a.e.}\ x\in\bar\Delta\, .$$
Since, $ \cL(\mathbf 1_{\bar{\Delta}})(x) = \sum_{z\in\bar F^{-1}(\{x\})}e^{\bar g(z)}  = 1$ for $\bar\nu$-a.e. $x$ and $e^{i s\bar\phi(z)+\gamma(z)-\gamma(x)-\theta)}$ are unit vectors and thus we conclude that
\begin{equation}\label{LatticeVal}
s\bar\phi+\gamma-\gamma\circ \bar{F}+\theta = 0\mod 2\pi\quad \bar\nu-a.e.
\end{equation}
which, combined with Assumption $(A)(1)$, contradicts the non-arithmeticity of $\phi$.
\end{proof}


Now let us focus on Assumption $(A)(3)$. The following lemma applies typically (but not only) when $\cB$ is a Banach algebra
\begin{lem}\label{lemA_B}
Suppose Assumption $(0)$ is true with $(\bar F,\bar \nu)$ exact. Assume that $\cL$ defines a continuous operator on $\cB$ and that 
the multiplication by $\bar\phi$ defines a continuous linear map on $\mathcal B$ and that $s\mapsto e^{is\bar\phi}\times\,\cdot\, $ defines a $C^1$ map from
$\mathbb R$ to $\mathcal L(\cB)$ with derivative $s\mapsto (i\bar\phi)e^{is\bar\phi}\times\,\cdot\, $. Then Assumption $(A)(3)$ holds for any $r\ge 0$.
\end{lem}
\begin{proof}
Let $s\in\mathbb R$ and $u\in\mathbb R\setminus\{0\}$. Observe that
\[
\frac 1u\Vert \cL_{i(s+u)}-\cL_{is}-\cL_{is}(i\bar\phi\,\,\cdot\,)u\Vert_{\cL(\cB)}=\left\Vert \cL\left(e^{is\bar\phi}\frac{e^{iu\bar\phi}-1-i\bar\phi u}u\,\, \cdot\,\right)\right\Vert_{\cL(\cB)}
\] goes to $0$ as $u\rightarrow 0$. Thus $s\mapsto \cL_{is}$ is $C^1$ from $\mathbb R$ to $\cL(\cB)$ with derivative $\cL_{is}(i\bar\phi\,\,\cdot\,)$ and we conclude by induction.
\end{proof}

We will see in \Cref{SFT} and \Cref{Towers} how these results can be applied in the case of subshifts of finite type and Young towers.

\subsection{Assumptions $(C)$ and $(D)$}\label{C_D}

We assume from now on that $\mathbb X=\mathbb R$.
Let us start by some comments on our
Assumptions $(C)$ and $(D)$. We first observe that both of them
deal with $\left\Vert \mathcal L_{is}^n\right\Vert_{\mathcal B_1 \rightarrow \cB_2}$
for some $\cB_1 \hookrightarrow \cB_2\hookrightarrow L^1(\bar\nu)$. 

Recall that  Assumption $(C)(1)$
is true if there exists $K>0$ such that
$$\exists \alpha\ge 0,\alpha_1>0,\hat\delta>0,\ \forall|s|>K, \quad \forall n\ge n_1,\quad \left\Vert \mathcal L_{is}^n\right\Vert_{\mathcal B_1 \rightarrow \cB_2}\le C |s|^\alpha e^{-n^{\alpha_1}\hat\delta |s|^{-\alpha}} $$
and that $(D)(1)$ holds true if for all $B>0$, there exists $K>0$ and $d_1\in [0,1]$ such that
$$
\int_{K}^{Bn^{(r-1)/2}}\frac{ \left\Vert \mathcal L_{is}^n\right\Vert_{\mathcal B_1 \rightarrow \cB_2} }{|s|^{d_1}}\, ds=  o(n^{-\frac r2})\,. 
$$
Observe that $(C)$ does not imply $(D)$ when $\alpha>0$ but when $\alpha=0$ it implies $(D)[r](1)$ for all $r$ with $d_1=0$ and $(C)(2)$ and $(D)(2)$ coincide (hence, $d_2=1$).


The next lemma gives a useful sufficient condition for Assumptions $(C)(1)$ and $(D)(1)$ and will be used in our examples. 

\begin{lem}\label{DIneqImpliesDC}
Assume $1_{\bar \Delta}\in\mathcal B_1$ and that there exist positive constants $\alpha_1\in(0,1]$, $K,C,\alpha$, and an integer $n_0$ such that
\begin{equation}\label{DI2}
\forall n\ge n_0,\ \forall|s|>K,\quad\|\cL^{n}_{is} \|_{\cB_1\rightarrow 
\cB_2} \leq C |s|^{\alpha}e^{-Cn^{\alpha_1}} \quad\mbox{with}\ \cB_1\hookrightarrow\cB_2\hookrightarrow L^1(\bar\nu)\, .
\end{equation}
Then Assumption $(C)(1)$ holds true.
Moreover, Assumption $(D)[r](1)$ hold also true for
any $d_1\in [0,1]$ and $r\geq 0$ 
\end{lem}
\begin{proof}[Proof of Lemma~\ref{DIneqImpliesDC}]
To establish Assumption $(C)(1)$, choose $\hat \delta \in (0, K^\alpha C)$. Due to~\eqref{DI2}, for any $n\ge n_0$
$$
\left\Vert \mathcal L_{is}^n\right\Vert_{\mathcal B\rightarrow L^1(\bar\nu)}\lesssim 
|s|^\alpha e^{-Cn^{\alpha_1}} \leq |s|^\alpha  e^{-Cn^{\alpha_1}\frac{K^\alpha}{|s|^\alpha}} \leq |s|^\alpha e^{-n^{\alpha_1} \frac{\hat \delta}{|s|^\alpha }}\, ,\,\,\text{for}\,\, |s|>K.
$$
Assumption $(D)[r](1)$ is also straightforward because
\begin{align*}
\int_{K<|s|<Bn^{\frac{r-1}2}}\frac{ \left\Vert \mathcal L_{is}^n\right\Vert_{\mathcal B_1 \rightarrow \cB_2} }{|s|^{d_1}}\, ds &\leq e^{-Cn^{\alpha_1}} \int_{K<|s|<Bn^{\frac{r-1}2}}|s|^{\alpha-d_1}\, ds \\ &\leq 2\max(K^{\alpha-d_1}, \sqrt{n}^{(\alpha-d_1)(r-1)}) n^{(r-1)/2} e^{-Cn^{\alpha_1}}.
\end{align*}
\end{proof}

In \Cref{sec:dolgoineg}, we present general strategies to prove Assumptions $(C)(1)$ and $(D)(1)$. In Section~\ref{sec:cond(C)}, 
we explain how to infer Assumption $(C)$ with $\alpha\ge 0$ for a tower over a Gibbs Markov map satisfying Assumption $(C)$.

\subsubsection{Dolgopyat type inequalities}\label{sec:dolgoineg}

In order to verify our Assumptions $(C)$ and $(D)$, we present here two strategies that have been extensively used since the seminal works of Dolgopyat \cite{D1,D2} to establish rates of decay of correlation for suspension flows.

These two strategies rely on two different argument: the uniform non-integrability (UNI) and the absence of approximate eigenfunctions (AAE, see Section~\ref{sec:cond(C)}). While the former (UNI)
will imply both $(C)$ and $(D)$, the latter (AAE)
will only lead to $(C)$, and may be used when the first strategy fails.
We refer the reader to \cite{AGY, D1, D3} for a discussion about UNI and its applications to decay of correlations for flows, and to \cite{D2,IM,IM18} for  AAE. 

Both conditions can be interpreted in terms of contraction of the transfer operators $\cL_{is}$, and to talk about examples in the literature in general, we assume that there exist seminorms $|\cdot|_{j}$, $j=1,2$ such that $|\cdot|_1+|\cdot|_2$ is a norm which is equivalent to $\|\cdot\|_{\cB_1}$
and such that for the new norm
\begin{equation}\label{defnormes}
\|\cdot\|_{(s)} =\ |\cdot |_{1}+ (1+|s|)^{-1}|\cdot |_2\, ,
\end{equation}
on $\cB_1$. UNI implies the following Dolgopyat inequality: there exist $c, C,K>0$ such that
\begin{equation}\label{DolgopyatIneq1}
\forall |s|>K,\ \forall n \geq c \log |s|,\quad
\|\cL^n_{is} \|_{(s)} \leq 
 e^{-Cn}
\, ,
\end{equation}
and AAE
leads to an estimate of the following form: there exist $c,C,K>0$ such that
\begin{equation}\label{AAE}
\forall |s|>K,\quad
\|\cL^{\lceil c\log|s|\rceil }_{is} \|_{(s)} \leq 
 1-C|s|^{-\alpha'}\le e^{-C|s|^{-\alpha'}}
\, .
\end{equation}

The following lemma describes how the above estimates lead to estimates in Assumptions $(C)$ and $(D)$.

\begin{lem}\label{DIneqImpliesD}
Let $K>1$. Suppose \begin{equation}\label{eq:DolgopyatIneq2}
\forall |s|>K,\quad \| \cL^{\lceil c \log |s| \rceil}_{is}  \|_{(s)} \leq e^{-g(s)}\, ,
\end{equation}
with $g(s)>0$ and $M=\sup_{s,n} \|\cL_{is}^n\|_{(s)} <\infty$. 
Then, for every $\eps>0 
$, 
\begin{equation}\label{DIneg1}
\forall n\ge 0,\ \forall |s|>\max\left(K,e^{c^{-1}}\right),\quad  \|\cL^{n}_{is} \|_{(s)} \le M e^{2\eps c\, g(s)} e^{-\eps\frac{ ng(s)}{\log s}}\, ,
\end{equation}
and hence, there exists $M_0>0$ such that
\begin{equation}\label{DIneg2}
\|\cL^{n}_{is} \|_{\cB_1}\leq M_0\max(|s|,e^{2\eps c\, g(s)}) e^{-\eps\frac{ ng(s)}{\log s}}\quad\mbox{and}\quad
\|\cL^{n}_{is} \|_{\cB_1\rightarrow |\cdot|_1}\leq  M e^{2\eps c\, g(s)} e^{-\eps\frac{ ng(s)}{\log s}}\, .
\end{equation}
\end{lem}
\begin{proof}[Proof of Lemma~\ref{DIneqImpliesD}] Without loss of generality assume $s>K$. 
From \eqref{eq:DolgopyatIneq2}, we have 
$$\| \cL^{k\lceil c \log s \rceil}_{is}  \|_{(s)}\leq \| \cL^{\lceil c \log s \rceil}_{is}  \|_{(s)}^k \leq e^{-k g(s)} \, .$$
Assume $s\ge e^{c^{-1}}$, which implies $c\log s\ge 1$ and so $\lceil c\log s\rceil \le c\log s+1\le 2 c\log s$.

If $n=k\lceil c \log s \rceil+r$ where $0\leq r< \lceil c \log s\rceil $ then
$$\| \cL^{n}_{is}  \|_{(s)} \leq e^{- k g(s)} \|\cL^{r}_{is}  \|_{(s)}  \leq M e^{-n\frac{k g(s)}{k \lceil c \log s \rceil + r }} 
 \leq M e^{-n\frac{k g(s)}{2(k +1)   c \log s}} \, .$$
If $n\geq\lceil c\log s\rceil$, then $k\ge 1$ and so $k/(k+1)\ge k/(2k)=1/2$ and so
$$\| \cL^{n}_{is}  \|_{(s)} \leq M e^{-\frac{ ng(s)}{4   c \log s}} \, .$$
If $n\leq \lceil c\log s\rceil$, then 
$
e^{-n\eps\frac{g(s)}{\log s}}\ge e^{-\eps\lceil c\log s\rceil\frac{g(s)}{\log s}}\ge 
 e^{-2\eps c\, g(s)}\, ,
$
and so
\[
\Vert\cL_{is}^n\Vert_{(s)}\le 
M\le M e^{2\eps c\, g(s)}
e^{-\eps\frac{ng(s)}{\log s}}\, .
\]
This ends the proof of \eqref{DIneg1}.

Inequalities \eqref{DIneg2} follows directly from this and from
$$ (1+s)^{-1}(|h|_1+|h|_2)\leq \|h\|_{(s)} \leq |h|_1+|h|_2 $$
combined with the fact that $|\cdot|_1+|\cdot|_2$ is equivalent to $\|\cdot\|_{\cB_1}$. In fact, taking $$c_0=\Vert \textsf{Id}\Vert_{\cB_1\rightarrow |\cdot|_1+|\cdot|_2}\Vert \textsf{Id}\Vert_{ |\cdot|_1+|\cdot|_2\rightarrow \cB_1},$$
we obtain
$$\| \cL^{n}_{is}  \|_{\cB_1}\leq   c_0 (1+s) \| \cL^{n}_{is}  \|_{(s)} 
\,\,\text{and}\,\,\,\| \cL^{n}_{is}  \|_{\cB_1\rightarrow|\cdot|_1} \leq \| \cL^{n}_{is}  \|_{(s)} 
\, .$$
\end{proof}

The following is a direct consequence of the Lemma \Cref{DIneqImpliesDC} and \Cref{DIneqImpliesD}.
\begin{cor}
Assume the hypothesis of \Cref{DIneqImpliesD}. 
\begin{itemize}[leftmargin=*]
\item If $C|s|^{-\alpha_2}\log |s|\le g(s)$, and 
\begin{itemize}[leftmargin=*]
\item if $e^{2\eps cg(s)}=\cO(|s|^\alpha)$ with $\cB_2=\{f\, :\, |f|_1<\infty\}$ and Assumptions $(B)(2),$ $(A)[0](3)$ hold$,$ then Assumption $(C)$ holds true. 
\item if $\max(|s|,e^{2\eps cg(s)})=\cO(|s|^\alpha)$ with $\cB_2=\cB_1$,  
and Assumptions $(B)(2),$ $(A)[0](3)$ hold$,$ then Assumption $(C)$ holds true. 
\end{itemize} 
\item If $0<\inf_sg(s)/\log(|s|)<\sup_sg(s)/\log(|s|)<\infty$, 
then Assumption $(D)[r]$ holds true for any $r\geq 0$ for any choice of $d_1\in [0,1]$ in the cases of $\cB_2=\cB_1$ and $\cB_2=\{f\, :\, |f|_1<\infty\}$.
\end{itemize}
\end{cor}

\begin{rem}
Note that \eqref{DolgopyatIneq1} implies~\eqref{eq:DolgopyatIneq2} with $g(s)=C c\log |s|$. $\sup_{s,n}\|\cL_{is}^n\|_{(s)}<\infty$ holds as soon as $\sup_{s,n}|\cL_{is}^n|_1<\infty$ and if $\cL_{is}$ satisfies a uniform Doeblin Fortet inequality of the following form
\[
\exists\rho\in(0,1),\exists \tilde C>0,\quad \forall s,\forall h\in\mathcal B,\quad |\mathcal L_{is}h|_2\le \rho |h|_2+\tilde C(1+s)|h|_1\, .
\]
\end{rem}
\begin{rem}\label{Cfordeltabis}
The same conclusion holds if $\|\cdot\|_{(s)}$ is replaced by $\ |\cdot |_{1}+ |s|^{-1}|\cdot |_2$
since $   (K/(K+1))   s^{-1}\le (1+s)^{-1}\le s^{-1}$, or if $\|\cdot\|_{(s)}$ is replaced by $\max \left( |\cdot |_{1}, |s|^{-1}|\cdot |_2\right)$ since $\max(a,b)\le a+b\le 2\max(a,b)$.
\end{rem}

Now we list some examples for which \eqref{DolgopyatIneq1} holds. 
\begin{itemize}[leftmargin=*]
\item Subshifts of finite type with the set of H\"older continuous functions as $\cB_1$:

\eqref{DolgopyatIneq1} follows directly from the work of Dolgopyat (see \cite{D3}) provided that $\bar \phi$ is \textit{strongly non-integrable}, an analogue of UNI in the symbolic setting. This condition is satisfied by an open and dense class of observables. This is detailed in \Cref{SFT}.\\ 

\item Uniformly expanding Markov maps (see \cite[Definitions 2.2, 2.3, Proposition 7.4]{AGY}) with the set of $C^1$ observables as $\cB_1$: 

Let $(\bar{\Delta},\bar \nu)$ be a John domain and  $\{\bar\Delta_{k}\}_{k\in \mathbb E}$ (with $\mathbb E\subset\naturals$) be a full measure partition of $\bar \Delta$, let 
$\bar F:\bigcup_{k \in \mathbb E}\bar\Delta_{k} \to \bar \Delta $ and set $J$ for
the inverse of the Jacobian of $\bar F$ with respect to $\bar \nu$.  Suppose that
there exists $\lambda>1$ such that, for every $k\in\mathbb E$, the following properties hold true:

\begin{enumerate}
\item $\bar{F}$ is a $C^1-$diffeomorphism between $\bar \Delta_k$ and $\Delta$,
\item There exists $C_k\ge \lambda$ such that for all $v\in T_x\Delta$, $\lambda \|v\|\leq \|D\bar F(x).v\|\leq C_{k}\|v\|,$
\item  $\log J$ is $C^1$ on $\Delta_k$,
\item There exists $C'$ such that for all inverse branches $h$ of $\bar F^n$, $\|D(\log J)\circ h\|_{\infty} \leq C'$. 
\end{enumerate}
Further assume that $\bar\phi :\bar\Delta\rightarrow\mathbb R$ is $C^1$ on each $\bar\Delta_k$ and there exists $C(\bar\phi)\in(0,+\infty)$ such that for all inverses branches of $h$ of $\bar F$, $\|D(\bar\phi\circ h)\|_\infty \leq C(\bar\phi)$.

Under these assumptions, $(\bar{F}, \bar \phi)$ satisfies UNI:
There exist $C>0, n_0\in\naturals$ such that for all $n \geq n_0$ there exist two $\bar{F}^n-$inverse branches $h,k$ and a continuous unitary vector field $\Upsilon$ such that 
$$|D(\bar{S}_n \circ h)(x).\Upsilon(x) - D(\bar{S}_n \circ k)(x).\Upsilon(x)|>C\, ,$$
if and only if
\begin{quote}
$\bar{\phi}$ is not almost surely cohomologous (up to a measurable coboundary) to a \\ function constant on each $\bar\Delta_k$.
\end{quote}
\vspace{5pt}

From this, it is proved in \cite[Proposition 7.16]{AGY} that $\Vert \cL_{is}^n\Vert_{(s)}\le C\min(1,e^{-cn}|s|)$ taking $|\cdot|_1=\Vert\cdot\Vert_\infty$
and $|\cdot|_2=\Vert D(\cdot) \Vert_\infty$. This implies $\|\cL_{is}^{\lceil (2/c)\log s\rceil}\|_{(s)}\le e^{-\log s}.$ Applying Lemma~\ref{DIneqImpliesD} with $\cB_1$ for the set of $C^1$-observables, we obtain $$\Vert \cL_{is}^n\Vert_{\cB_1}\le C'|s|e^{- \eps n}\,\,\, \text{and}\,\,\, \Vert \cL_{is}^n\Vert_{\cB_1\rightarrow L^\infty}\le C'|s|^{\frac {4\eps}c}e^{-\eps  n}$$ for any $\eps\in(0,c/8)$, and hence, Assumption $(C)$ with $\cB_2=\cB_1$ and $\alpha_1=\alpha=1$, Assumption $(C)$ with $\cB_2=L^\infty$, $\alpha_1=1$ and any $\alpha>0$, and finally, Assumption $(D)$ for $(f,\cM,\mu)=(\bar F,\bar\Delta,\bar\nu)$ if $\psi\in L^1(\bar\nu)$ hold.\\


\item A particular case of the previous example -- Uniformly expanding maps of the torus:

Let $\bar{F}\in C^{r}(\bbT,\bbT)$ with $r \geq 2$ be such that $\inf{|\bar F\, '|} \geq \lambda >1$ and $\bar \phi \in C^{r-1}(\bbT,\reals)$. Then 
$(\bar{F}, \bar \phi)$ satisfies UNI in the following sense:
There exist $C>0$, $n_0 \in \bbN$ such that for all $n \geq n_0$ there exist two inverse branches $h,k$ of $\bar{F}^n$ satisfying 
$$\left| \frac{d}{dx} \bar S_n\circ h (x)-\frac{d}{dx} \bar S_n \circ k (x)\right| \geq C$$
for all $x\in \bbT$ if and only if
\begin{quote}
$\bar{\phi}$ is not cohomologous to a function constant on each maximal invertibility domain.
\end{quote}
 This implies that \eqref{DolgopyatIneq1} holds. See \cite[Appendix B]{DL}. 
\end{itemize}

\subsubsection{Assumption $(C)$  for towers via \emph{AAE} 
}\label{sec:cond(C)}

We assume here that $(\bar F,\bar\Delta,\bar\nu)$ is a tower over
a Gibbs Markov map  $(\widetilde F:=\bar F^{R(\cdot)},Y,\bar\nu_Y:=\bar\nu(\cdot|Y))$ where $Y=\bar\Delta_0$ is the base of the tower $\bar\Delta$ and
$R:\bar\Delta\rightarrow \mathbb N_0\cup\{\infty\}$ is the first visit time to $Y$ of $(\bar F^l (\cdot))_{l\ge 1}$, so $R(y)-1$ is the height 
of the tower $\bar \Delta$ over $y\in Y$.  Let $\ell:\bar\Delta\rightarrow\mathbb N_0$ be the level function given by $\ell(x,l)=l$.
Let $\widetilde{\mathcal L}$ be the transfer operator associated to $\widetilde{F}$ with respect to $\bar\nu_Y=\bar\nu(\,\cdot\,|Y)$. We have
\begin{equation}\label{decomptildeL}
\widetilde{\mathcal L}^k=\sum_{n\ge 1}\mathcal L^n(1_{Y\cap \{R_k=n\}}\,\cdot\,)\, ,
\end{equation}
where $\cL$ is the transfer operator corresponding to the tower and 
$R_k(\cdot):=\sum_{j=0}^{k-1}R(\wt{F}^j \cdot )$.

We consider a separation time $s_0:Y\times Y\rightarrow \mathbb N_0\cup\{+\infty\}$ such that $s_0(y,y')\ge m$ if and only if for every $j=0,...,m-1$, $\widetilde F^j(y)$ and $\widetilde F^j(y')$ belong to the same atom of the partition $\pi=\{Y_i,\ i\}$ of the Gibbs Markov map. Recall that $R|_{Y_i}$ is constant (let us write $r_i$ for this constant), $\widetilde F_i:=\widetilde F|_{Y_i}:Y_i\rightarrow Y$ defines a bijection and that $\widetilde {\cL}$ has the following form:
\begin{equation}\label{formuletildeL}
\widetilde {\cL}h=\sum_i(e^{g_0}h)\circ\widetilde F_{i}^{-1},\quad
\mbox{with}\ \ |e^{g_0(y)-g_0(y')}-1|\le C_{g_0}\beta^{s_0(y,y')}\, .
\end{equation}
The map $(\bar \Delta,\bar F)$ is also endowed with a separation time
$\hat s:\bar\Delta\times\bar\Delta\rightarrow \mathbb N_0\cup\{\infty\}$ which satisfies:
\[
  \hat s(\bar F^j(y),\bar F^j(y'))\ge 
 \left(\sum_{k=0}^{s_0(y,y')-1}R\circ \widetilde F^k(y)\right)-j\ge s_0(y,y')-j\, .
\]
We also write $\widetilde\phi:Y\rightarrow \mathbb R$ for the induced function associated to $\bar\phi:\bar\Delta\rightarrow\mathbb R$, i.e.,
\[
\widetilde\phi:=\sum_{k=0}^{R(\cdot)-1}\bar\phi\circ \bar F^k\, .
\]

Consider the family of operators $(\widetilde{\mathcal L}_{is,iu})_{s\in\mathbb R,\ u\in[-\pi,\pi]}$ given by
\[
\widetilde{\mathcal L}_{is,iu}:=\widetilde{\mathcal L}(e^{is \widetilde\phi+iuR})\, .
\]

Fix $\beta\in(0,1)$. Let $\widetilde{\mathcal B}$ be the set of $\beta$-dynamically Lipschitz functions on $Y$, i.e.~the set of functions $ h:Y\rightarrow\mathbb C$ such that the following quantity is finite
$$
\Vert h\Vert_{\widetilde \cB}=\Vert h\Vert_\infty+|h|_{\beta},\,\,\,\text{with}\,\,\, 
|h|_{\beta}:=\sup_{y,y'\in Y,\ y\ne y',\ s_0(y,y')\ge 1}\frac{|h(y)-h(y')|}{\beta^{s_0(x,y)}} 
$$
Let us also consider, as in Section~\ref{sechyp},  $\mathcal B_\beta$ to be the set  of $\beta$-dynamically Lipschitz functions on $\bar\Delta$, i.e., the set of functions $h:\bar\Delta\rightarrow\mathbb C$ such that
$$\Vert h\Vert_{\cB_\beta}=\Vert h\Vert_\infty+
\sup_{y,y'\in Y,\ y\ne y',\ \hat s(y,y')\ge 1}\frac{|h(y)-h(y')|}{\beta^{\hat s(x,y)}}<\infty\,.$$ 

Assume that $\bar\phi\in L^1(\bar\nu)$ 
and that
\begin{equation}\label{lipint}
\exists \gamma\in(0,1],\quad
\sum_i \bar\nu_Y(Y_i)\, 
\big|\widetilde\phi
|_{Y_i}\big|^\gamma_{\beta^{\frac 1\gamma}}<\infty\, .
\end{equation}
This implies that $\widetilde\phi$ is in $L^1(\bar\nu_Y)$ and that
$$\sum_i \bar\nu_Y(Y_i)\, \sup_{s\in\mathbb R}\frac{\big|e^{is\widetilde\phi}|_{Y_i}\big|_{\beta}}{1+|s|}<\infty, $$ which is enough to apply ideas in \cite{IM18}.

Observe that $\|\wt{\cL}_{is,iu}\|_\infty\leq 1$. Moreover 
it has been proved in \cite[Propositions 7.7 and 8.10]{IM18} (combined with \cite[Lemma 3.14]{IM}) that
$$\exists \bar C>0,\ \forall n\ge 0,\  \forall h\in \wt{\cB},\quad |\wt{\cL}^n_{is,iu} h|_{\beta} \leq \bar{C}(1+|s|)\|h\|_{\infty}+\bar{C}\beta^n |h|_{\beta}.$$
Here, we define the norm $\|\cdot\|_{(s)}$ as in~\eqref{defnormes} with $|\cdot|_1=\Vert\cdot\Vert_\infty$, $|\cdot|_2=|\cdot|_{\beta}$.

\begin{defin}\label{Dioph1}
We will say that $(Y,\widetilde F,\widetilde\phi,R,\beta)$ has no approximate eigenfunction $($hereafter written as $(Y,\widetilde F,\widetilde\phi,R,\beta)$ is \emph{AAE}$)$ if there exist a subset $Z_0\subset Y$ and $\alpha_0>0$ such that for any $\alpha,\xi>\alpha_0$ and $C>0$ and any sequences $(s_k,u_k,\psi_k,h_k)_k$ with $|s_k|\rightarrow+\infty$, $u_k\in[-\pi,\pi]$, $\psi_k\in[0,2\pi)$, $h_k\in\widetilde\cB$ where $|h_k|\equiv 1$ and $|h_k|_{\beta}\le C|s_k|$, 
\[
\exists y\in Z_0,\ \exists k\ge 1,\quad \left|e^{is_k\sum_{j=0}^{n_k-1}\widetilde\phi\circ \widetilde F^j(y)+iu_kR_{n_k}}h_k( \widetilde F^{n_k}(y))-e^{i\psi_k}h_k(y)\right|>C|s_k|^{-\alpha}\, ,
\]
for $n_k:=\lfloor \xi\log |s_k|\rfloor$.
\end{defin}
We refer to \cite[Section 5]{IM18} for a presentation of different criteria (temporal distance function, Diophantine conditions on periods and good asymptotics, all of which are innately dynamical) ensuring AAE, which is proved to hold for a wide class of $\bar{\phi}$.

\begin{prop}[Lemmas 7.7 and 7.12 of \cite{IM18}, Lemma 3.14 of \cite{IM}]\label{Lisn1-C}
If  $\bar\phi$ is a real valued $\bar\nu$-integrable function satisfying \eqref{lipint} and
if $(Y,\widetilde F,\widetilde\phi,\beta)$ is \emph{AAE},
then there exist 
$\alpha',\beta'>0$ and $C,C',C''\geq 1$ such that 
$\Vert \widetilde{\mathcal L}_{is,iu}^{\lceil \beta'\log|s|\rceil}\Vert_{(s)}\le 1-C'|s|^{-\alpha'}$ and $\Vert\widetilde \cL_{is,iu}^n\Vert_{(s)}\le C''$ for all $s \in \reals$ with $|s|>1$, and $u \in [-\pi,\pi]$.
\end{prop}
These estimates are key to establishing estimates 
of the form $\|(\text{I}-\wt{\cL}_{is,iu})^{-1}\|_{(s)} \leq C|s|^\alpha$
and ensuring, due to Lemma~\ref{DIneqImpliesD}, that for every positive integer $n$ and every real number such that $|s|\ge \max(1,e^{\tilde\beta^{-1}})$,
\[
\Vert \widetilde \cL_{is,iu}^n\Vert_{\widetilde\cB\rightarrow L^\infty}\le\Vert \widetilde\cL_{is,iu}^n\Vert_{(s)}\le C'' e^{\frac{2\eps\beta'C'}{|s|^{\alpha'}}}e^{-\frac{\eps C'n}{|s|^{\alpha'} \log s}}\le C''' e^{-\frac{\eps C'n}{|s|^{\alpha'} \log s}}\le C''' e^{-\frac{\eps C'n}{|s|^{\alpha}}}\, ,
\]
with $C'''=C''e^{2\eps\beta'C'}$, for any $\alpha>\alpha'$.
Therefore Assumption $(C)$ holds for $(\tilde F,Y,\bar\nu_Y)$ for the observable $\widetilde\phi$ with $\alpha_1=1$, any $\alpha>\alpha'$, $\mathcal B_1=\widetilde\cB$ and $\cB_2=L^\infty$.
The next lemma provides an estimate of $\cL^n_{is}$ by a direct proof (let us indicate that the strategy used in \cite[Lemma 4.4]{IM} and \cite{IM18} to prove such a result is based on the resolvent).


Let $\varsigma,\varsigma_0:\mathbb N_0\rightarrow[1,+\infty)$ and let $\mathcal Z$ be the set of $h:\bar\Delta\rightarrow\mathbb C$ such that
\[
\Vert h\Vert_{\mathcal Z}:=\sup_{j,k: \, k\le r_j}\Big[\varsigma(k)^{-1}\Vert (h\circ\bar F^k)|_{Y_{j}}\Vert_{\infty}+\varsigma_0(k)^{-1}\big| (h\circ\bar F^k) |_{Y_{j}}\big|_{\beta}\Big]<\infty\, .
\]
In Section~\ref{Towers}, we will consider generalization of the spaces considered by Young in \cite{Y}, which corresponds to taking $\varsigma,\varsigma_0$ of the form $\varsigma(k)=e^{k\eps}$ and $\varsigma_0(k)=e^{k\eps'}$ with $\eps'>\eps$.

\begin{lem}\label{BaseToTower}
Let $p_0\in(1,+\infty)$ and $\gamma\in(0,1]$.
Assume the following:
\begin{itemize}
\item $M:=\left\Vert \varsigma(\ell)\right\Vert_{L^{p_0}(\bar\nu)}<\infty,$\, $\mathbb E_{\bar\nu}(\varsigma_0(\ell))<\infty$ and
 \begin{equation}\label{hypophibasetotower}
\sum_{j}\bar\nu_Y(Y_j)
\sum_{k=0}^{r_{j}-1}k\, \varsigma(k)\big|(\bar\phi\circ \bar F^k) |_{Y_{j}}\big|_{\beta^{\frac 1\gamma}}^\gamma
<\infty\,  .
\end{equation}
\item there exist $c_1>0$ and $\vartheta_1\in(0,1)$ such that for every $j\ge 0$, $\bar\nu_Y(R\ge j)\le c_1\vartheta_1^j$
\item there exist $\alpha\ge 0$ and positive real numbers $\widetilde\alpha_0,  n_0,\widetilde C>0$ and $K>1$ such that
\begin{equation}\label{UNIbase}
\forall n\ge n_0,\ \forall |s|>K,\sup_{u\in[-\pi,\pi]} 
\left\Vert \widetilde{\mathcal L}_{is,iu}^n \right\Vert_{
\widetilde{\cB}\rightarrow L^{p_0}(\bar\nu_Y)} \le \widetilde C |s|^\alpha e^{- \frac{\widetilde\alpha_0 n}{|s|^{\alpha}}}
\, .
\end{equation}
\end{itemize}
Then for all $p_1\in[1,p_0)$,
there exist $C>0$ and $\alpha_0>0$ such that
\[
\Vert \cL_{is}^n\Vert_{\cZ\rightarrow L^{p_1}(\bar\nu)}\le  C |s|^{1+2\alpha} e^{-\frac{\alpha_0 n}{2|s|^\alpha}}\quad \mbox{for all }|s|>K\mbox{ and }n\ge n_0.
\]
In particular, Assumption $(C)(1)$ is satisfied with 
$\delta=K,\, \cB_1=\cB_\beta,\, \cB_2=L^{p_1}(\bar\nu),\, \alpha_1=1$
and $\alpha$ being replaced by $1+2\alpha$.
\end{lem}
\begin{proof}
Let $s_0\in(1,+\infty)$ be such that $\frac 1{p_0}+\frac 1{s_0}=\frac 1{p_1}$.
Let $u>0$ be such that $\mathbb E_{\bar\nu_Y}(e^{uR})$ is finite. Let us prove that there exists $x>0$ and $\vartheta_2\in(0,1)$ such that
\begin{equation}\label{LargedeviationRN}
\bar\nu_Y(R_N>xN) =\mathcal O \left(\vartheta_2^N\right)\, .
\end{equation}
Due to \eqref{formuletildeL}, $\widetilde {\cL}h=\sum_i(e^{g_0}h)\circ\widetilde F_{i}^{-1}$ and
\[
e^{g_0(\tilde F_i^{-1}(y))}\le \int_Y e^{g_0(\tilde F_i^{-1}(y'))}\left(1+C_{g_0}\beta^{s_0(y,y')+1}\right)\, d\bar\nu_Y(y')\le e^{C_{g_0}}\int_Y\cL(1_{Y_i})\, d\bar\nu_Y=e^{C_{g_0}}\bar\nu_Y(Y_i)\, ,
\]
and thus,
\[
\forall h\in L^\infty(\bar\nu_Y),\quad
\Vert\widetilde {\cL}(e^{uR}h)\Vert_{\infty}\le \sum_i e^{C_{g_0}}\bar\nu_Y(Y_i)e^{ur_i}\Vert h\Vert_\infty\le  e^{C_{g_0}}\mathbb E_{\bar\nu_Y}(e^{uR})\Vert h\Vert_\infty\, .
\]
Therefore, the linear map
$\widetilde{\cL}_{0,u}=\widetilde{\mathcal L}\left(e^{uR}\times\cdot\right)$ acts continuously on $L^\infty(\bar\nu_Y)$.
Let $x$ be such that $\vartheta_2:=e^{-xu}\Vert \widetilde{\cL}_{0,u}\Vert_\infty<1$. Then $\mathbb E_{\bar\nu_Y}(e^{uR_N})=\mathbb E_{\bar\nu_Y}(\widetilde{\cL}_{0,u}^N(1_Y))\leq \Vert \widetilde{\cL}_{0,u}\Vert_\infty^N$ and hence, it follows from the Markov inequality that
\[
\bar\nu_Y(R_N>xN)=\bar\nu_Y(e^{uR_N}>e^{xuN})\le e^{-xuN}\mathbb E_{\bar\nu_Y}(e^{uR_N}) =\mathcal O \left(\vartheta_2^N\right)
\]
and we have proved~\eqref{LargedeviationRN}.

We write $\bar\Delta_j:=\bar F^{j}(Y\cap\{R>j\})$ for the $j$-th level of the tower $\bar\Delta$ and
$\widetilde S_m:=\sum_{k=0}^{m-1}\widetilde \phi\circ \widetilde F^k$. Recall that $R_m:=\sum_{k=0}^{m-1}R\circ \widetilde F^k$.
Let $N$ be a positive integer.
If $x=\bar F^j(y)\in\bar\Delta_j$, then
\begin{equation}\label{decompSN}
\bar S_N(x)=\bar S_N(\bar F^j(y))=\widetilde S_{k}(y)-\bar S_j(y)+S_m(\widetilde F^k(y))\, ,
\end{equation}
with $N=R_k(y)-j+m$ and $0\le m<R\circ\tilde F^k(y)$ (i.e. $k,m$ are such that $\bar F^N(x)\in\bar\Delta_m$ and $\#\{\ell=1,...,N-1\, :\, \bar F^\ell(x)\in Y\}=k$).  Thus
\begin{equation}\label{AAAA0}
\mathcal L_{is}^{N}=\sum_{j,k,m\ge 0\, :\,  m+k-j\le N}\mathcal L_{is, N,j,k,m}\, ,
\end{equation}
with $\mathcal L_{is, N,j,k,m}=\mathcal L_{is}^N\left(1_{\bar\Delta_j\cap\{N=R_k+m\}\cap \bar F^{-N}(\bar\Delta_m)}\cdot\right)$.
Let $\epsilon\in(0,1)$ and $h\in\cZ$. Observe first that
\begin{align}
\bigg\Vert\sum_{m\ge \epsilon N,m,j} &\mathcal L_{is, N,j,k,m}(h)\bigg\Vert_{L^{p_1}(\bar\nu)} \nonumber \\ &\le  \left\Vert\mathcal L_{is}^N\left(1_{\bigcup_{m\ge \epsilon N}\bar\Delta_j}\circ\bar F^Nh\right)\right\Vert_{L^{p_1}(\bar\nu)}\le\left\Vert 1_{\bigcup_{m\ge \epsilon N}\bar\Delta_m}\mathcal L_{is}^N(h)\right\Vert_{L^{p_1}(\bar\nu)} \nonumber\\
\nonumber&\le\bar\nu\left(\bigcup_{m\ge \epsilon N}\bar\Delta_m\right)^{\frac 1{s_0}}\Vert \mathcal L_{is}^N(h) \Vert_{L^{p_0}(\bar\nu)}\le \left(\sum_{m\ge \epsilon N}\bar\nu_Y\left(R\ge m\right)\right)^{\frac 1{s_0}} M\Vert h \Vert_{\cZ}\\
&\le \left(\bar\nu(Y)\sum_{m\ge \epsilon N}\bar\nu_Y\left(R\ge m\right)\right)^{\frac 1{s_0}}M\Vert h \Vert_{\cZ}=\cO\left(\vartheta_1^{\frac{\epsilon}{s_0} N}\Vert h \Vert_{\cZ}\right)\, .
\label{smallm}
\end{align}
Note that if $m\le\epsilon N$, then $R_k=N+j-m\ge (1-\epsilon)N+j\ge (1-\epsilon)N$ and that, in this case, $k\le a_N$ implies $R_{a_N}\ge (1-\epsilon)N$.
Set  $a_N:= (1-\epsilon)N/x$.
Set $p\in (1,+\infty)$. 
Observe that
\begin{align*}
\left\Vert\sum_{k\le a_N,\,  j \ge 0,\, 0\le m\le\eps N} \mathcal L_{is, N,j,k,m}(h)\right\Vert_{L^{p_1}(\bar\nu)} 
&\le \left\Vert\mathcal L^N\left(\sum_{j\ge 0} 1_{\bar\Delta_j\cap \bar F^j(Y\cap\{R_{a_N}\ge (1-\epsilon)N\})}|h|\right)\right\Vert_{L^{p_1}(\bar\nu)}
\\
&\le \left\Vert \sum_{j\ge 0}1_{\bar F^j (Y\cap\{R>j,\, R_{
a_N}\ge(1-\epsilon) N\})}\right\Vert_{L^{s_0}
(\bar\nu)}
\Vert h\Vert_{L^{p_0}(\bar\nu)}
\\
&\le \left(\bar\nu(Y)\sum_{j\ge 0}
\bar\nu_Y\left( R>j,\, R_{a_N}\ge (1-\epsilon)N\right)\right)^{\frac 1{s_0}}
M\Vert h\Vert_\cZ
\\
&\le \left(
\bar\nu(Y)\mathbb E_{\bar\nu_Y}\left( R 1_{\{R_{a_N}\ge (1-\epsilon)N\}}\right)
\right)^{\frac 1{s_0}} M\Vert h\Vert_\cZ \\
&\le \left(\bar\nu(Y)\Vert R\Vert_{L^p(\bar\nu_Y)}\right)^{\frac 1{s_0}}
\left(\bar\nu_Y\left(R_{a_N}\ge (1-\epsilon)N\right)\right)^{\frac {p-1}{s_0p}}M\Vert h\Vert_\cZ\, .
\end{align*}
But
$\bar\nu_Y\left(R_{a_N
}\ge (1-\epsilon)N\right)=\cO\big(\vartheta_2^{\frac{(1-\epsilon)N}x}\big)$ 
due to~\eqref{LargedeviationRN},
and so
\begin{equation}\label{AAAA2b}
      \left\Vert\sum_{k\le\frac{(1-\epsilon)N}x,\  j,m\ge 0} \mathcal L_{is, N,j,k,m}(h)\right\Vert_{L^{p_1}(\bar\nu)}
      =\mathcal O\left(\vartheta_2^{\frac{(1-\epsilon)(p-1)N}{ps_0x}}\Vert h\Vert_\cZ\right)\, .
\end{equation}
It remains to estimate
$\left\Vert\sum_{k\ge
a_N 
,0\le j,m\le \epsilon N} \mathcal L_{is, N,j,k,m}(h)\right\Vert_{L^{p_1}(\bar\nu)}$.
To this end, we observe that
\begin{align*}
\mathcal L_{is, N,j,k,m}( h)&
=\mathcal L^{N}\left(e^{is\bar S_N}\, 1_{\bar\Delta_j\cap\{N=R_k+m\}\, \cap \bar F^{-N}(\bar\Delta_m)}\, h\right)\\
&=\mathcal L^{N+j}\left((e^{is\bar S_N} 1_{\bar\Delta_j\cap\{N=R_k+m\}\cap \bar F^{-N}(\bar\Delta_m)}h)\circ\bar F^j\right)\\
&=
\mathcal L^{N+j}\left(e^{is\bar S_{N}\circ \bar F^j}
1_{Y\cap\{R>j\}}(h\circ \bar F^j)1_{\{R_k=N+j-m\}}1_{Y\cap\{R>m\}}\circ\widetilde F^k\right)\, .
\end{align*}
Therefore, using~\eqref{decompSN} and \eqref{decomptildeL}, we obtain
\begin{align*}
\mathcal L_{is, N,j,k,m}( h)
&=\mathcal L^{N+j}\left((1_{Y\cap\{R>m\}}e^{isS_m})\circ \widetilde F^k\, 1_{\{R_k=N+j-m\}}e^{is\widetilde S_{k}}\, 1_{Y\cap\{R>j\}}\, e^{-is\bar S_j}\, h\circ\bar F^j\right)\\
&= \mathcal L^m\left(\widetilde{\mathcal L}^{k}\left((1_{Y\cap\{R>m\}}e^{isS_m})\circ \widetilde F^k 1_{\{R_k=N+j-m\}}\, e^{is\widetilde S_{k}}\, e^{-is\bar S_j}1_{Y\cap\{R>j\}}\, h\circ\bar F^j\right)\right)\\
&= \mathcal L^m\left((1_{Y\cap\{R>m\}}e^{isS_m})\widetilde{\mathcal L}^{k}\left( 1_{\{R_k=N+j-m\}}\, e^{is\widetilde S_{k}}\, e^{-is\bar S_j}\, 1_{Y\cap\{R>j\}}\, h\circ\bar F^j\right)\right)\, .
\end{align*}
Moreover, setting $H_j:= e^{-is\bar S_j}\, 1_{Y\cap\{R>j\}}\, h\circ\bar F^j$, we obtain
\begin{align*}
\widetilde {\cL}^{k}\left(1_{\{R_k=N+j-m\}}e^{is\widetilde S_{k}}H\right)
&=\frac 1{2\pi} \int_{[-\pi,\pi]}e^{-iu(N+j-m)}\widetilde{\cL}^k\left(e^{is\widetilde S_{k}+uR_k}H_j\right)\, du\\
&=\, \frac 1{2\pi} \int_{[-\pi,\pi]}e^{-iu(N+j-m)}
\widetilde{\cL}_{is,iu}^ {k-1}\left(\widetilde{\cL}_{is,iu}(H_j)\right)\, du\, .
\end{align*}
Now using~\eqref{UNIbase}
combined with the H\"older inequality with $\frac 1{p_0}+\frac 1{s_0}=\frac 1{p_1}$,
we obtain for $|s|>K$,
\begin{align}
&\sum_{k\ge a_N}\sum_{j,m\ge  0} 
\left\Vert\mathcal L_{is, N,j,k,m}\right\Vert_{L^{p_1}(\bar\nu)} \nonumber\\
&\hspace{5em}=\mathcal O\left(\sum_{k\ge a_N}\sum_{j,m\ge  0}|s|^\alpha e^{-\frac{\widetilde\alpha_0 (k-1)}{|s|^\alpha}}\left\Vert  \widetilde{\cL}_{is,iu}(H_j)\right\Vert_{\widetilde{\mathcal B}}\left\Vert 1_{\{R>m\}}e^{-is\bar S_m}\right\Vert_{L^{s_0}(\bar\nu_Y)}\right)\nonumber\\
&\hspace{5em}=\mathcal O\left(\sum_{k\ge a_N}\sum_{j,m\ge 0}|s|^\alpha e^{-\frac{\widetilde\alpha_0 {k-1}}{|s|^\alpha}} \Vert  \widetilde{\cL}_{is,iu}(H_j)\Vert_{\widetilde{\mathcal B}}\, \bar\nu_Y(R>m)^{\frac 1{s_0}}\right)\nonumber\\
&\hspace{5em}=\mathcal O\left( |s|^{2\alpha}\sum_{j\ge 0}\Vert  \widetilde{\cL}_{is,iu}(H_j)\Vert_{\widetilde{\mathcal B}} \, e^{-\frac {\widetilde\alpha_ 0  a_N}{|s|^\alpha}}\right)\, ,\label{majointerm}
\end{align}
since $\sum_{k\ge 0} e^{-\frac
{\widetilde\alpha_ 0 k}{
|s|^\alpha}}=O(
|s|^\alpha
)$.
Now it remains to prove that
\begin{equation}\label{Goodmajo}
\sum_{j\ge 0}\Vert  \widetilde{\cL}_{is,iu}(H_j)
\Vert_{\widetilde{\mathcal B}}\le C''(1+|s|)\Vert h\Vert_{\mathcal Z}\, 
\end{equation}
where $H_j:= e^{-is\bar S_j}\, 1_{Y\cap\{R>j\}}\, h\circ\bar F^j$.
First we observe that
\[
\widetilde{\cL}_{is,iu}(H_j)=\sum_{j':r_{j'}>j}\left(e^{g_0+iu\, r_{j'}+is\sum_{k=j}^{r_{j'}-1}\bar\phi\circ\bar F^k}\, h\circ\bar F^j\right)\circ\widetilde F^{-1}_{j'}\, .
\]
Therefore,
\begin{equation}\label{infinitenorm}
\sum_{j\ge 0}\Vert  \widetilde{\cL}_{is,iu}(H_j)\Vert_\infty\le \sum_{j\ge 0} \varsigma(j)\Vert h\Vert_\cZ e^{C_0}\bar\nu_Y(R>j)\le   \Vert h\Vert_\cZ e^{C_0}\mathbb E_{\bar\nu}(\varsigma)/\bar\nu(Y)\, .
\end{equation}
Next, let us control $\sum_{j\ge 0}\big| \widetilde{\cL}_{is,iu}(H_j)\big|_\beta$.
Consider $y,y'$ such that $s_0(y,y')\ge 1,$ and their primages
$y_{j'},y'_{j'}$ under $\widetilde F$ belonging to $Y_{j'}$, then
\begin{equation}\label{Lipshitzh}
e^{g_0(y_{j'})}|h(\bar F^j(y_{j'}))-h(\bar F^j(y'_{j'})|\le e^{C_0}\bar\nu_Y(Y_{j'})\, \Vert h\Vert_\cZ \varsigma_0(j)
    \beta^{\hat s(y_j,y'_j)}
\, ,
\end{equation}
\begin{equation}\label{Lipshitzg}
\left|h(\bar F^j(y_{j'}))(e^{g_0(y_{j'})}-e^{g_0(y'_{j'})})\right|\le \Vert h\Vert_\cZ \varsigma(j) e^{g_0(y_{j'})} C_0\beta^{s_0(y,y')}\le \Vert h\Vert_\cZ e^{C_0}\bar\nu_Y(Y_{j'}) \varsigma(j) C\beta^{s_0(y,y')} \, ,
\end{equation}
and
\begin{align}
\nonumber\bigg|h(\bar F^j(y_{j'}))e^{g_0(y_{j'})}&(e^{is\sum_{k=j}^{r_{j'}-1}\bar\phi\circ\bar F^k(y_{j'})}-e^{is\sum_{k=j}^{r_{j'}-1}\bar\phi\circ\bar F^k(y_{j'})})\bigg|\\
&\ \ \ \ \ \ \le \Vert h\Vert_\cZ \varsigma(j) e^{C_0}\bar\nu_Y(Y_{j'})2\left(|s| \sum_{k=j}^{r_{j'}-1}\big|(\bar\phi\circ \bar F^k)|_{Y_{j'}}\big|_{\beta^{\frac 1\gamma}} \beta^{\frac{s_0(y,y')+
1}{\gamma}}\right)^\gamma\, .
\label{Lipshitzphi}
\end{align}
Gathering \eqref{Lipshitzh}, \eqref{Lipshitzg} and \eqref{Lipshitzphi} (using the fact that $\left(\sum_k a_k\right)^\gamma\le \sum_k a_k^\gamma$ since $\gamma\in(0,1]$) and summing over $\sum_{j\ge 0}\sum_{j'\, :\, r_{j'}>j}$, we obtain
\begin{align*}
\sum_{j\ge 0}&\left| \widetilde{\cL}_{is,iu}(H_j)\right|_\beta \\ &\le C'\Vert h\Vert_\cZ \beta^{s_0(y,y')}\left(
\mathbb E_{\bar\nu}(\varsigma_0(\ell)+\varsigma(\ell))/\bar\nu(Y)+
|s|^\gamma \sum_{j'}\bar\nu_Y(Y_{j'})
\sum_{k=0}^{r_{j'}-1}k\,\varsigma(k)\big|(\bar\phi\circ \bar F^k)|_{Y_{j'}}\big|^\gamma_{\beta^{\frac 1\gamma}} 
\right)
\, ,
\end{align*}
which is $\mathcal O\left((1+|s|)\Vert h\Vert_{\cZ}\right)$ due to our assumption \eqref{hypophibasetotower} on $\bar\phi$. Combining this with \eqref{infinitenorm}, we obtain
\eqref{Goodmajo} and end the proof of the Lemma thanks to \eqref{smallm},  \eqref{majointerm} and \eqref{AAAA2b}.
\end{proof}

Therefore, whenever the factor $(\bar F,\bar \Delta, \bar \nu)$ in our setting (recall Figure~\ref{setup}), and the observable $\bar \phi$ given by Assumption $(A)$, satisfy the hypothesis of the \Cref{BaseToTower}, we have Assumption $(C)$. In fact, there is a large class of nonuniformly hyperbolic dynamical systems for which this Lemma applies. For example, the quotient system defined in \Cref{sechyp} for a Young tower $(F,\Delta, \nu)$ is a tower over a Gibbs Markov map. So any system $(f,\cM,\mu)$ modeled by a Young tower with exponential tales is a suitable candidate. We describe them in detail in \Cref{Towers}. 


\section{Subshifts of Finite Type}\label{SFT}

In this section, we establish exact limit theorems for mixing invertible subshifts of finite type (SFTs). Many concrete dynamical systems like Axiom A diffeomorphisms can be studied by converting them to SFTs via a symbolic coding.  Hence, the exact limit theorems we establish here, apply beyond the setting in which they are introduced. To illustrate this, we end this section with an application  of our results to co-compact group actions on $\bbH^2$.

\subsection{Context and results}
Let us recall some facts about SFTs without proof. \cite[Chapters 1--4]{PP} contain a detailed account of the theory as well as proofs of the following.

Let $A$ be a $k \times k$ matrix with only $0$ and $1$ as entries. Define  
$$\Sigma_A=\Big\{ \vec{x}=(x_j)_{j\in\mathbb Z}\, :\, x_j \in \{1,2,\dots,k\},\  A(x_j,x_{j+1})=1, \ \forall j \in \integers  \Big\}.$$
We consider the shift $\sigma : \Sigma_A \to \Sigma_A$ acting on a sequence by moving elements to the left by one position, i.e., $\sigma\big((x_n)_{n\in \integers}\big)=(x_{n+1})_{n\in\integers}$. Then, $(\sigma,\Sigma_A)$ is called a subshift of finite type (also known as a topological Markov chain). Define the period $d$ of $A$ by $d = \gcd \{n\ |\ \exists j, A^n_{jj}>0 \}$. If $d=1$, $A$ is called aperiodic. Also, $A$ is called irreducible if for all $i,j$ there exists $N$ such that $A^N_{ij}>0$. We assume from now on that $A$ is irreducible and aperiodic.
Let $\beta \in (0,1)$. 
We endow $\Sigma_A$ with the metric $\dd$
given by $\dd(\vec{x},\vec{y})=\beta^N$ where $N \in \integers$ is the supremum of the nonnegative integers $N$ such that $x_j=y_j$ for all $|j|<N$. For $\vec{x} \in \Sigma$, we write $\vec{x}_+=(x_n)_{n\geq 0}$ and $\vec{x}_-=(x_n)_{n \leq 0}$.

Take $F_\beta$ to be the set of complex valued Lipschitz continuous functions on $\Sigma_A$. We endow this space with the norm
 $\|\cdot\|_\beta = |\cdot|_\infty+ |\cdot|_\beta$ where
$|h|_\infty = \sup_{\Sigma_A}|h|$ and
$|h|_\beta= \sup_{\vec x\ne \vec y}\frac{|f(\vec x)-f(\vec y)|}{\dd(\vec x,\vec y)} $.
Then, $(F_\beta,\|\cdot\|_\beta)$ is a Banach space such that $\|\cdot\|_\beta-$bounded sets are $|\cdot|_\infty-$compact. 


\begin{defin}[\cite{PP}]
Let $\phi\in F_\beta$ be real valued. $\phi$ is said to be  {\bf generic} if the only solution $H\in F_\beta$ to $H(\sigma(\x)) = e^{itf(\x)}H(\x)$ is a constant $H$ and $t=0$. 
\end{defin}
\begin{defin}[\cite{D3}]
Let $\phi\in F_\beta$ be real valued. $ \phi$ is said to be {\bf strongly non-integrable} if  there exist $\delta, \vec{y}^1, \vec{y}^{2}$ with $\vec{y}^1_+=\vec{y}^2_+$ and a neighbourhood of $\vec{y}^1$, $U$, such that for all $\vec{x}^1 \in U$ with $\vec{x}^1_-=\vec{y}^1_-$ there exist $\vec{x}^2,\vec{x}^3,\vec{x}^4$ with 
$\vec{x}^2_-=\vec{y}^1_-$, $\vec{x}^2_+=\vec{x}^4_+$, $\vec{x}^4_-=\vec{x}^3_-=\vec{y}^2_-$, $\vec{x}^3_+=\vec{x}^1_+,$ $\dd(\vec{x}^1, \vec{x}^2)\leq \beta^N, \dd(\vec{x}^3, \vec{x}^4) \leq \beta^N $
and $|\varphi(\vec{x}^1,\vec{x}^2,\vec{x}^3,\vec{x}^4)| \geq \delta \beta^{N/2}\
$
where $\varphi$ is the temporal distance function 
\begin{equation*}
\varphi(\vec{x}^1,\vec{x}^2,\vec{x}^3,\vec{x}^4):=\sum_{k\in \integers} [\phi(\sigma^k\vec{x}^1)-\phi(\sigma^k\vec{x}^2)-\phi(\sigma^k\vec{x}^3)+\phi(\sigma^k\vec{x}^4)]\, .
\end{equation*}
\end{defin}

Note that strong non-integrability implies genericity, and due to ideas in \cite{D1, D2, D3}, the strong non-integrability condition will be the key to ensure Assumptions $(C)$ and $(D)$. Moreover, this condition is satisfied by an open and dense subset of observables $\phi$ of $F_\beta$.

Given a real valued $g\in F_\beta$, called potential,
we consider the unique invariant probability measure $\nu_g$ on $\Sigma_A$ which is $\sigma^{(+)}$-invariant and maximise $\mu\mapsto h_\mu(\sigma^{(+)})+\int_{\Sigma_A} g\, d\mu$ where $h_\mu(\sigma^{(+)})$ is the entropy of $\sigma$ with respect to $\mu$. This measure $\nu_g$ is called the stationary equilibrium state of $g$, or Gibbs measure with potential $g$.
\begin{thm}\label{StongExpOrd1SFT}
Let $\beta\in(0,1)$.
Suppose $(\sigma,\Sigma_A,\nu_g)$ is an invertible subshift of finite type with an irreducible, aperiodic $A$, endowed with a Gibbs measure, $\nu_g$, with potential $g \in F_\beta$.  Let $\phi \in F_\beta$ be $\mathbb X-$valued. 
Assume $\phi$ is $\nu_g$-centered and set $S_n:=\sum_{k=0}^{n-1}\phi\circ \sigma^k$.
\begin{itemize}[leftmargin=*]
\item If $\bbX=\reals $ and $\phi$ is generic, for any probability measure $\mathbb P$ absolutely continuous with respect to $\nu_g$ having density $\psi\in F_{\sqrt{\beta}}$, then the order $1$ Edgeworth exapnsion for $S_n$ exists. 
If, moreover, $\phi$ is strongly non-integrable, then all order Edgeworth expansions for $S_n$ exist.
\item If $\bbX=\reals$ and $\phi$ is strongly non-integrable, or if $\bbX=\integers$, 
then, for all $\psi, \xi \in F_{\sqrt{\beta}}$, and 
for all $g \in \fF^{q+2}_{0}$ where $q>0$, both local and global expansions in the \emph{MLCLT} of every order exist. 
\end{itemize}

\end{thm}
\begin{proof}
Up to add a constant to $g$, we will assume without any loss of generality that $\int_{\Sigma_A}g\, d\nu=-h_{\nu}(\sigma)$.
Recall that SFTs fit the framework of Section~\ref{sechyp}
 with $(F,\Delta)=(\sigma,\Sigma_A)$
and with $(\bar F,\bar\Delta)$ isomorphic to $(\sigma^+,\Sigma^+_A)$,
where
$$\Sigma^+_A=\Big\{ \vec{x}=(x_j)_{j\ge 0}\, :\, x_j\in\{1,2,\dots,k\},\,  A(x_j,x_{j+1})=1, \ \forall j \in \naturals_0  \Big\},$$ 
and $\sigma^+\big((x_n)_{n\geq 0}\big)=(x_{n+1})_{n\geq 0}$.

$F_\beta$  corresponds to
$ {\cB}^{(0)}_\beta$ in Section~\ref{sechyp} with $\hat s(x,y)=\inf\{n\ge 0\, :\, x_n\ne y_n\}$
and $\Vert f\Vert_\beta\le\Vert f\Vert^{(0)}_{ \beta}\le 2\Vert f\Vert_\beta$.
The corresponding function spaces on $\Sigma^+_A$ are defined analogously, replacing $\integers$ by the set of nonnegative integers, and are denoted by a superscript $+$, and $(F^+_\beta,\|\cdot\|_\beta)$ is also a Banach space. It corresponds exactly to $(\mathcal B_\beta,\Vert\cdot\Vert_{\mathcal B_\beta})$
of Section~\ref{sechyp}.

With the above identifications, 
$(\Delta,F,\nu)=(\Sigma_A,\sigma,m)$
satisfies the condition $(H0)$ of Section~\ref{sechyp}.
Hence, Assumption $(A)(1)$ follows from Lemma~\ref{LEM0}, which 
ensures the existence of $\bar\phi\in F^+_{\sqrt{\beta}}$ and of $\chi\in F_{\sqrt{\beta}}$ such that $\phi=\bar\phi\circ\mathfrak p+\chi-\chi\circ F$.
Assumption $(0)$ holds true with $(f,\mathcal M,\mu)=(F,\Delta,\nu)=(\sigma,\Sigma_A,\nu_g)$ and $(\bar F,\bar \Delta,\bar \nu)=\big(\sigma^+,\Sigma_A^+,\bar{\mathfrak p}_*(\nu_{\bar g})\big)$ with $\bar{\mathfrak p}(\vec x)=(x_n)_{n\ge 0}$ and
$\mathfrak p=id$. The non-arithmeticity is ensured by the genericity of $\phi$.

It is known that $\bar{\mathfrak p}_*(\nu_{\bar g})$ corresponds to the Gibbs measure on $\Sigma_A^+$ associated to a potential $\bar g\in F_{\sqrt\beta}^+$ such that $\phi=\bar\phi\circ\mathfrak p+h-h\circ F$ with $h\in F_{\sqrt{\beta}}$. Thus
Assumption $(H1)$ of Section~\ref{sechyp} holds also true with this $\bar g$. Condition $(A)[r](2)$ then follows due to the second part of Lemma~\ref{LEM0}.
It is clear that $(\cL_{is})_{s\in\reals}$ is an analytic family of bounded linear operators on  $F_{\sqrt\beta}^+$ (see \Cref{lemA_B}). So we have $A[r](3)$ for all $r$.

Assumption $(B)(1)$ follows from the Ruelle-Perron-Frobenius Theorem (see \cite[Theorem 2.2]{PP}). The aperiodicity and irreducibility of A implies that the system is exact. Combining all this with the genericity of $\phi$ and $F_{\sqrt{\beta}}^+ \hookrightarrow L^2(\bar\nu)$, we have the hypothesis of \Cref{periphericalspect}, and hence Assumption $(B)$.   
This ends the proof of the first part of the first statement. 

Assume from now on that $\phi$, and so $\bar\phi$, is $\reals-$valued and strongly non-integrable.
Due to the work of Dolgopyat (see \cite{D1, D2}), the 
strong non-integrability condition implies the existence of  $C,s_0,\alpha'>0$ and $\gamma \in (0,1)$ such that
\begin{equation}\label{SFTDI}
\forall |s|>s_0, \quad\|\cL^n_{is}\|_{F^+_{\sqrt\beta}} \leq C \min\{\gamma^n|s|^{\alpha'}, 1\}\, .
\end{equation}
Without loss of generality we assume $\alpha'\ge 1$.
Fix $\alpha>0$ such that $q=1>\alpha\big(1+\frac{r+1}{2}\big)$.
Due to Lemmas~\ref{DIneqImpliesD} with $c=(\alpha'-1)/|\log(\gamma)|$ and $g(s)=\log |s|$, this ensures Assumptions $(C)$
with $\cB_1=F^+_{\sqrt{\beta}}$, $\cB_2=L^\infty$ and $\alpha_1=1$. 
Also $(D)[r]$ holds for $\psi\in F_\beta$ for all $r$.
When $\phi$ is $\mathbb Z$-valued, Assumptions $(A)[0](3)$ and $(B)(2)$ imply Assumption $(C)$ with $\alpha=0$ and $\alpha_1=1$. This establishes the rest of the theorem. 
\end{proof}

\subsection{An application to co-compact group action on $\bbH^2$} Let $\Gamma\subseteq PSL(2,\reals)$ be a co-compact or a convex co-comapct group acting on $\mathbb{H}^2$ via linear fractional transformations. This action can be coded using a subshift of finite type $(\Sigma^+_A, \sigma^+)$ by associating it to finite directed graph whose edges are labelled by the generators of $\Gamma$ and their inverses. Fix $y\in \bbH^2$. Then for a sequence of infinite path $\vec{x}\in \Sigma^+_A$ in the associated directed graph, there is induces an action on $\partial \mathbb{H}^2$ via the map $\pi: \Sigma^+_A \to \partial \mathbb{H}^2$ given by $\pi(\vec{x})= \lim_{n \to \infty} g_{x_{n-1}} \dots  g_{x_0} y$. 

Let $g: \Sigma \to \reals$ be a $\beta-$H\"older continuous function and $\mu_g$ be the unique Gibbs measure associated to $g$. Then there is an associated measure $\pi^*(\mu_g)$ on $\partial \mathbb{H}^2$ which is absolutely continuous (same measure class as the Patterson-Sullivan measure) when $\mu_g$ is the measure of maximal entropy. For a detailed discussion about this construction and claims we make without proofs, we refer the reader to \cite{PS}. 

Our results imply the existence of Edgeworth expansions, local and global expansions in the MLCLT for $S_n=\phi(\vec{x})+\dots+\phi(\sigma^{n-1}\vec{x})=\log|(g^{-1}_{x_{n-1}}\dots g^{-1}_{x_0})'(\pi(\vec{x}))$ where $\phi(\vec{x})=\log{|(g^{-1}_{x_{0}} )'(\pi(\vec{x}))|}-\lambda_\mu$ and $\lambda_\mu = \lim_{n \to \infty} \frac{1}{n}\textrm{d}(x,g_{x_0}\dots g_{x_{n-1}}x)$ (as a $\mu$-a.s.~limit). The associated operators $\cL_{is}$ acting on $F_\beta$ satisfy Assumptions $(A)[r], (B)$ and $(D)[r]$ for all $r$ (recall that we are in the case $(\bar{F},\bar{\Delta},\bar \nu)=(f,\cM,\mu)$). In particular, given $\delta>0$, there exist $C_\delta>0$, $\theta_\delta \in (0,1)$ and $\gamma >0$ such that for all $|t|\geq \delta$, 
$$\|\cL^n_{is}\|_{F^+_\beta} \leq C_\delta \min(\theta_\delta^n|t|^\gamma, 1).$$

As a consequence, apart from the results in \Cref{StongExpOrd1SFT} (which is a significant refinement of \cite[Corollary 1.2]{PS}, prescribing all order asymptotics), we have the following statistical results due to \cite[Section 5]{FL}.

\begin{enumerate}[leftmargin=*]
\item Moderate deviations: Then for all $c \in (0,r)$, when $1 \leq z\leq \sqrt{c \sigma^2\ln n},$
$$\lim_{n\to \infty} \frac{1-\mu \Big\{\vec{x}\in \Sigma^+_A\, \Big|\, \log|(g_{x_{n-1}}\dots g_{x_0})'(\pi(\vec{x}))|-n\lambda_\mu \leq  z \sqrt{n} \Big\}}{1-\fN(z)}=1.$$
\item Local limit theorems: Let $r \in \naturals$. Suppose $\eps_n \to 0$ and $\eps_n n^{r/2} \to \infty$ as $n \to \infty$. Then 
\begin{align*}
\frac{\sqrt{n}}{2\epsilon_n}\mu \Big\{\vec{x}\in \Sigma^+_A\, \Big|\, \log|(g_{x_{n-1}}\dots g_{x_0})'&(\pi(\vec{x}))|-n\lambda_\mu \in (u - \eps_n , u+\eps_n ) \Big\}\\ &=\frac{1}{\sqrt{2\pi \sigma^2}}e^{-\frac{u^2}{2n\sigma^2}}+\cO\big(\min(n^{-1/2},\eps_n n^{-1})\big) + o( \eps^{-1}_n n^{-r/2})
\end{align*}
as $n\to \infty$, uniformly for $u \in \reals$. Thus, we recover \cite[Theorem 1.3]{PS} but with more precise asymptotics.
\end{enumerate}

\section{Young Towers}\label{Towers}
\subsection{Context}\label{contextYoung}
In \cite{Y}, Young considered hyperbolic dynamical systems $(f,\mathcal M,\mu)$ and modeled them using towers by considering a subset $
\Lambda=\bigcup_{i\ge 0}\Lambda_i\subset \mathcal M$ with a product structure and with a return time $R:\Lambda\rightarrow\mathbb N^*$ (that is, $\widetilde F(x):=f^{R(x)}(x)\in\Lambda$ for all $x\in\Lambda$) which is constant equal to some positive integer $r_i$ on each $\Lambda_i$. Let us insist on the fact that $R$ is a priori not the first return time to $\Lambda$. 
We will recall only the properties we will use. 

Recall that the tower $(\Delta,F)$ is given by:

\begin{itemize}[leftmargin=20pt]
\item The space $\Delta$ is given by $\Delta=\bigcup_{i\ge 0}\left(\Lambda_i\times\{0,...,r_i-1\}\right)$,
\item The map $F$ is given by $F(x_0,l)=(x_0,l+1)$ and $F(x_0,r_i-1)=(f^{r_i}(x_0),0)$ if $x_0\in\Lambda_i$ and $l<r_i-1$.
\end{itemize}

For any integer $l\geq 0$, we write $\Delta_l$ for the $l$-th level of the tower $\Delta$, i.e., $\Delta_l=(\bigcup_{i:r_i>l}\Lambda_i)\times\{l\}$.
There is an SRB measure $\nu$ on $\Delta$ such that $(F,\Delta,\nu)$ is an extension by $\mathfrak p:\Delta\rightarrow\mathcal M$ of the initial system $(f,\mathcal M,\mu)$ with $\mathfrak p(x,l)=f^l(x)$. We consider also the quotient
tower $(\bar F,\bar\Delta,\bar\nu)$ obtained from the hyperbolic tower $(F,\Delta,\nu)$
by quotienting along stable curves, which is a factor of $(F,\Delta,\nu)$ by $\bar{\mathfrak p}:\Delta\rightarrow\bar\Delta$ (which is a projection along the $\gamma^s$ on $\gamma_0^u\times\mathbb N$ where $\gamma_0^u$ is some fixed unstable variety). 
We set $\bar\Delta_l:=\bar{\mathfrak p}(\Delta_l)$ for the $l$-th level of the quotient tower $\bar\Delta$.
Recall that each $\bar\Delta_l$
admits an at most numerable partition in $\{\bar\Delta_{l,j};\ j\}$ (such that $\{\bar F^{-1}(\bar\Delta_{l+1,j});\ j\}$ is finer than
$\{\bar F^{-1}(\bar\Delta_{j+1})\cap \bar\Delta_{l,j};\ j\}$).

Let us consider the separation time 
$\hat s(x,y)$ corresponding to the infimum of $n$ such that $\bar F^n(x)$
and $\bar F^n(y)$ do not belong to the same atom of the partition $\{\bar\Delta_{l,j};\ l,j\}$. 
Recall that these systems fit the general scheme of Section~\ref{sechyp} 
(in particular, Assumptions $(H0)$ and $(H1)$ therein).

The following family of complex Banach spaces of functions defined on $\bar\Delta$ is used in \cite{Y}:
$$\cB_{\beta,\eps} =\{h:\bar{\Delta} \to \complex\, |\, \|h\|_{\beta,\eps}<\infty\}\, .$$
This family is labeled by  $(\beta,\eps)\in(0,1)\times[0,\infty)$; where $\|\cdot\|_{\beta,\eps}$ is defined by $\|h\|_{\beta,\eps}=| h |_{\eps,\infty}+|h|_{\beta,\eps,{\text{Lip}}}$ with 
\begin{align}\label{defnormYoung}
|h|_{\eps,\infty}: = \sup_{l} e^{-l\epsilon}\|h|_{\bar{\Delta}_{l}}\|_{\infty}\,\,\, \text{and}\,\,\, |h|_{\beta,\eps,\text{Lip}} = \sup_{l,j} e^{-l\epsilon} \sup_{y,y'\in \bar{\Delta}_{l,j}}\frac{|h(y)-h(y')|}{\beta^{\hat s(y,y')}}\, .
\end{align}
We define the height function $\ell:\bar\Delta\rightarrow\mathbb N$ given by
\[\ell(x,l)=l\, .\]
This function will play an important role in our exposition. Observe that $\mathcal B_{\beta,0}$ corresponds to the set $\mathcal B_\beta$
of Lipschitz functions considered in Section~\ref{sechyp} and that $\mathcal B_{\beta,\eps}$ corresponds to the set of functions $h:\bar\Delta\rightarrow \mathbb C$ such that $e^{-\eps\ell}h$ is in the Lipschitz space $\mathcal B_{\beta}=\mathcal B_{\beta,0}$.

Due to \cite[Lemmas 1 and 2]{Y}, there exist $\beta_0\in (0,1)$ and a function $\bar g:\bar\Delta\rightarrow\mathbb R$ which is null outside $\bar\Delta_0$ such that~\eqref{formuleOp} (of Assumption (H1)) holds, i.e.,
\begin{equation}\label{formuleOpYoung}
\bar C_{\bar g}:=\sup_{x,y:\, \hat s(x,y)>1}\frac{|e^{\bar g(y)-\bar g(x)}-1|}{\beta_0^{\hat s(x,y)}}<\infty\quad\mbox{and}
\quad \mathcal L\psi(x)=\sum_{z\in \bar F^{-1}(x)}e^{\bar g(z)}\psi(z)\, ,
\end{equation}
with $\cL$ the transfer operator associated to
$(\bar F,\bar\Delta,\bar \nu)$ (see \cite[Lemma 2 and Section 3.2]{Y}).
\eqref{bijpreimages} of Assumption $(H1)$ also holds.

Moreover, there exists $c_1>0$ and $\vartheta_1\in (0,1)$ such that
\begin{equation}
\bar\nu(R\ge j)\le  c_1\vartheta_1^j\, .
\end{equation}
Note that the first part of \eqref{formuleOpYoung} implies that
\begin{equation}\label{regulg}
\hat s(x,y)>1\quad\Rightarrow\quad    e^{-\bar C_{\bar g}\beta_0^{\hat s(x,y)}}\leq e^{\bar g(x)-\bar g(y)}\leq  e^{\bar C_{\bar g}\beta_0^{\hat s(x,y)}}\, .
\end{equation}
Also there exists $\eps_0>0$ such that the transfer operator $\mathcal L$ of the quotient Young tower $(\bar F,\bar\Delta,\bar\nu)$ is quasicompact on $\cB_{\beta,\varepsilon}$ for every $\varepsilon\in (0,\varepsilon_0]$ and every $\beta\in [ \beta_0,1)$. 

The assumptions made by Young on $\eps_0$ can be expressed as follows in terms
of the height function $\ell$: first,
$e^{\eps_0\ell}\in L^2(\bar \nu)$, and second,
$ \bar m(e^\ell\mathbf 1_{\bar F^{-1}(\Delta_0)})\le 2\bar m (\bar\Delta_0)$
where $\bar m$ is the reference measure used by Young and which is equivalent to $\bar\nu$ with a density 
taking values in a compact subset of $(0,\infty)$.

Here, we relax these two conditions as well as Young's assumption that $R\ge N$
for some $N$ large enough
by assuming instead that
\begin{equation}\label{integrabilityeeps}
e^{\eps_0\ell(\cdot)}\in L^1(\bar \nu)\, ,
\end{equation}
which ensures that
$\mathcal B_{\beta,\eps}\hookrightarrow L^1(\bar\nu)$ for any $\eps\in[0,\eps_0]$.
Moreover, we assume from now on that $\text{gcd}(r_i)=1$, which will imply that 1 is the single dominating eigenvalue of $\mathcal L$ and that it has multiplicity 1.

Below, for $\bar\phi:\bar \Delta\rightarrow \mathbb C$, we set $\widetilde\phi:\Lambda\times\{0\}\rightarrow\mathbb C$ defined by $\widetilde\phi(x)=\sum_{k=0}^{R(x)-1}\bar\phi\,\circ\bar F^k(x)$.

\subsection{Main results}
Our goal is to study the case of possibly unbounded observables $\phi$.
We first state a result in the case of the expanding Young towers for possibly unbounded observables satisfying integrability conditions close 
to the optimal ones in the i.i.d.~setting.

The more natural unbounded observables $\bar\phi:\bar\Delta\rightarrow\mathbb R$ are the ones belonging to a Young space $\mathcal B_{\beta_1,\eps_1}$. In order to study these observables, it is natural to define the following new spaces (see the begining of Section~\ref{proofThmTower1}) that generalize the Young spaces:
$$\cB_{\beta,\eps,\eps'} :=\{h:\bar{\Delta} \to \complex\, |\, \|h\|_{\beta,\eps,\eps'}<\infty\}$$
where $\|\cdot\|_{\beta,\eps,\eps'}$ is defined by $\|h\|_{\beta,\eps}=
| h |_{\eps,\infty}+|h|_{\beta,\eps',{\text{Lip}}}$, so that $\cB_{\beta,\eps,\eps}$ coincide with
$\cB_{\beta,\eps}$.

\begin{thm}\label{EdgeExpforExpTowers}
Consider the expanding Young tower system $(\bar F,\bar\Delta,\bar\nu)$
with $\beta_0,\eps_0$ satisfying~\eqref{formuleOpYoung} and \eqref{integrabilityeeps}. Suppose this tower has exponential tails and satisfies $\emph{gcd}(r_i) =1$. 
Let $\bar\phi\in\mathcal B_{\beta_1,\eps_1,\eps'_1}$ be $\mathbb X$-valued  
and $\bar\nu$-centered with $\max(\beta_0,\beta_1)<e^{-\eps'_1},\, \eps_1,\eps'_1,\eps\ge 0,\,\gamma\in (0,1),$ $
\eps+(r+2)\eps_1+2\gamma\eps'_1<\eps_0$ and $\beta_0\leq \beta_1^\gamma\le e^{-\eps'_1}$. 
Set $\bar S_n:=\sum_{k=0}^{n-1}\bar\phi\circ\bar F^k$

\begin{itemize}[leftmargin=*]
\item If $\bbX=\reals$ and $\bar \phi$ non-arithmetic, for any probability measure $\mathbb P$ absolutely continuous with respect to $\bar\nu$ $\psi$ having density $\psi \in \mathcal B_{\beta_1^\gamma,
\eps,\eps+\gamma\eps'_1},$ then the order $1$ Edgeworth expansion for $\bar S_n$ exists.

\item  If $\bbX=\reals,$ $\eps'_1<\eps_0,$ $(\Lambda,\widetilde F,\widetilde\phi,R,\beta_1^\gamma)$ satisfies \emph{AAE}  
and $\alpha>\alpha'$ $($with $\alpha'$ as in Proposition~\ref{Lisn1-C}$)$, or if $\bbX=\integers,$  then for all
$\psi\in\mathcal B_{\beta_1^\gamma,
\eps,\eps+\gamma\eps'_1}$ and for all $\xi\in L^{\frac{\eps_0}{\eps_0-\eps-(r+2)\eps_1}+\eta}(\bar\nu)$ $($for some $\eta>0)$ and 
for all $g \in \fF^{q+2}_{0}$ with $q>(2\alpha+1)\big(1+\frac{r}{2}\big)$,
both local and global expansions in the \emph{MLCLT} of every order exist.
\end{itemize}
\end{thm}

\begin{rem}\label{ClosetoIID}
Observe that the condition on $\bar\phi$ of Theorem~\ref{EdgeExpforExpTowers} is very close to the optimal moment condition in the i.i.d.~setting $($existence of a moment of order $r+2)$.

First, since $\eps_0$ satisfies $e^{\eps_0\ell(\cdot)}\in L^1(\bar\nu)$, 
$\mathcal B_{\beta_1,\eps_1,\eps'_1}\subset L^{\frac{\eps_0}{\eps_1}}(\bar\nu)$.
The condition $ (r+2)\eps_1+2\gamma\eps'_1< \eps_0$ implies that $\eps_0/\eps_1>(r+2)+2\gamma\eps'_1/\eps_1$.

Second, suppose $e^\ell$ admits a critical integrability order $k_0,$ i.e.$,$ $e^{\ell}$ admits moment of every order strictly smaller than $k_0$ and no moment of order strictly larger than $k_0$. Such a $k_0$ exists for example if $e^{-bm}<\bar\nu_{\Lambda}(R>m)<e^{-am}$ for some positive $a,b$. Then if
$(r+2)\eps_1+2\gamma\eps'_1<k_0<(r+2+\theta)\eps_1$ $($we can take for example $\eps'_1=0)$, the space $\mathcal B_{\beta_1,\eps_1,\eps'_1}$ is contained in $L^{r+2}(\bar\nu)$ but not in $L^{r+2+\theta}(\bar\nu)$.
\end{rem}

Now, we state an analogous result for the initial map $(f,\mathcal M,\mu)$. Recall that, by the Young tower construction, a function $\phi:\mathcal M \to \reals$ being H\"older continuous, translates to $\bar\phi \in \cB_{\sqrt{\beta},0}$ and  $\phi\,\circ\,\mathfrak p \in \mathcal B^{(0)}_\beta$ 
for some $\beta \in (0,1)$ (see \eqref{tildeB}).
In the next result, we allow $\phi\circ\mathfrak p$ to be
in the set $\mathcal V^{(0)}_{\beta,\eps_1}$
of possibly unbounded functions $h:\Delta\rightarrow\mathbb C$ such that 
$he^{-\ell\eps_1}$ belongs to the space $\mathcal B^{(0)}_\beta$ defined by \eqref{tildeB}.
As we will see it later, this holds true for example when $(f,\mathcal M,\mu)$ is the Sinai billiard and $\phi(x)=h(x)d(x,S_0)^{-\alpha}$
for small $\alpha$, where $S_0$ is the set of unit vectors tangent to the boundary of the billiard domain and where $h:\mathcal M\rightarrow\mathbb R$ is Lipschitz continuous. \vspace{10pt}

\begin{thm}\label{EdgeExpforTowers}
Suppose $(f,\cM,\mu)$ is a nonuniformly hyperbolic map 
modeled as in \cite{Y} by
a Young tower with exponential tails, $\emph{gcd}(r_i) =1$. Let 
$\phi,\,\psi,\,\xi : \mathcal M\rightarrow\mathbb R$ with $\phi:\mathcal M\rightarrow\mathbb X$ non-arithmetic  and $\mu$-centered. Set $S_n:=\sum_{k=0}^{n-1}\phi\circ f^k$. Assume
\begin{itemize}[leftmargin=20pt]
\item either that $\phi,\psi,\xi$ are H\"older$,$
\item or that there exist $\beta,\gamma\in (0,1)$ and some 
$\eps_1\ge 0$ such that
$\max(\beta_0,
\sqrt{\beta })<e^{-\eps_1}$, $
(r+2+2\gamma)\eps_1<\eps_0$ and $\beta_0\leq 
\beta^{\frac \gamma 2}
\le
e^{-\eps_1}
$ such that
$$\phi\,\circ\,\mathfrak p\in \mathcal V_{\beta,\eps_1}^{(0)},\,\,\, \psi\,\circ\,\mathfrak p\in\mathcal V^{(0)}_{
\beta^{\frac\gamma 2},\eps_3},\,\,\, 
\xi\circ\mathfrak p\in\mathcal V^{(0)}_{
\beta^{\frac\gamma 2},\eps_2}\, ,$$ 
where $\eps_2+\eps_3+(r+2+2\gamma)\eps_1< \eps_0$.
\end{itemize}
Then$,$ for any probability measure $\mathbb P$ absolutely continuous with respect to $\bar\nu$ $\psi$ having density $\psi$, the order $1$ Edgeworth expansion for $S_n$ exists.

If, in addition, Assumption $(C)$ holds true $($this is true with $\alpha_1=1,$ for example$,$ if $\mathbb X=\mathbb Z$ $($while dropping the assumption of $\bar \phi$ non-arithmaticity$),$ or if $(\Lambda,\widetilde F,\widetilde\phi,R,\beta^{\frac\gamma 2})$ is \emph{AAE} and $\alpha>\alpha'$ with $\alpha'$ as in Proposition~\ref{Lisn1-C}$),$ then 
for all $g \in \fF^{q+2}_{0}$ where $q>(2\alpha+1)\big(1+\frac{r}{2\alpha_1}\big)$ both local and global expansions in the \emph{MLCLT} of every order exist.
\end{thm}

\subsection{An application to Sinai-Billiards}
For completness, we provide an illustration of the above result in the context of the Sinai billiard with finite horizon. Consider a finite family $O_1,...,O_I$ of $I$ open convex sets in the torus $\mathbb T^2=\mathbb R^2/\mathbb Z^2$, with pairwise disjoint closures and with boundary $C^3$-smooth, with non null curvature.
The $O_i$'s are called obstacles.
We consider the billiard domain given by $Q:=\mathbb T^2\setminus\bigcup_{i=1}^IO_i$. We assume the horizon of the billiard to be finite, which means that the projection on $\mathbb T^2$ of every line in $\mathbb R^2$ intersects at least one obstacle $O_i$.
The space $\mathcal M$ of configurations of the billard system is the set of $(q,\vec v)\in Q\times S^1$ where $\vec v$ is a post-collisional vector, i.e. is a vector pointing inward into $Q$, i.e. $\vec v$ is such that $\langle \vec n_q,\vec v\rangle\ge 0$ where
$\vec n_q$ is the unit vector normal to $\partial Q$ at $q$ directed inward into $Q$.
The map $f:\mathcal M\rightarrow\mathcal M$ maps a post-collisional vector $(q,\vec v)\in\mathcal M$ to the post-collisional vector at the first collision time $s>0$ such that $(q+s\vec v)\in\partial Q$. This maps preserves the probability measure $\mu$ absolutely continuous with respect to the Lebesgue measure, with density $(q,\vec v)\mapsto \frac{\langle \vec n_q,\vec v\rangle}{2|\partial Q|}$.

Let us write $S_0$ for the set of $(q,\vec v)\in \mathcal M$ corresponding to 
vectors tangent to $\partial Q$, i.e. $S_0=\{(q,\vec v)\in\mathcal M\, :\, \langle \vec n_q,\vec v\rangle=0\}$. Observe that $f$ is discontinuous at points of $f^{-1}(S_0)$.
The presence of these discontinuities and the fact that the differential explodes at these points complicate seriously the study of this system. Nevertheless
$f^n$ is $C^1$ from $\mathcal M\setminus \bigcup_{k=0}^nf^{-k}(S_0)$ to  $\mathcal M\setminus \bigcup_{k=0}^nf^{k}(S_0)$. 

The ergodicity of the Sinai billiard $(f,\mathcal M,\mu)$ has been proved by Sinai in \cite{Sin70}.
Since this seminal work, further stochastic properties of this system have been studied. Central limit theorems have been proved in \cite{BS81,BCS91}. Exponential rate of mixing for H\"older observables has been proved in \cite{Y}. Using the tower constructed by Young in \cite{Y} to model the Sinai billiard, Sz\'asz and Varj\'u established the local limit theorem in \cite{SV}. We refer to \cite{ChernovMarkarian}
for a general reference on billiard systems.
\begin{exmp}
Assume $(f,\mathcal M,\mu)$ is the Sinai billiard system with finite horizon as described above.
Let $\phi:\mathcal M\rightarrow \mathbb R$ $\mu$-centered and non-arithmetic be such that
\begin{equation}\label{hypphibill}
|\phi(x)|\le C(d(x,S_0))^{-\alpha}\ \mbox{and}\ 
     |\phi(x)-\phi(y)|\le Cd(x,y)\max(d(x,S_0)^{-1-\alpha},d(y,S_0)^{-1-\alpha}).
\end{equation}
For any integer $r\ge 0$, if $\alpha$  is small enough, then the hypotheses of Theorem~\ref{EdgeExpforTowers} $($except maybe Assumption $(C))$ hold true for some parameters $\beta,\gamma,\eps_1$ and so the corresponding conclusions.
\end{exmp}
Observe that, if $h$ is Lipschitz continuous, then $\phi(\cdot)=(d(\cdot,S_0))^{-\alpha}h(\cdot)$ satisfies \eqref{hypphibill}.
\begin{proof}
Recall, from \cite{Y}, that, for every $x,y\in\Delta$, 
$\hat s_0(x,y)\le s_1(\mathfrak p(x),\mathfrak p(y))$ where $\hat s_0$ has been defined in \eqref{hats0} and 
where $s_1(x,y)$ is the infimum of the integers $n\ge 0$ such that $x$ and $y$ do not lie in the same connected component of $\mathcal M\setminus\bigcup_{k=-n}^{n}f^{-k}(S_0)$. 
By hyperbolicity of $f$ (see for example \cite{Sin70,ChernovMarkarian}), there exist $\widetilde c_0>0$ and $\widetilde \vartheta_0\in (0,1)$ such that, for all $x,y\in
\mathcal M$,
$d(x,y)\le  \widetilde c_0\widetilde\vartheta_0^{s_1(x,y)}$.
Finally, we recall that there exists $\widetilde\vartheta$ such that for every $\widetilde\vartheta_1\in (1,\widetilde\vartheta]$, Young constructed in \cite[Section 8.4]{Y} a tower 
$(F,\Delta,\nu)$ with base $\Lambda\subset\cM$ 
on which 
$d(f^{\pm
n}(\cdot),S_0)\ge\widetilde c_1\widetilde\vartheta_1^{-n}$ for some $\widetilde c_1>0$.

Now, let us study $\phi\circ\mathfrak p$. Observe first that for any $(x,l)\in\Delta$,
\[
\phi(\mathfrak p(x,l))=\phi (f^lx)\le C(d(f^l(x),S_0))^{-\alpha}\le
C\left(\widetilde c_1\widetilde\vartheta_1^{-l}\right)^{-\alpha}\le 
C\widetilde c_1^{-\alpha}\widetilde\vartheta_1^{l\alpha}\, .
\]
Moreover, for every $u\in(0,1]$ and $x,y\in\Delta$ such that $\hat s_0(x,y)\ge 1$,
\begin{align*}
    |\phi(\mathfrak p(x))-\phi(\mathfrak p(y))|&\le C\min\left(2 \max_{z\in\{x,y\}}d(\mathfrak p(z),S_0))^{-\alpha},d(\mathfrak p(x),\mathfrak p(y))\max_{z\in\{x,y\}}d(\mathfrak p(z),S_0)^{-1-\alpha}\right)\\
&\le  C 2^{1-u} \max_{z\in\{x,y\}}d(\mathfrak p(z),S_0))^{-u-\alpha} d(\mathfrak p(x),\mathfrak p(y))^{u}\\
&\le  C 2^{1-u} \left(\widetilde c_1\widetilde\vartheta_1^{-\ell(x)}\right)^{-u-\alpha} \widetilde c_0^u\widetilde\vartheta_0^{u\, s_1(\mathfrak p(x),\mathfrak p(y))}\\
&\le C 2^{1-u} \left(\widetilde c_1\widetilde\vartheta_1^{-\ell(x)}\right)^{-u-\alpha} \widetilde c_0^u\widetilde\vartheta_0^{u\hat s_0(x,y)}\, .
\end{align*}
Thus the function $\phi\circ\mathfrak p$ is in the space $\mathcal V^{(0)}_{
\beta,
\eps_1}$ with 
$
\beta=\widetilde\vartheta_0^u$ and $e^{
\eps_1}=\widetilde\vartheta_1^{u+\alpha}$.

It remains to prove that, if $\alpha$ is small enough, the conditions on $\beta,\eps_1,\gamma$ of Theorem~\ref{EdgeExpforTowers}  are satisfied for a good choice of tower and an appropriate adjustment of parameters. 
Let us recall the dependences between $\widetilde\vartheta_0,\widetilde\vartheta_1,\eps_0,\beta_0$: $\widetilde\vartheta_0$ and $\beta_0$ are related to the billiard system $(f,\mathcal M,\mu)$, $\widetilde\vartheta_1>1$ can be taken as close to 1 as we wish; the tower $(F,\Delta,\nu)$ and so $\eps_0$ depend on the choice of $\widetilde\vartheta_1$.
Let $\gamma\in(0,1]$ small. 
We fix $\widetilde\vartheta_1\in(1,\widetilde\vartheta]$ so that $\widetilde\vartheta_0^{\frac\gamma 2}\widetilde\vartheta_1^{r+2+\gamma}<1$ (which implies in particular that
$\widetilde\vartheta_0\widetilde\vartheta_1^2<1$).
Next choose $u\in(0,1]$ small enough so that $\beta_0\le\sqrt{\beta}=\widetilde\vartheta_0^{\frac u2}$ and $e^{\eps_0}> \widetilde\vartheta_1^{(r+2+2\gamma)u}$ (this is true for example if $e^{\eps_0}> \widetilde\vartheta_0^{-\frac{(r+2+2\gamma)u}2}$). 
Assume $\alpha>0$ is small enough so that
\[
\beta^{\frac\gamma 2}e^{\eps_1(r+2+\gamma)}=\sqrt{\widetilde\vartheta_0^{\gamma u}\widetilde\vartheta_1^{2(u+\alpha)(r+2+\gamma)}}<1
\, ,
\]
\[
\beta^{\frac\gamma 2}=\widetilde\vartheta_0^{\frac {u\gamma}2}<\widetilde\vartheta_1^{-(u+\alpha)}=e^{-\eps_1}
\, ,
\]
\[
e^{\eps_0}> \widetilde\vartheta_1^{(r+2+2\gamma)(u+\alpha)}\quad\mbox{so that }
(r+2+2\gamma)\eps_1<\eps_0\, .
\]
This is possible by continuity since all these inequalities holds true for $\alpha=0$.
\end{proof}

\subsection{Proof of Theorem~\ref{EdgeExpforTowers} when $\bar\phi,\xi,\psi$ are H\"older}
We will check our Assumptions of Section~\ref{se:assum} with $p_0=\infty$ and
$\cX_a=\cX_a^{(+)}=\mathcal B_{\beta^{\frac\gamma2},\eps}$ (for some $\gamma\in(0,1)$ and $\eps<\eps_0$) and $\mathcal B_1=\mathcal B_{\sqrt{\beta},0}$
and $\cB_2=L^{p_1}(\bar\nu)$. Assumptions $(H0)$ and $(H1)$ of Section~\ref{sechyp} are satisfied (see \cite{Y} for details). Thus, due to Lemma~\ref{LEM0}, Assumption $(A)[r](1,2)$ holds true for every $r\ge 0$.

The fact that the $\cL_{is}$ have essential spectral radius strictly smaller than 1
follows from Proposition~\ref{DF} combined with Lemmas~\ref{DFHP} and~\ref{compactinclusion} applied with $\eps_1=\eps'_1=0$. Along with exactness and Lemma~\ref{periphericalspect}, this implies that Assumption $(B)$ holds true.
The fact that Assumption $A[r](3)$ holds true for any $r\ge 0$ is due to Lemma~\ref{lemA_B} combined with $\bar\phi\in\mathcal B_{\beta^{\frac\gamma 2},0}$ and from the following fact
\[
\forall h\in
\mathcal B_{\beta^{\frac\gamma 2},0},\ \forall g\in
\mathcal B_{\beta^{\frac\gamma 2},\eps},\quad
\Vert gh\Vert_{\beta^{\frac\gamma 2},\eps}\le \Vert h\Vert_{\beta^{\frac\gamma 2},0}\Vert g\Vert_{\beta^{\frac\gamma 2},\eps}\, .
\]
Assumption $(C)$ follows from AAE for $(\Lambda,\widetilde F,\widetilde\phi,R,\beta^{\frac\gamma 2})$ if $\mathbb X=\mathbb R$. Apply Proposition~\ref{Lisn1-C} and Lemma~\ref{BaseToTower} with $\varsigma=\varsigma_0=\mathbf 1$ since $\big|(\bar\phi\circ\bar F^k)|_{Y_j}\big|_{\sqrt{\beta}}\le|\bar\phi|_{\sqrt{\beta},0,\text{Lip}}\beta^{r_{j}-k-1}$ and  $\sum_j\bar\nu_Y(Y_j)r_j<\infty$.


\subsection{Proof of Theorem~\ref{EdgeExpforExpTowers}}\label{proofThmTower1}
Before stating the proof of the theorem, let us make a few remarks. We assume for the moment that $\bar\phi\in\mathcal B_{\beta_1,\eps_1}$ with
$\eps_1>0$. 

The first difficulty in the study of $\cL_{is}=\cL(e^{is\bar\phi}\,\cdot\,)$ is that the multiplication by $e^{is\bar\phi}\in\mathcal B_{\beta_1,0,\eps_1}$ does not preserve the spaces $\cB_{\beta_1,\eps_1}$. Indeed it maps $\cB_{\beta_1,\eps}$ into $\cB_{\beta_1,\eps,\eps+\eps_1}$
and more generally maps $\cB_{\beta_1^\gamma,\eps}$ into $\cB_{\beta_1^\gamma,\eps,\eps+\gamma\eps_1}$ for any $\gamma\in(0,1]$  (since
 $e^{is\bar\phi}\in\mathcal B_{\beta_1,0,\eps_1}\subset  \mathcal B_{\beta_1^\gamma,0,\gamma\eps_1}$).
This remark has led us to the introduction of these new Young spaces with three parameters. This first difficulty is solved by noticing that the multiplication by $e^{is\bar\phi}$
and then $\cL_{is}$ acts continuously on $\cB_{\beta_1^\gamma,\eps,\eps+\gamma\eps_1}$. 

The second difficulty is that $s\mapsto \cL_{is}$ is not continuous from $\mathbb R$ to $\mathcal L(\cB_{\beta_1^\gamma,\eps,\eps+\gamma\eps_1})$. But it is continuous from $\mathbb R$ to $\mathcal L(\mathcal B_{\beta_1^\gamma,\eps,\eps+\gamma\eps_1},\mathcal B_{\beta_1^\gamma,\eps+\eps'',\eps+\gamma\eps_1+\eps''})$ (for $\eps''>0$), and more generally, it is $C^r$  from $\mathbb R$ to $\mathcal L(\mathcal B_{\beta_1^\gamma,\eps,\eps+\gamma\eps_1},\mathcal B_{\beta_1^\gamma,\eps+r\eps_1+\eps'',\eps+r\eps_1+\gamma\eps_1+\eps''})$  (for $\eps''>0$). So we cannot hope our Assumption $(A)$ to be true 
with a single space.
However, our Assumption $(A)$ is true with a double chain of spaces.
Let us assume for the moment the following result, the proof of which is provided in Appendix~\ref{appendYoung}.

\begin{prop}\label{propCond3'}
Let $(\bar F,\bar\Delta,\bar\nu)$ and $\beta_0,\eps_0$ be as above.
Let $\bar\phi\in\mathcal B_{\beta_1,\eps_1,\eps'_1}$ with $\max(\beta_0,\beta_1)<e^{-\eps'_1}$ and $0<\eps_1$.
Let $\eps,\eps''>0$ and $\gamma\in(0,1)$  such that $\eps+(r+2)\eps_1+2\gamma\eps'_1+\eps''\le\eps_0$ and $\beta_0\leq \beta_1^\gamma< e^{-\eps'_1}$.
Then Assumptions $(A)[r](3,4)$ and $(B)$ hold true with 
$\cX_j=\cV^{(\eps)}_{a_j}$ and $\cX_j^{(+)}=\cV^{(\eps)}_{b_j}$ and
$p_0:=\frac{\eps_0}{\eps_0-\gamma\eps'_1},$ where
we set
$$\mathcal V^{(\eps)}_\theta:=\mathcal B_{\beta_1^\gamma,\eps+\theta\eps_1,\eps+\theta\eps_1+\gamma\eps'_1},\,\,\,a_k:=k\left(1+\frac{\eps''}{(r+3)\eps_1}\right)\,\,\,\text{and}\,\,\, b_k:=a_k+\frac{\eps''}{2(r+3)\eps_1}. $$
Moreover, the $\mathcal L_{is}$ are quasicompact on the $\cV_\theta^{(\eps)}$for every $\theta\in[0,r+2+\frac{\eps''}{\eps_1}]$.
\end{prop}

\begin{proof}[Proof of Theorem~\ref{EdgeExpforExpTowers}]
Since $\mathfrak p=\bar{\mathfrak p}=\text{Id}$, Assumption $(A)(1)$ is automatic with 
$\chi=0$. Thus $h^{(j)}_{k,s,H}=H\circ \bar F^k$ and 
 we can take $k=\vartheta^k=0$
(the term appearing in \eqref{hk1} is now null and the term appearing in \eqref{hk2} is dominated by a constant).
Our assumptions ensure that 
$\psi\in \cX_{0} $ and 
$\xi\in(\cX_{r+2}^{(+)})'$ if $\eps''$ is small enough since
$\cX_{r+2}^{(+)}\subset L^{\frac{\eps_0}{\eps+(r+2)\eps_1+\eps''}}(\bar\nu)$.
This combined with Proposition~\ref{propCond3'} ends the proof of Assumptions $(A)$
and $(B)$.

Assume now that $\mathbb X=\mathbb R$ and let us prove that
Assumption $(C)$ holds true with $\alpha_1=1$, $\cB_1=\cX_0$ and $\cB_2=L^{p_1}(\bar\nu)$  if $(\Lambda,\widetilde F,\widetilde\phi,R,\beta_1^\gamma)$ is AAE. 
Indeed, $\bar\phi\in\mathcal B_{\beta_1,\eps_1,\eps'_1}$ implies that
$\bar\phi\in L^1(\bar\Delta,\bar\nu)$ (since $\eps_1\le\eps_0$)
 and 
\[
|\widetilde\phi_{|\Lambda_j}|^\gamma_{\beta_1}\le \sum_{k=0}^{r_j-1}|(\bar\phi\circ \bar F^k)_{|\Lambda_j}|^\gamma_{\beta_1}\le 
 \sum_{k=0}^{r_j-1}|\bar\phi|^\gamma_{\beta_1,\eps'_1,\text{Lip}}e^{\gamma k\eps'_1} \beta_1^{\gamma(r_j-1-k)}
\le
 \frac{\beta_1^{-\gamma}e^{\gamma\eps'_1 r_j}}{e^{\gamma\eps'_1}\beta_1^{-\gamma}-1}|\bar\phi|^\gamma_{\beta_1,\eps'_1,\text{Lip}}
\]
and $e^{\gamma\eps'_1R}\in L^1(\Lambda,\bar\nu(\cdot|\Lambda))$ (since $\gamma \eps'_1\le\eps_0$).
This allows us to apply Proposition~\ref{Lisn1-C}, which combined with Lemma~\ref{DIneqImpliesD}, ensures (with the notations of Section~\ref{sec:cond(C)}) that 
$
\Vert \widetilde \cL_{is,iu}^n\Vert_{\widetilde\cB\rightarrow L^\infty}\le C''e^{-\frac{Cn}{|s|^\alpha \log s}}$.
Finally we conclude by applying Lemma~\ref{BaseToTower}
with $\varsigma(l)=e^{\eps l}$ $\varsigma_0(l)=e^{(\eps+\gamma\eps'_1)l}$ (so that $\cZ=\cX_0$), $p_0=\frac{\eps_0}{\eps}>p_1$, $\gamma$ and $\beta=\beta_1^{\gamma}$. 
Indeed
the quantity appearing in \eqref{hypophibasetotower} is bounded by
\[
|\bar\phi|^\gamma_{\beta_1,\eps'_1,\text{Lip}}
\sum_{j}\bar\nu_Y(Y_j)
\sum_{k=0}^{r_{j}-1}k e^{\eps k} 
e^{\gamma\eps'_1 k}\beta_1^{r_j-k-1}\le 
|\bar\phi|_{\beta_1,\eps'_1,\text{Lip}}
\mathbb E_{\bar\nu}(\ell e^{(\eps+\gamma\eps'_1) \ell})
 <\infty\,  ,
\]
since $\eps'_1<\eps_0$. 
Moreover $\varsigma(\ell)\in L^{\frac{\eps_0}{\eps_1}}(\bar\nu)$
and $\varsigma_0(\ell)\in L^{\frac{\eps_0}{\eps+\gamma\eps'_1}}(\bar\nu)\subset L^1(\bar\nu)$
since $\eps+\gamma\eps'_1\le\eps_0$. 
Our assumption on $\xi$ ensures that $\xi\in L^{\frac{\eps_0}{\eps_0-\eps_1}}(\bar\nu)$.
Thus Assumption $(C)$ holds true for some 
and $\alpha_1=1$. We conclude by Lemma~\ref{DIneqImpliesD} that  Assumption $(C)$ holds true with any $\delta>0$.
\end{proof}

\subsection{Proof of Theorem~\ref{EdgeExpforTowers} when $\bar\phi$ is not H\"older}
Let $\phi\in\mathcal V_{\beta,\eps_1}^{(0)}$. Due to Lemma~\ref{LEM0}
applied to $\phi e^{-\ell\eps_1}$, there exist two functions $\bar\phi\in\mathcal B_{\beta_1:=\sqrt{\beta},\eps_1,\eps_1}$ and 
$\chi\in\cV^{(0)}_{\sqrt{\beta},\eps_1} $
such that $\phi-\chi+\chi\circ F=\bar\phi\circ\bar{\mathfrak p}$.
Observe that both $\chi$ and $\bar\phi$ are dominated by a constant times $e^{\eps_1\ell}$ which is in $L^{\frac{\eps_0}{\eps_1}}(\bar\nu)$. Thus Assumption $(A)(1)$
hods true with $r_0:=\frac{\eps_0}{\eps_1}-2>r+2\gamma$.

In view of Condition $(A)[r]$ with  $p_0:=\frac{\eps_0}{\eps_0-\gamma\eps_1}$, we set $\beta_1=\sqrt{\beta}$. Take
$\eps,\eps''>0$ such that $\eps_3\le\eps\le \eps_0-\eps_2-(r+2+2\gamma)\eps_1$ and $\eps''\le\eps_0-\eps-\eps_2-(r+2+2\gamma)\eps_1$ and $\beta_0\leq \beta_1^\gamma< e^{-\eps_1}$. 
We set $$ \left(\mathcal{V}^{(\eps)}_\theta := \mathcal{B}_{\beta_1^\gamma,\eps+\theta\eps_1,\eps+(\theta+\gamma)\eps_1}\right)_{\theta\in[0,r+2+\frac{\eps''}{\eps_1}]} $$ and $\cX_j=\cV^{(\eps)}_{a_j}$ and 
$\cX_j^{(+)}=\cV^{(\eps)}_{b_j}$
with $$ a_k=k\left(1+\frac{\eps''}{(r+3)\eps_1}\right)\,\,\,\text{and}\,\,\, b_k=a_k+\frac{\eps''}{2(r+3)\eps_1}.$$

Assumption $(A)[r](2)$ follows from Lemma~\ref{CondA2Young}
and from our assumptions on $\phi,\psi,\xi$ setting $q(\xi)=\frac{\eps_2}{\eps_1}+\gamma$, $q(\psi)=\frac{\eps_2}{\eps_3}+\gamma$ and using the fact that $\eps_3\le\eps+\eps''$ and that $\eps_2\le \eps_0-\eps-(r+2)\eps_1-\eps''$.

Assumptions $(A)[r](3,4)$ and $(B)$ come from Proposition~\ref{propCond3'}.

For Assumption $(C)$, we proceed exactly as in the proof of Theorem~\ref{EdgeExpforExpTowers}
since $\bar\phi\in\mathcal B_{\sqrt{\beta},\eps,\eps}$, with
$\cB_1=\cX_0$ and $\cB_2=L^{p_1}(\bar\nu)$  with $p_1\in(1,\frac{\eps_0}{\eps})$.
Indeed $\Vert h_{k,s,\xi}\Vert_{L^{\frac {\eps_0}{\eps_2}}}\le \Vert \xi\Vert_{L^{\frac {\eps_0}{\eps_2}}}<\infty$ and $\eps_2+\eps<\eps_0$
and the estimate on $\psi$ has already been proved (see again \eqref{hk2young2} and \eqref{hk2young1}).

\section{Random Matrix Products}\label{RWalks}

As the last example, we describe briefly how our results apply to random matrix products, and more generally, to random walks on split semisimple Lie groups. The ideas we use are from \cite{BL, GY, LJ}. We refer the readers to those references and the references therein for the historical development of the subject as well as for the complete statements of the results we use and their proofs.

Let $V$ be a $d$-dimensional $\reals$-vector space with $d>1$. Fix a scalar product on $V$ and let the associated norm be $\|\cdot\|$. Write $X:=\Prob V$ for the projective space of $V$ with a suitable Riemannian distance $\textsf{d}(\cdot,\cdot)$ (as introduced in \cite[Chapter II]{BL}). Given $x\in V$ and a sequence $(g_n)_{n\ge 1}$ of i.i.d.~random variables with common distribution $\mu$ and with values in $G:=\text{GL}(V)$ the group of $d\times d$ invertible matrices over $V$, we are interested in the long term behaviour of $(g_n\dots g_1\cdot x)_{n\ge 1}$, and more precisely of $(S_n(x))_{n\ge 1}$ with 
\[
S_n(x):=\log \frac{\|g_n\dots g_1 \cdot x\|}{\|x\|}\, .
\]

A local limit theorem has been established in \cite{GY} under the following 
assumptions.

\begin{itemize}[leftmargin=*]
\item Suppose $\mu$ has an exponential moment, i.e., there exists $\delta>0$ such that 
$$\int_G \max(\|g\|, \|g^{-1}\|)^\delta\, d\mu(g)<\infty.$$
This  implies, among other things, the existence of the two following quantities:
the first Lyapunov exponent defined by
\begin{equation}\label{LypExp}
\lambda_{1} = \lim_{n \to \infty} \frac{1}{n}
\EXP\, [\log \|g_n\dots g_1\|]
\end{equation} 
and the asymptotic variance
\begin{equation*}
\sigma^2= \lim_{n\to\infty} \frac{1}{n}
\EXP\, [(\log \|g_n\dots g_1\| - n\lambda_{1})^2 ]\, .
\end{equation*}
\item Suppose the semigroup generated by $\text{supp}\mu $, $\Gamma_\mu$, is strongly irreducible, i.e., no finite union of proper subspaces is $\Gamma_\mu$-invariant, and contains a proximal element i.e., $g \in G$ such that $g$ has a simple dominant eigenvalue.
\end{itemize}

Recall that the quantity $\lambda_{1}$ given by \eqref{LypExp} is the long term average behaviour of the norm-cocyle in the following sense: 
For all $x\in \Prob V$, $\lim_{n\to\infty}\frac{1}{n}S_n(x)= \lambda_{1}$ almost surely (This convergence happens also in $L^1$ uniformly in $x$). Therefore, in order to consider a(n) (asymptotically) centered observable, we need to replace $\bar \phi$ by $\bar\phi - \lambda_{1}$.

Under the assumptions $\sigma^2>0$, the non-degenerate CLT takes the form:
\begin{equation}\label{CLTforRMP}
\lim_{n\to\infty}\Prob
\left(\frac{1}{\sqrt{n}}\big[S_n(x)-n\lambda_{1}\big]\le z\right)= \frac{1}{\sqrt{2\pi \sigma^2}} \int_{-\infty}^z \exp\left(-\frac{y^2}{2\sigma^2}\right) \, dy =: \fN_\mu(z)
\end{equation}
uniformly in $x\in\Prob V$.

Our next result provides more precise estimates.

\begin{thm}\label{Ord1RMP}
Suppose $\{g_n\}_{n\geq 0}$ is a sequence of i.i.d.~random matrices in \emph{GL}$(V)$ where $V$ is a $d-$dimensional $\reals-$vector space with $d>1$. Suppose the common distribution $\mu$ of the $g_n$ has an exponential moment and that the semigroup generated by $\emph{supp}\mu $ is strongly irreducible and contains a proximal element. Let $x\in \Prob V$. Then there exists a polynomial $P_1$ $($which depends on both $x$ and $\mu)$ such that 
$$\Prob\left(\frac{1}{\sqrt{n}}\Big[S_n(x)-n\lambda_{1}\Big]\le z\right) = \fN_\mu(z) + \fN'_\mu(z)\frac{P_1(z)}{\sqrt{n}}+ o(n^{-1/2})$$
uniformly in $z$.

If moreover $\mu$ is supported and Zariski dense in a connected algebraic subgroup of \emph{GL}$(V)$ and if $\xi: X\rightarrow \mathbb R$
is H\"older continuous, then there exist polynomials $P_k$ $($which depend on both $x$ and $\mu)$ such that 
$$\Prob\left(\frac{S_n(x)-n\lambda_1}{\sqrt{n}}\le z\right) = \fN_\mu(z) + \fN'_\mu(z)\sum_{j=1}^r\frac{P_k(z)}{n^{k/2}}+ o(n^{-r/2})$$
uniformly in $z$, for all $r \geq 0$ and
there exist polynomials $R_k$ and $Q_k$ such that 
$$\EXP\left( g(S_n(x)-n\lambda_1)\, \xi(g_n\dots g_1 \cdot x)\right) =\sum_{j=0}^{r} \frac{1}{n^{j/2}}\int_{\mathbb R}g(z \sqrt{n})\,R_{j}(z)\,\fN'_\mu(z) \, dz+ C^{q+2}(g)\cdot o(n^{-r/2})\, ,$$
for all $g \in \fF^{q+2}_{0}$ where $q>0$
and 
$$\sqrt{n}\EXP\left( g(S_n(x)-n\lambda_1)\, \xi(g_n\dots g_1 \cdot x)\right) =\sum_{j=0}^{[r/2]} \frac{1}{n^{j}}\int_{\mathbb R}g(z)Q_j(z) \, dz+C^{q+2}_{r+1}(g)\cdot o(n^{-r/2})\, ,$$
for all $g \in \fF^{q+2}_{r+1}$ where $q>0$.
\end{thm}

\begin{proof}

Observe that 
\[
\mathbb E\left(e^{is(S_n(x)-n\lambda_1)}\xi(g_n\dots g_1\cdot x)\right)
=\cL_{is}^n(\xi)(x)
\]
with 
\[
\cL_{is}h(x):=\int_Ge^{is(\bar\phi(g, x)-\lambda_1)}h(g\cdot x)\, d\mu(g)
\]
with $\bar\phi(g,x)=\frac{\Vert g\cdot x\Vert}{\Vert x\Vert}$.

This is enough to follow the proofs of our main theorems up to checking 
Assumptions $(\alpha)-(\delta)$ with a single space. This combined with 
$\lim_{n\rightarrow+\infty}\frac 1n\mathbb E[S_n-n\lambda_1]=0$
and the definition of $\sigma^2$ which will lead to the asymptotic expansion of the dominated eigenvalue.

We study the action of the family of operators $(\cL_{is})_{s\in\mathbb R}$
on the space
$(\cB_{\eps}(X),\|\cdot\|_\eps)$ of $\eps$-H\"older continuous functions endowed with the norm $\|\,\cdot\,\|_\eps = \|\,\cdot\,\|_\infty+|\, \cdot \,|_\eps$ where $$| h |_\eps := \sup_{y_1\neq y_2} \frac{|h(y_1)-h(y_2)|}{\textsf{d}(y_1,y_2)^\eps}$$
is the the H\"older constant of $h$. This Banach space is compactly embedded in $C(X) \hookrightarrow L^\infty(X)$. In particular, 
\begin{equation}\label{majoespLis}
\left| \mathbb E\left(e^{is(S_n(x)-n\lambda_1)}\xi(g_n\dots g_1\cdot x)\right)
\right|\le \Vert \cL_{is}^n\Vert_\eps \Vert \xi\Vert_\eps\, .
\end{equation}

From the results in \cite[Section 2]{GY}, 
we have that $s \mapsto \cL_{is}$ is analytic and that there exists $\delta>0$ such that for any $|s|<\delta$,
\[
\cL_{is}^n=\lambda_{is}^n\Pi_{is}+R_{is}^n
\]
with, for all $j$, $\sup_{|s|<\delta}\Vert\lambda_{is}^{-n} (R_{is}^n)^{(j)}\Vert_{\eps}$
vanishing at an exponential rate,
with $\lambda_{is}$ $C^\infty$ with $\lambda_{is}=1-\frac{\sigma^2}2s^2+o(s^2)$ with $\sigma^2>0$.
In particular
\[
\sup_{|s|<\delta}\left|(\lambda_{is}^{-n}\mathbb E\left(e^{is(S_n(x)-n\lambda_1)}\xi(g_n\dots g_1\cdot x)\right))^{(j)}-\Pi_{is}^{(j)}\xi(x)\right|\le  \sup_{|s|<\delta}\Vert (\lambda_{is}^{-n}R_{is}^n)^{(j)}\Vert_{\eps}\Vert \xi\Vert_\eps
\]
This ensures $(\alpha)$ with $B_j=(\Pi_{i 0}^{(j)}(\xi))(x)$ up to
decrease if necessary the value of $\delta$ to get the second bound.
Morever $(\beta)$ is also proved in \cite{GY} via the aperiodicity of $\bar\phi$ ensuring that $\sup_{\delta<s<K}\Vert \cL_{is}^n\Vert_\eps$ vanishes at exponential rate for any $0<\delta<K$. This gives us the first part of the theorem.

The recent work \cite{LJ} obtains an estimate that imply our Assumptions $(\gamma)$ and $(\delta)$ using the techniques introduced by Dolgopyat in \cite{D1}. More precisely, let $K < G=\emph{GL}(V)$ be algebraic and connected, and suppose that $\mu$ is supported and Zariski dense on  $K$.
Then \cite[Theorem 4.19]{LJ} yields that, for $\eps$ small enough, there exists $C,c,K>0$ such that for all $|s|>K$,
$\Vert\cL_{is}^n\Vert_\eps\le C |s|^{2\eps}e^{-cn}$ which gives $(\delta)[r]$ for all $r$ and also, due to Lemma~\ref{DIneqImpliesD}, Assumption $(\gamma)$ holds true with $\alpha=2\eps$ and $\alpha_1=1$. To conclude, take
$\eps$ such that $q>2\eps\left(1+\frac{r+1}2\right)$.
\end{proof}

In fact, one can replace $\bar \phi$ by any non-arithmetic continuous function on $X$, and obtain the first order Edgewroth expansion for its Birkhoff sum. More generally, one can consider any $\mu-$contracting, strongly irreducible and measurable action of a Polish semigroup $G$ on a complete separable compact metric space $(X,\textsf{d})$. Then the assumptions $(\alpha)[r]$ for $\psi=\xi=1$ for all $r$ and Assumption $(\beta)$ hold. See \cite{GY}. Hence, 
the first result \Cref{Ord1RMP} can be further generalized. It should be noted that the recent work \cite{JWZ} proves first order Edgeworth expansions in the case of $GL(V)$ while relaxing the assumption of exponential moments but they do not discuss the more general setting of Polish semigroups. 

Higher order expansions cannot be extended in this manner because the results of \cite{LJ} hold only for the specific choice of the norm cocyle $\bar \phi$. Yet, the results on higher order expansions can be generalised in a different direction to include group actions groups of real points of connected semisimple algebraic groups defined and split over $\reals$. See \cite{BQ} and \cite{LJ} for details in this direction. In order to keep the exposition elementary, we decided to present the results for subgroups of $GL(V)$.

\begin{appendix}
\section{Additional proofs for Young towers}\label{appendYoung}

This appendix contains the technical results and proofs for Young towers. It completes Section~\ref{Towers}.

\subsection{Assumption $(A)(3,4)$ for expanding Young towers}
Here, we focus on the (quotient) expanding Young tower $(\bar F,\bar\Delta,\bar\nu)$ 
as in Section~\ref{contextYoung} along with the notations therein. 
Our goal is to study the family of operators $(\cL_{is}=\cL(e^{is\bar\phi}\cdot))_{s\in\mathbb R}$ when $\bar\phi\in\mathcal B_{\beta_1,\eps_1,\eps'_1}$ is a real valued centered observable. In particular, we will prove the quasi-compactness of these operators on appropriate Banach spaces  thanks to a Doeblin-Fortet inequality (Lemma~\ref{DFHP}) and a compact inclusion (Lemma~\ref{compactinclusion}). We will, moreover, prove the $C^r$ smoothness of $s\mapsto\mathcal L_{is}$ as a function with values on some spaces of the form $\mathcal L(\mathcal Y_{0},\mathcal Y_{1})$, with $\mathcal Y_{0}\ne \mathcal Y_{1}$ (Lemma~\ref{contlinYoung}). We end this section with the proof of  Proposition \ref{propCond3'}.

Let $\beta_0\in (0,1)$ and $\eps_0$ be as in Section~\ref{contextYoung}.
Recall that $\cB_{\beta,\eps,\eps'}$ is the set of $h:\bar{\Delta} \to \complex$
such that $\|h\|_{\beta,\eps,\eps'}<\infty$
where $\|\cdot\|_{\beta,\eps,\eps'}$ is defined by $\|h\|_{\beta,\eps}=
| h |_{\eps,\infty}+|h|_{\beta,\eps',{\text{Lip}}}$. (See~\eqref{formuleOpYoung}, \eqref{integrabilityeeps} and \eqref{defnormYoung}).
Let $\beta_1\in(0,1)$ and  $\eps_1,\eps'_1\in[0,\infty)$.
We consider a real-valued centered observable $\bar\phi \in \mathcal B_{\beta_1,\eps_1,\eps'_1}$. We set $\bar S_n:=\sum_{k=0}^{n-1}\bar\phi\circ\bar F^k$ and $\mathcal L_{is}:=\mathcal L(e^{is\bar\phi}\,\cdot\,)$ for every $s\in\mathbb R$.

Observe that the multiplication by $e^{is\bar\phi}$ does not preserve $\mathcal B_{\beta_1,\eps_1,\eps'_1}$. We will see in the next lemma that this multiplication defines a continuous operator from $\mathcal B_{\beta_1,\eps_1,\eps'_1}$ to $\mathcal B_{\beta_1,\eps_1,\eps_1+\eps'_1}$ and also from $\mathcal B_{\beta_1^\gamma,\eps,\eps}$ to $\mathcal B_{\beta^\gamma_1,\eps,\eps+\gamma\eps'_1}$ for any $\gamma\in(0,1]$ (since $e^{is\bar\phi}\in\mathcal B_{\beta_1,0,\eps'_1}\subset
\mathcal B_{\beta_1^\gamma,0,\gamma\eps'_1}$).

\begin{lem}\label{lem:product}Let $\beta\in(0,1)$ and $\eps,\eps',\eps_2,\eps'_2 \geq 0$ 
If $g\in \cB_{\beta,\eps,\eps'} $ and $h\in \cB_{\beta_2,\eps_2,\eps'_2} $, then $$gh\in \cB_{\max(\beta,\beta_2),\eps+\eps_2,\max(\eps+\eps'_2,\eps'+\eps_2)}$$
and
\[
\|gh\|_{\max(\beta,\beta_2),\eps+\eps_2,\max(\eps+\eps'_2,\eps'+\eps_2)}\leq\|g\|_{\beta,\eps,\eps'}\|h\|_{\beta_2,\eps_2,\eps'_2}\, .
\]
\end{lem}
\begin{proof}
First, observe that
\[
|gh|_{\eps+\eps_2,\infty}: = \sup_{l} e^{-l(\eps+\eps_2)}\|(gh)|_{\bar{\Delta}_{l}}\|_{\infty}\leq \sup_{l} e^{-l\eps}\|g|_{\bar{\Delta}_{l}}\|_{\infty}\sup_{l} e^{-l\eps_2}\|h|_{\bar{\Delta}_{l}}\|_{\infty}\leq 
|g|_{\eps,\infty}|h|_{\eps_2,\infty}\, .
\]
Next, observe that for every $l,j$ and every $y,y'\in \bar{\Delta}_{l,j}$,
\begin{align*}
|(gh)(y)-(gh)(y')|&\leq |g(y)|\, |h(y)-h(y')|+|g(y)-g(y')|\, |h(y')|\\
&\leq |g|_{\eps,\infty}e^{l\eps}|h|_{\beta_2,\eps'_2,\text{Lip}}\beta_2^{\hat s(y,y')}e^{l\eps'_2}+|g|_{\beta,\eps',\text{Lip}}
\beta^{\hat s(y,y')}e^{l\eps'}|h|_{\eps_2,\infty}e^{l\eps_2}\\
&\leq  (|g|_{\eps,\infty}|h|_{\beta_2,\eps'_2,\text{Lip}}
    +|g|_{\beta,\eps',\text{Lip}}   |h|_{\eps_2,\infty} )
    (\max(\beta,\beta_2))^{\hat s(y,y')}e^{l\max(\eps+\eps'_2,\eps'+\eps_2)}
\end{align*}
from which we conclude.
\end{proof}

\begin{lem}\label{continuousaction}
Let $\beta\in(0,1)$ and $\eps,\eps'\in[0,\eps_0]$.
$\cL$ acts continuously on $\cB_{\beta,\eps,\eps'}$.
\end{lem}
\begin{proof}
Assume first that $l\ne 0$. On $\Delta_l$, $\bar F^{-1}$ is well defined by $\bar F^{-1}(x,l)=(x,l-1)$ and $\mathcal L(h)(x,l)=h(x,l-1)$. Therefore
\begin{align}
\big|\cL( h)|_{\bar\Delta_l}\big|_{\eps,\infty} \le    e^{-\eps}\big\Vert h|_{\bar\Delta_{l-1}}\big\Vert_\infty e^{-(l-1)\eps} \le e^{-\eps}|h|_{\eps,\infty}\label{equa1previous}
\end{align}
and
\begin{align}
\big|\cL ( h)|_{\bar\Delta_\ell}\big|_{\beta,\eps',\text{Lip}}
\le 
\beta e^{-\eps'}\sup_{l'}\sup_{y,y'\in\bar\Delta_{l',j}\, :\ \hat s(y,y')>1}\frac{|h(y)-h(y')|}{\beta^{ \hat s(y,y')}}e^{-l'\eps'}
\le   \beta e^{-\eps'}|h|_{\beta,\eps',\text{Lip}}\, .\label{equa1bisprevious}
\end{align}

Now, we study the case $l=0$. It follows from \eqref{formuleOpYoung} combined with \eqref{regulg} that
\begin{align}
\big|\cL ( h)|_{\bar\Delta_{0}}\big|_{\eps,\infty} &\le  \sup_{x\in\bar\Delta_{0}} \sum_{i} |e^{\bar g(\bar F_i^{-1}(x))}h(\bar F_i^{-1}(x))| \nonumber\\
&\le  \sup_{x\in\bar\Delta_{0
}} \sum_{i} e^{\bar g(\bar F_i^{-1}(x))}|h|_{\eps,\infty} e^{\eps(r_i-1)}
\nonumber\\
&\le  |h|_{\eps,\infty}\sup_{x\in\bar\Delta_{0
}} \sum_{i}\frac{e^{\bar C_{\bar g}}}{\bar\nu(\bar\Delta_0)} \int_{\bar\Delta_0}e^{\bar g(\bar F_i^{-1}(y))}\, d\bar\nu(y) e^{\eps(r_i-1)}\label{AAAAAA1previous}
\end{align}
where $\bar F_i $ stands for the restriction of $\bar F$ to $\bar{\mathfrak{p}}(\Lambda_i\times\{r_i-1\}) $.
But 
\[
 \int_{\bar\Delta_0}e^{\bar g(\bar F_i^{-1}(y))}\, d\bar\nu(y)=\int_{\bar\Delta}\cL(\mathbf 1_{\bar F_i^{-1}(\bar\Delta_0)})(y) \, d\bar\nu(y)\\
= \bar\nu(\bar F_{i}^{-1}(\bar\Delta_{0
}))   \, .
\]
Combining this with \eqref{AAAAAA1previous}, we obtain
\begin{align}
\big|\cL (h)|_{\bar\Delta_{0}}|_{\eps,\infty}&\le |h|_{\eps,\infty}\frac{ e^{\bar C_{\bar g}}}{\bar\nu(\Delta_{0})}  \sum_{i}\bar\nu(\bar F_{i}^{-1}(\bar\Delta_{0 })) e^{\eps(r_i-1)}=|h|_{\eps,\infty}\frac{ e^{\bar C_{\bar g}}}{\bar\nu(\Delta_{0})}  \Vert e^{\eps\ell}\Vert_{L^1(\bar\nu)}\ .\label{totoprevious}
\end{align}
Note that $\Vert e^{\eps\ell}\Vert_{L^1(\bar\nu)}$ is finite since
$\eps\le\eps_0$. Finally,
\begin{align}
\big|\cL ( h)|_{\bar\Delta_{0,j}}\big|_{\beta,\eps',\text{Lip}} &\le \sup_{x\in\bar\Delta_{0,j}} \sum_{i} \sup_{y\in \bar F^{-1}_{i}(\bar\Delta_0)}\frac{\left| e^{\bar g(\bar F_i^{-1}(x))}h(\bar F_i^{-1}(x))-e^{\bar g(y)}h(y)\right|}{\beta^{\hat s(F^{-1}_i(x),y)-1}} \nonumber \\
\nonumber&\le  \beta \sup_{x\in\bar\Delta_0} \sum_{i} \left(e^{\bar g(\bar F_i^{-1}(x))}|h|_{\beta,\eps',\text{Lip}}e^{\eps'(r_i-1)}
+e^{\bar g(\bar F_i^{-1}(x))}e^{\bar C_{\bar g}} |h|_{\eps,\infty}e^{\eps(r_i-1)}
\right)\\
\nonumber&\le  \beta \left(\big| \cL (e^{\eps'\ell})|_{\bar \Delta_{0}}\big|_{\eps',\infty}\,  |h|_{\beta,\eps',\text{Lip}}
+e^{\bar C_{\bar g}}\big| \cL (e^{\eps\ell})|_{\bar \Delta_{0}}\big|_{\eps,\infty}\, |h|_{\eps,\infty}
\right)\\
&\le  \beta \frac{ e^{\bar C_{\bar g}}}{\bar\nu(\Delta_{0})} \left( \Vert e^{\eps'\ell}\Vert_{L^1(\bar\nu)}\,  |h|_{\beta,\eps',\text{Lip}}
+e^{\bar C_{\bar g}} \Vert e^{\eps\ell}\Vert_{L^1(\bar\nu)}\, |h|_{\eps,\infty}
\right)
\label{BBBB1previous}
\end{align}
where we used \eqref{totoprevious} with both $\eps$ and $\eps'$ 
together with the fact that $\eps,\eps'\le\eps_0$. We conclude the proof of the lemma by gathering \eqref{equa1previous}, \eqref{equa1bisprevious}, \eqref {totoprevious}and \eqref{BBBB1previous}.
\end{proof}

\begin{lem}\label{contlinYoung}

Let $\gamma\in(0,1)$ and $\eps\ge 0$ be such that $\beta_0\le \beta_1^\gamma$ and
$\eps+\gamma\eps'_1<\eps_0$. Then $(\mathcal L_{is})_{s\in\mathbb R}$ is a family of bounded linear operators on  $\cB_{\beta_1^\gamma,\eps, \eps+\gamma\eps'_1}$. 
Moreover, for any nonnegative integer $r$ and any $\eps''>0$ such that $\eps+
r\eps_1+\gamma\eps'_1+\eps''\le\eps_0$, $s\mapsto\cL_{is}$ is 
$C^r$ from $\mathbb R$ to  $\mathcal L(\cB_{\beta_1^\gamma,\eps, \eps+\gamma\eps'_1},\cB_{\beta_1^\gamma,\eps+r\eps_1+\eps'', \eps+r\eps_1+\gamma\eps'_1+\eps''})$ with $(\mathcal L_{is})^{(m)}=\mathcal L_{is}((i\bar\phi)^m\,\cdot\,)$.
\end{lem}


\begin{proof}
Since we proved in Lemma~\ref{continuousaction} that
$\mathcal L$ is a bounded linear operator on $\cB_{\beta_1^\gamma,\eps, \eps'}$ for any $\eps,\eps'\in[0,\eps_0]$, it is enough to prove that, for any $\eps,\eps'\ge 0$ and $\eps''>0$, the linear map $s\mapsto (e^{is\bar\phi}\times\cdot)$ acts continuously on $\cB_{\beta_1^\gamma,\eps, \eps'}$ and is $C^r$ from $\mathbb R$ to  $\mathcal L(\cB_{\beta_1^\gamma,\eps, \eps+\gamma\eps'_1},\cB_{\beta_1^\gamma,\eps+r\eps_1+\eps'', \eps+r\eps_1+\gamma\eps'_1+\eps''})$ with the multiplication by $(i\bar\phi)^re^{is\bar\phi}$ as the $r^{\text{th}}$ derivative. This follows from the points given below.
\begin{itemize}[leftmargin=20pt]
\item for every $s\in\mathbb R$, the multiplication by $e^{is\bar\phi}$ is a bounded linear operator on 
$\cB_{\beta_1^\gamma,\eps, \eps+\gamma\eps'_1}$;
\item $s\mapsto e^{is\bar\phi}\times\cdot$ is a continuous from $\mathbb R$ to
$\mathcal L(\cB_{\beta_1^\gamma,\eps, \eps'},\cB_{\beta_1^\gamma,\eps+\eps'',\max(\eps',\eps+\gamma\eps'_1)+\eps''})$;
\item $s\mapsto e^{is\bar\phi}\times\cdot$ is differentiable from $\mathbb R$ to
$\mathcal L(\cB_{\beta_1^\gamma,\eps, \eps'},\cB_{\beta_1^\gamma,\eps+\eps_1+\eps'', \max(\eps',\eps+\gamma\eps'_1)+\eps_1+\eps''})$ with the derivative $i\bar\phi e^{is\bar\phi}\times \cdot\,$.
\end{itemize}

To prove the $C^r$ smoothness of $s\mapsto (e^{is\bar\phi}\times\cdot)$, we proceed by induction on $r$.
Let us write $\mathcal X_{\eps,r}:=\cB_{\beta_1^\gamma,\eps+r\eps_1, \eps+r\eps_1+\gamma\eps'_1}$.
The step $r=0$ follows from the second point above applied with $\eps'=\eps+\gamma\eps'_1$. Let $r\ge 0$. By induction hypothesis, $s\mapsto e^{is\bar\phi}\times\cdot$ is $C^r$ from $\mathbb R$ to  $\cL(\mathcal X_{\eps,0},\mathcal X_{\eps+\eps''/3,r})$
with $r^{\text{th}}$ derivative $e^{is\bar\phi}(i\bar\phi)^r\times\cdot$. But due to the third point above, 
$s\mapsto e^{is\bar\phi}$ is differentiable from $\mathbb R$ to $\mathcal L(\cX_{\eps+\eps''/3,r},\cB_{\eps+2\eps''/3,r+1})$ with derivative $i\bar\phi e^{is\bar\phi}\times \cdot\,$. So  $s\mapsto e^{is\bar\phi}\times\cdot$ is $(r+1)$ times differentiable from $\mathbb R$ to  $\mathcal L(\cX_{\eps, 0},\cX_{\eps+2\eps''/3, r+1)}$ with  $e^{is\bar\phi}(i\bar\phi)^{r+1}\times\cdot$ as the order $(r+1)$ derivative. 
In particular, $((i\bar\phi)^{r+1}\times\cdot)\in \mathcal L(\cX_{\eps, 0},\cX_{\eps+2\eps''/3, r+1})$. It follows from the second point above that
$s\mapsto e^{is\bar\phi}\times\cdot$ is continuous from $\mathbb R$ to 
$\mathcal L(\cX_{\eps+2\eps''/3, r+1},\cX_{\eps+\eps'', r+1})$. Thus 
$s\mapsto e^{is\bar\phi}(i\bar\phi)^{r+1}\times\cdot$ is continuous from
$\mathbb R$ to 
$\mathcal L(\cX_{\eps,0},\cX_{\eps+\eps'', r+1})$, and so $s\mapsto e^{is\bar\phi}\times\cdot$ is $C^{r+1}$ from $\mathbb R$ to  $\cL(\mathcal X_{\eps,0},\mathcal X_{\eps+\eps'',r+1})$
with $e^{is\bar\phi}(i\bar\phi)^{r+1}\times\cdot$ as the order $(r+1)$ derivative, which ends the proof by induction of the $C^r$-smoothness of $e^{is\bar\phi}\times\cdot$.

To complete the proofs, let us prove the three points above. For the first point, we observe that $e^{is\bar\phi}\in\cB_{\beta_1^\gamma,0,\gamma\eps'_1}$ since $\Vert e^{is\bar\phi}\Vert_\infty=1$ and
\[
\forall\gamma\in(0,1],\quad 
|e^{is\bar\phi(x)}-e^{is\bar\phi(y)}|\le \min(2,|s|\, |\bar\phi(x)-\bar\phi(y)|)\le
     2(|s|\, |\bar\phi(x)-\bar\phi(y)|)^\gamma\, 
\]
and since $\bar\phi\in\mathcal B_{\beta_1,\eps_1,\eps'_1}$.
Thus, the fact that the multiplication by $e^{is\bar\phi}$ is a bounded linear operator on 
$\cB_{\beta_1^\gamma,\eps, \eps+\gamma\eps'_1}$ comes from Lemma~\ref{lem:product}.

For the second point, due to Lemma~\ref{lem:product}, it is enough to prove that
$e^{is\bar\phi}-e^{it\bar\phi}$ is in 
$\mathcal B_{\beta_1^\gamma,\eps'',\gamma\eps'_1+\eps''}$ with norm going to 0 as $|s-t|\rightarrow 0$.
To see this, we first observe that
\begin{align*}
|e^{is\bar\phi}-e^{it\bar\phi}|_{\eps'',\infty} 
&\leq \sup_{l} e^{-l\eps''}\min(2,|s-t|\|\bar\phi|_{\bar{\Delta}_{l}}\|_{\infty})\\
&\le \sup_l e^{-l\eps''}\min(2,|s-t|\, |\bar\phi|_{\eps_1,\infty}e^{l\eps_1})\\
&\le   2 \left(|s-t|\, |\bar\phi|_{\eps_1,\infty}\right)^{\min(1,\frac {\eps''}{\eps_1})}
\end{align*}
and second that, for all $l,j$ and $y,y'\in\bar\Delta_{l,j}$ and for
$\gamma'\in(0,\eps'')$ such that $\gamma'\le 1-\gamma$,
\begin{align*}
|e^{is\bar\phi(y)}-&e^{it\bar\phi(y)}-e^{is\bar\phi(y')}+e^{it\bar\phi(y')}|\\
&\leq \min\left(|e^{is\bar\phi(y)}-e^{it\bar\phi(y)}|+|e^{is\bar\phi(y')}-e^{it\bar\phi(y')}|,|e^{is\bar\phi(y)}-e^{is\bar\phi(y')}|+|e^{it\bar\phi(y)}-e^{it\bar\phi(y')}|\right)\\
&\leq  \min\left(4,|s-t|\, |\bar\phi|_{\eps_1,\infty}e^{l\eps_1}
,(|s|+|t|)|\bar\phi|_{\beta_1,\eps'_1,\text{Lip}}\beta_1^{\hat s(y,y')}e^{l\eps'_1}\right)\\
&\leq   4\left(|s-t|\, |\bar\phi|_{\eps_1,\infty}e^{l\eps_1}\right)^{\gamma'}\left((|s|+|t|)|\bar\phi|_{\beta_1,\eps'_1,\text{Lip}}e^{l\eps'_1}\beta_1^{\hat s(y,y')}\right)^{\gamma}\, .
\end{align*}

For the third point, using again Lemma~\ref{lem:product}, it is enough to prove that
$$\frac{e^{i(t+h)\bar\phi}-e^{it\bar\phi}(1+ih\bar\phi)}{h} \in \mathcal B_{\beta_1^\gamma,\eps_1+\eps'',\eps_1+\gamma\eps'_1+\eps''}$$ with norm going to 0 as $h\rightarrow 0$. To this end, we observe that, for every $l,j$ and every $y,y'\in\bar\Delta_{l,j}$,
\begin{align*}
|e^{i(t+h)\bar\phi(y)}-e^{it\bar\phi(y)}(1+ih\bar\phi(y))|&= |e^{ih\bar\phi(y)}-1-ih\bar\phi(y)|\\
&\leq 2|h\bar\phi(y)|^{\min(2,1+\frac{\eps''}{\eps_1})}\leq 2 | h\, \bar\phi|^{\min(2,1+\frac{\eps''}{\eps_1})}_{\eps_1,\infty}  e^{l(\eps_1+\eps'')}
\end{align*}
and that, for all $\gamma_1\in(0,1]$
\begin{align*}
|e^{i(t+h)\bar\phi(y)}-e^{it\bar\phi(y)}(1+ih\bar\phi(y))&-e^{i(t+h)\bar\phi(y')}
+e^{it\bar\phi(y)}(1+ih\bar\phi(y'))|\\
&\ \ \leq \min\left(4|h\bar\phi(y)|^{1+\gamma_1},
|g_{t,h}(\bar\phi(y))-g_{t,h}(\bar\phi(y'))|\right)
\end{align*}
with $g_{t,h}(z):=e^{i(t+h)z}-e^{itz}(1+ihz)$. Since 
$$g'_{t,h}(z)=(i(t+h)(e^{ihz}-1)+thz)e^{itz},$$we have $ |g'_{t,h}(z)|\leq (2|t|+|h|)|h||z|$, and hence,
\begin{align*}
|e^{i(t+h)\bar\phi(y)}&-e^{it\bar\phi(y)}(1+ih\bar\phi(y))-e^{i(t+h)\bar\phi(y')}
+e^{it\bar\phi(y)}(1+ih\bar\phi(y'))|\\
\leq &\min\left(4|h|^{1+\gamma_1}\, |\bar\phi|_{\eps_1,\infty}^{1+\gamma_1}e^{\ell\eps_1(1+\gamma_1)},
(2|t|+|h|)|h| |\bar\phi|_{\eps_1,\infty}e^{\ell\eps_1}|\bar\phi|_{\beta_1,\eps'_1,\text{Lip}}\beta_1^{\hat s(y,y')}e^{\ell\eps'_1}\right)\\
\leq &4|h|^{1+\gamma_1(1-\gamma)}\, |\bar\phi|_{\eps_1,\infty}^{(1+\gamma_1)(1-\gamma)}e^{\ell\eps_1(1+\gamma_1)(1-\gamma)}
(2|t|+|h|)^\gamma |\bar\phi|_{\eps_1,\infty}^\gamma e^{\ell\eps_1\gamma}|\bar\phi|^\gamma_{\beta_1,\eps'_1,\text{Lip}}\beta_1^{\gamma\hat s(y,y')}e^{\ell\gamma\eps'_1}.
\end{align*}
We conclude by taking $\gamma_1\in(0,1]$ such that $0<\eps_1\gamma_1(1-\gamma)\le\eps''$.
\end{proof}

The Doeblin Fortet inequality given by the next result is a key estimate in the proof of the quasicompacity of $\cL_{is}$ on $\cB_{\beta_1^\gamma,\eps,\eps'}$.

\begin{lem}[Doeblin-Fortet inequality]\label{DFHP}

Assume $\max(\beta_0,\beta_1)<e^{-\eps'_1}$.
Let $T>1$, $0<\gamma_0<\gamma_1<1$ 
be such that $\beta_0<\beta_1^{\gamma_i}<e^{-\eps'_1}$ and
$\gamma_0<\eps_0/\eps'_1$. Then, for every $\eps'_0\in(0,\eps_0-2\gamma_1\eps'_1]$, for every $\theta\in (0,1)$, there exist $N'$ and $K'_\infty$ such that for every $s\in[-T,T]$, $\gamma\in[\gamma_0,\gamma_1]$ with $\gamma\eps'_1<\eps_0$ and for every $\eps\in[\eps'_0,\eps_0-2\gamma\eps'_1]$, 
\[
\quad \|\cL^{N'}_{is}h\|_{\beta_1^\gamma,\eps,\eps':=\eps+\gamma\eps'_1} \leq  \theta\|h\|_{\beta_1^\gamma,\eps,\eps'}+ K'_\infty\Vert h\Vert_{L^{\frac{\eps_0}{\eps_0-\gamma\eps'_1}}(\bar\nu)}\,, \forall h\in\mathcal B_{\beta_1^\gamma,\eps,\eps'}.
\]
\end{lem}
Observe that, for any  $h\in\mathcal B_{\beta_1^\gamma,\eps,\eps'}$, $|h|\le|h|_{\eps,\infty}e^{\eps\ell}\in L^{\frac{\eps_0}{\eps_0-\gamma\eps'_1}}(\bar\nu)$
since $\eps\le \eps_0-\gamma\eps'_1$.
\begin{proof}
Note that the condition on $\gamma$ implies that $\beta_0\le\beta_1^\gamma$.
Let $T,\gamma,\eps'_0,\eps,\eps',s,h,\theta$ be as in the statement of the Lemma.
\begin{itemize}[leftmargin=*]
\item Assume first that $l\ge N$. On $\Delta_l$, $\bar F^{-N}$ is well defined by $\bar F^{-N}(x,l)=(x,l-N)$ and $\mathcal L_{is}^N(h)=(e^{is\bar S_N}h)\circ\bar F^{-N}$. Thus
\begin{align}
\nonumber|(\cL_{is}^N ( h))_{|\bar\Delta_l}
|_{\eps,\infty} &\le |(\cL^N ( |h|))_{|\bar\Delta_l}
|_{\eps,\infty}
\le   \Vert (\cL^N ( |h|))_{|\bar\Delta_l}\Vert_\infty e^{-l\eps}\\
&\le   e^{-N\eps}\Vert h_{|\bar\Delta_{l-N}}\Vert_\infty e^{-(l-N)\eps}
\le   e^{-N\eps}|h|_{\eps,\infty}\le \frac\theta 2\Vert h\Vert_{\beta_1^\gamma,\eps,\eps'}\label{equa1}
\end{align}
if $N$ is large enough so that $e^{-\eps'_0N}<\theta/2$.
Moreover,
\begin{align*}
&|(\cL_{is}^N ( h))_{|\bar\Delta_\ell}|_{\beta_1^\gamma,\eps',\text{Lip}}=|(\cL^N (e^{is\bar S_N} h))_{|\bar\Delta_\ell}|_{\beta_1^\gamma,\eps',\text{Lip}}\\
&\le 
\beta_1^{\gamma N} e^{-\eps'N}\sup_{l,j}\sup_{y,y'\in\bar\Delta_{l,j}\, :\ \hat s(y,y')>N}\frac{|e^{is\bar S_N(y)}h(y)-e^{is\bar S_N(y')}h(y')|}{\beta_1^{\gamma \hat s(y,y')}}e^{-l\eps'}\, .
\end{align*}
But,
for any $l\ge N$, any $j$ and any $y,y'\in\bar\Delta_{l,j}$ such that $\hat s(y,y')>N$, we have
\begin{align*}
|e^{si\bar S_N(y)}&h(y)-e^{is\bar S_N(y)}h(y')| \\ &\ \ \le |h(y)-h(y')|+|h(y)|\, |e^{is\bar S_N(y)}h(y)-e^{is\bar S_N(y')}h(y')|\\
&\ \ \le |h|_{\beta_1^\gamma,\eps',\text{Lip}}\beta_1^{\gamma\hat s(y,y')}e^{l\eps'}
+|h|_{\eps,\infty}e^{l\eps}\min\left(2,|s|\sum_{k=0}^{N-1}|\bar\phi(\bar F^k(y))-\bar\phi(\bar F^k(y'))|\right)\\
&\ \ \le |h|_{\beta_1^\gamma,\eps',\text{Lip}}\beta_1^{\gamma\hat s(y,y')}e^{l\eps'}
+|h|_{\eps,\infty}e^{l\eps}\min\left(2,|s|\sum_{k=0}^{N-1}|\bar\phi|_{\beta_1,\eps'_1,\text{Lip}}\beta_1^{\hat s(y,y')-k}e^{(l+k)\eps'_1}\right)\\
&\ \ \le |h|_{\beta_1^\gamma,\eps',\text{Lip}}\beta_1^{\gamma\hat s(y,y')}e^{l\eps'}
+|h|_{\eps,\infty}e^{l\eps}\min\left(2,|s||\bar\phi|_{\beta_1,\eps'_1,\text{Lip}}\beta_1^{\hat s(y,y')}\frac{\beta_1^{-N}e^{\eps'_1N}}{\beta_1^{-1}e^{\eps'_1}-1}e^{l\eps'_1}\right)\\
&\ \ \le |h|_{\beta_1^\gamma,\eps',\text{Lip}}\beta_1^{\gamma\hat s(y,y')}e^{l\eps'}+2|h|_{\eps,\infty}e^{l\eps}|s|^\gamma|\bar\phi|^\gamma_{\beta_1,\eps'_1,\text{Lip}}\beta_1^{\gamma\hat s(y,y')}\frac{\beta_1^{-\gamma N}e^{\eps'_1\gamma N}}{(\beta_1^{-1}e^{\eps'_1}-1)^\gamma}e^{\gamma l\eps'_1}\\
&\ \ \le \left(|h|_{\beta_1^\gamma,\eps',\text{Lip}}+2|h|_{\eps,\infty}|s|^\gamma|\bar\phi|^\gamma_{\beta_1,\eps'_1,\text{Lip}}\frac{\beta_1^{-\gamma N}e^{\eps'_1\gamma N}}{(\beta_1^{-1}e^{\eps'_1}-1)^\gamma}\right)\beta_1^{\gamma\hat s(y,y')}e^{l\eps'}\, .
\end{align*}
Therefore, for $l\ge N$,
\begin{align}
|(\cL_{is}^N ( h))_{\bar\Delta_l}|_{\beta_1^\gamma,\eps',\text{Lip}}&\le \beta_1^{\gamma N} e^{-\eps'N} |h|_{\beta_1^\gamma,\eps',\text{Lip}} + 2\frac{e^{-\eps N}|T|^\gamma|\bar\phi|^\gamma_{\beta_1,\eps'_1,\text{Lip}}}{(
\beta_1^{-1}e^{\eps'_1}-1)^\gamma} |h|_{\eps,\infty}\le \frac\theta 2\Vert h\Vert_{\beta_1^\gamma,\eps,\eps'}\, ,
\label{equa2}
\end{align}
if $N$ is large enough so that $(\beta_1^{\gamma_0} e^{-\eps'_0})^N\le\frac\theta 2$ and 
$2\frac{e^{-\eps'_0 N}|T|^{\gamma_1}(1+|\bar\phi|^{\gamma_1}_{\beta_1,\eps'_1,\text{Lip}})}{\min(1,(
\beta_1^{-1}e^{\eps'_1}-1)^{\gamma_1})}\le \frac\theta 2$.
\item Assume from now on that $l< N$. 
We will use \eqref{formuleOpYoung}.
Then
\begin{align}
|(\cL_{is}^N ( h))_{|\bar\Delta_{l
}}
\nonumber|_{\eps,\infty} &\le  \sup_{x\in\bar\Delta_{l
}} \sum_{z\in \bar F^{-N}(\{x\})} |e^{S^{\bar g}_N(z)}e^{is
\bar S
_N(z)}h(z)|e^{-l\eps}\\
&\le  \sup_{x\in\bar\Delta_{l
}} \sum_{z\in \bar F^{-N}(\{x\})} e^{S^{\bar g}_N(z)}|h(z)|e^{-l\eps}\, .\label{AAAAAA1}
\end{align}
Note that for any $x\in\bar\Delta_l$ and any $z\in \bar F^{-N}(\{x\})$,
$\bar F^{-l}(x)\in\bar\Delta_0$ and 
$S_N^{\bar g}(z)=S_{N-l}^{\bar g}(z)$.
Moreover,  due to \eqref{regulg}, for all $x,y\in\bar\Delta_{l}$ and $z\in \bar F^{-N}(\{x\})$,
\begin{align}\label{regulsumg}
e^{S^{\bar g}_{N-l}(z)-S^{\bar g}_{N-l}(\bar F_{N-l,z}(\bar F^{-l}(y))
)}&\leq  e^{\bar C_{\bar g}\sum_{k=0}^{N-l-1}\beta_0^{\hat s(x,y)+N-k}} \leq  e^{\bar C_{\bar g}\frac{\beta_0^{\hat s(x,y)+l}}{1-\beta_0}}\, ,
\end{align}
where $\bar F_{N-l,z}^{-1}$ is the inverse branch of $\bar F^{N-l}$
sending $z$ to $\bar F^{-l}(x)$.
So for any $y'\in \bar F_{N-l,z}^{-1}(\bar\Delta_0)$,
\[
e^{S^{\bar g}_{N-l}(z)}
\leq e^{S^{\bar g}_{N-l}(
y'
)
}  e^{\frac{\bar C_{\bar g}}{1-\beta_0}}\, . 
\]
Thus,
\begin{align}
e^{S^{\bar g}_{N-l}(z)} \bar\nu(\bar\Delta_{0
})= \int_{\bar\Delta_0}e^{S^{\bar g}_{N-l}(z)}\, d\bar\nu(y')
 &\le e^{\frac{\bar C_{\bar g}}{1-\beta_0}} 
\int_{\bar\Delta_0}e^{S^{\bar g}_{N-l}(\bar F_{N-l,z}^{-1}(y'))}
\, d\bar\nu(y') \nonumber \\
\nonumber&\le   e^{\frac{\bar C_{\bar g}}{1-\beta_0}}\int_{\bar\Delta}(\cL^{N-l}\mathbf 1_{\bar F_{N-l,z}^{-1}(\bar\Delta_0)})(y')
\, d\bar\nu(y')\\
&\le \bar\nu(\bar F_{N-l,z}^{-1}(\bar\Delta_{0
}))   e^{\frac{\bar C_{\bar g}}{1-\beta_0}} \, .
\label{AAAAAA2}
\end{align}
Moreover,
\begin{align}\label{AAAAAA2bis}
\left|\bar\nu(\bar F_{N-l,z}^{-1}(\bar\Delta_{0
}))|h(z)|-\int_{\bar F_{N-l,z}^{-1}(\bar\Delta_{0
})}|h|\, d\bar\nu\right|\le\bar\nu(\bar F_{N-l,z}^{-1}(\bar\Delta_{0})) \vert h\vert_{\beta_1^\gamma,\eps',\text{Lip}}\beta_1^{\gamma(N-l)} e^{\eps'\ell(z)}.
\end{align}
Combining this with \eqref{AAAAAA1} and \eqref{AAAAAA2}, we obtain
\begin{align}
|(\cL_{is}^N& (h))_{|\bar\Delta_{l,j}}|_{\eps,\infty}\nonumber\\&\le\frac{ e^{\frac{\bar C_{\bar g}}{1-\beta_0}}}{\bar\nu(\Delta_{0})} e^{-l\eps}\sup_{x\in\bar\Delta_l} \sum_{z\in \bar F^{-N}(\{x\})} \left(\int_{\bar F_{N-l,z}^{-1}(\bar\Delta_{0})}|h|\, d\bar\nu+ \bar\nu(\bar F_{N-l,z}^{-1}(\bar\Delta_{0 })) \vert h\vert_{\beta_1^\gamma,\eps',\text{Lip}}\beta_1^{\gamma(N-l)} e^{\eps'\ell(z)}\right)\nonumber\\
&\le\frac{ e^{\frac{\bar C_{\bar g}}{1-\beta_0}}}{\bar\nu(\Delta_{0})} e^{-l\eps}\left(\Vert h\Vert_{L^1(\bar\nu)}+ \vert h\vert_{\beta_1^\gamma,\eps',\text{Lip}}\beta_1^{\gamma(N-l)} \Vert e^{\eps'\ell}\Vert_{L^1(\bar\nu)}\right)\nonumber\\
&\le\frac{ e^{\frac{\bar C_{\bar g}}{1-\beta_0}}}{\bar\nu(\Delta_{0
})} 
 \left(\Vert h\Vert_{L^1(\bar\nu)}+ \vert h\vert_{\beta_1^\gamma,\eps',\text{Lip}}\min(\beta_1^{\frac{\gamma N}2},e^{-\frac {N\eps}2}) \Vert e^{\eps'\ell}\Vert_{L^1(\bar\nu)}\right)
\end{align}
Indeed, either $ N\le 2l$ and then $\beta_1^{\gamma(N-l)}e^{-l\eps}\le e^{-l\eps}\le e^{-\frac{N\eps}2}$ or $N>2l$ and then $\beta_1^{\gamma(N-l)}e^{-l\eps}\le \beta_1^{\gamma(N-l)}\le\beta_1^{\frac {\gamma N}2}
$.
Note that $\Vert e^{\eps'\ell}\Vert_{L^1(\bar\nu)}$ is finite since
$\eps'\le\eps_0$.
Thus
\begin{equation}
|(\cL_{is}^N ( h))_{|\bar\Delta_{l,j}}|_{\eps,\infty}\leq\frac{ e^{\frac{\bar C_{\bar g}}{1-\beta_0}}}{\bar\nu(\Delta_{0
})} 
\Vert h\Vert_{L^1(\bar\nu)}+\frac\theta 2 \Vert h\Vert_{\beta_1^\gamma,\eps,\eps'}\, ,
\label{equa3}
\end{equation}
if $N$ is large enough so that
$\frac{ e^{\frac{\bar C_{\bar g}}{1-\beta_0}}}{\bar\nu(\Delta_{0
})} 
\min(\beta_1^{\frac{\gamma_0 N}2},e^{-\frac {N\eps'_0}2}) \Vert e^{\eps_0\ell}\Vert_{L^1(\bar\nu)}\le\frac\theta 2$.

It remains to estimate the Young Lipschitz constant in the case $l<N$.
\begin{align}
|(\cL_{is}^N ( h))&_{|\bar\Delta_{l,j}}|_{\beta_1^\gamma,\eps',\text{Lip}} \nonumber \\ &\le  e^{-l\eps'}\sup_{x\in\bar\Delta_{l,j
}} \sum_{z\in \bar F^{-N}(\{x\})} \sup_{y\in \bar F^{-1}_{N-l,z}(\bar\Delta_0)\, :\, \hat s(y,z)\ge N}\frac{\left| e^{S^{\bar g}_N(z)}e^{is\bar S_N(z)}h(z)-e^{S^{\bar g}_N(y)}e^{is\bar S_N(y)}h(y)\right|}{\beta_1^{\gamma (\hat s(y,z)-N)}} \nonumber \\
\nonumber&\le  e^{-l\eps'}\beta_1^{\gamma N}\sup_{x\in\bar\Delta_{l,j
}} \sum_{z\in \bar F^{-N}(\{x\})} \left(e^{S^{\bar g}_N(z)}|h|_{\beta_1^\gamma,\eps',\text{Lip}}e^{\eps'\ell(z)}
+\omega_N(z)\sup_{\bar F_{N-l,z}^{-1}(\bar\Delta_0)} |h|
\right)\\
&\le  \beta_1^{\gamma N}  |h|_{\beta_1^\gamma,\eps',\text{Lip}}   |\cL^N(e^{\eps'\ell})|_{\eps',\infty}+e^{-l\eps'}\beta_1^{\gamma N} 
\sup_{x\in\bar\Delta_{l,j
}} \sum_{z\in \bar F^{-N}(\{x\})}\left(\omega_N(z)\sup_{\bar F_{N-l,z}^{-1}(\bar\Delta_0)} |h|
\right)\, 
\label{BBBB1}
\end{align}
where 
$$\omega_N(z):=\sup_{y\in \bar F^{-1}_{N-l,z}(\bar\Delta_0)\, :\, \hat s(y,z)\ge N}\frac{\left| e^{S^{\bar g}_N(z)}e^{is\bar S_N(z)}-e^{S^{\bar g}_N(y)}e^{is\bar S_N(y)}\right|}{\beta_1^{\gamma\hat s(y,z)}}.$$

Observe first that the first term of the right hand side of 
\eqref{BBBB1} can be dominated by
\begin{equation}\label{BBBB1bis}
\beta_1^{\gamma N}  |h|_{\beta_1^\gamma,\eps',\text{Lip}} \frac{
e^{\frac{\bar C_{\bar g}}{1-\beta_0}}}{\bar\nu(\Delta_{0
})} 
\Vert e^{\eps'\ell}\Vert_{L^1(\bar\nu)}\le  \frac\theta 4  \Vert h\Vert_{\beta_1^\gamma,\eps,\eps'}\, ,
\end{equation}
thanks to \eqref{equa3} taking $N$ large enough so that
$\beta_1^{\gamma_0 N}  |\frac{
e^{\frac{\bar C_{\bar g}}{1-\beta_0}}}{\bar\nu(\Delta_{0
})} 
\Vert e^{\eps_0\ell}\Vert_{L^1(\bar\nu)}\le  \frac\theta 4 $. 
Note that
\begin{align}
\omega_N(z)\le\sup_{y\in \bar F^{-1}_{N-l,z}(\bar\Delta_0)\, :\, \hat s(y,z)\ge N} &\frac{e^{S^{\bar g}_N(z)}|e^{is\bar S_N(z)}-e^{is\bar S_N(y)}|}{\beta_1^{\gamma\hat s(y,z)}} \label{BBBB2} \\
&\ \ \ \ \ \ \ \ \ \ \ \ \ \ \ +\sup_{y\in \bar F^{-1}_{N-l,z}(\bar\Delta_0)\, :\, \hat s(y,z)\ge N} \frac{|e^{S^{\bar g}_N(z)}-e^{S^{\bar g}_N(y)}|}{\beta_1^{\gamma\hat s(y,z)}}\, .\nonumber
\end{align}
But, due to~\eqref{regulsumg},
\begin{align}
\nonumber|e^{S^{\bar g}_N(z)}-e^{S^{\bar g}_N(y)}|&=e^{S^{\bar g}_N(z)}\left|1-e^{S^{\bar g}_N(y)-S^{\bar g}_N(z)}\right|
\\
\nonumber&\le e^{S^{\bar g}_N(z)}
\left(e^{\bar C_{\bar g}\frac{\beta_0^{\hat s(y,z)-N}}{1-\beta_0}}-1\right)\\
&\le e^{S^{\bar g}_N(z)+\frac{\bar C_{\bar g}}{1-\beta_0}}\bar C_{\bar g}\frac{\beta_0^{\hat s(y,z)-N}}{1-\beta_0}\, .\label{BBBB3}
\end{align}
Moreover,
\begin{align*}
|e^{is\bar S_N(z)}-e^{is\bar S_N(y)}|&\le \min\left(2,\sum_{k=0}^{N-1}|s\bar \phi|_{\beta_1,\eps'_1,\text{Lip}}\beta_1^{\hat s(y,z)-k}e^{\eps'_1(\ell(z)+k)}\right)\\
&\le \min\left(2,|s\bar \phi|_{\beta_1,\eps'_1,\text{Lip}}\frac{\beta_1^{\hat s(y,z)-N}e^{\eps'_1(\ell(z)+N)}}{\beta_1^{-1} e^{\eps'}-1}\right)\\
&\le 2|s\bar \phi|^\gamma_{\beta_1,\eps'_1,\text{Lip}}\frac{
e^{\eps'_1(\gamma\ell(z)+\gamma N)}}{(\beta_1^{-1} e^{\eps'}-1)^\gamma}\beta_1^{\gamma(\hat s(y,z)-N)}\, .
\end{align*}
Combining this with \eqref{BBBB2} and \eqref{BBBB3} and using the fact that
$\beta_0\le\beta_1^\gamma$, we obtain
\[
\omega_N(z)\le e^{S^{\bar g}_N(z)+\frac{\bar C_{\bar g}}{1-{\beta_0}}}\bar C_{\bar g} \frac{{\beta_1}^{-\gamma N}}{1-{\beta_0}}+ 2e^{S^{\bar g}_N(z)}|T\bar \phi|^\gamma_{\beta_1,\eps'_1,\text{Lip}}\frac{\beta_1^{-\gamma N}
e^{\eps'_1(\gamma\ell(z)+\gamma N)}}{(\beta_1^{-1} e^{\eps'}-1)^\gamma}\, ,
\]
and so
the last term of \eqref{BBBB1} is less than
\begin{align*}
c\, e^{-l\eps'}
\sup_{x\in\bar\Delta_{l,j
}} \sum_{z\in \bar F^{-N}(\{x\})} &e^{S^{\bar g}_N(z)}
 \left(1+ 
e^{\eps_1'(\gamma\ell(z)+\gamma N)}
\right)\sup_{\bar F_{N-l,z}^{-1}(\bar\Delta_0)} |h|\\
&\le 2c\, e^{-l\eps'} \sup_{x\in\bar\Delta_{l,j}} \sum_{z\in \bar F^{-N}(\{x\})} e^{S^{\bar g}_N(z)}e^{\eps_1'(\gamma\ell(z)+\gamma N)}
 \sup_{\bar F_{N-l,z}^{-1}(\bar\Delta_0)} |h|\, ,
\end{align*}
where $c$ is some positive constant depending on $\gamma_0,\gamma_1,T$. To control this term, we will proceed as for \eqref{equa3}.
Recall that \eqref{AAAAAA2} says that
\[
e^{S^{\bar g}_{N}(z)} \leq \frac{\bar\nu(\bar F_{N-l,z}^{-1}(\bar\Delta_{0
}))}{\bar\nu(\bar\Delta_{0
})}   e^{\frac{\bar C_{\bar g}}{1-\beta_0}} \, 
\]
and note that, analogously to \eqref{AAAAAA2bis},
\begin{align*}
\bigg|\bar\nu(\bar F_{N-l,z}^{-1}(\bar\Delta_{0}))\sup_{\bar F_{N-l,z}^{-1}(\bar\Delta_0)} |h|e^{\eps'_1\gamma\ell(z)}-&\int_{\bar F_{N-l,z}^{-1}(\bar\Delta_{0})}|h|e^{\eps'_1\gamma\ell(z)}\, d\bar\nu\bigg| \\ &\le \bar\nu(\bar F_{N-l,z}^{-1}(\bar\Delta_{0}))\vert h\vert_{\beta_1^\gamma,\eps',\text{Lip}}\beta_1^{\gamma(N-l)} e^{\eps'\ell(z)}e^{\eps'_1\gamma\ell(z)}\, .
\end{align*}
Since $\eps'=\eps+\gamma\eps'_1$, 
we obtain a domination of  the last term of \eqref{BBBB1} by
\begin{align*}
&2c\, e^{-l\eps'}e^{\eps_1'\gamma N} \frac{
e^{\frac{\bar C_{\bar g}}{1-\beta_0}} }{\bar\nu(\bar\Delta_{0
})}   \left(\Vert he^{\gamma\eps'_1\ell}\Vert_{L^1(\bar\nu)}+
\vert h\vert_{\beta_1^\gamma,\eps',\text{Lip}}\beta_1^{\gamma(N-l)} \Vert e^{(\eps'+\eps'_1\gamma)\ell}\Vert_{L^1(\bar\nu)}
\right)\\
&\le 2c\, \frac{
e^{\frac{\bar C_{\bar g}}{1-\beta_0}} }{\bar\nu(\bar\Delta_{0
})}   \left(e^{\eps_1'\gamma N} \Vert h\Vert_{L^{\frac{\eps_0}{\eps_0-\gamma\eps'_1}}(\bar\nu)}\, \Vert e^{\gamma\eps'_1\ell}\Vert_{L^{\frac{\eps_0}{\gamma\eps'_1}}(\bar\nu)}+
\vert h\vert_{\beta_1^\gamma,\eps',\text{Lip}}\max(e^{-\frac{\eps'_0N} 2 },(\beta_1 e^{\eps'_1})^{\frac {\gamma_0 N}2})  \Vert e^{(\eps'+\eps'_1\gamma)\ell}\Vert_{L^1(\bar\nu)}
\right)\\
\end{align*}
Indeed, since $\eps'=\eps+\gamma \eps'_1$, 
$e^{(\eps'_1\gamma N-\eps' l)}\beta_1^{\gamma(N-l)}
=(\beta_1 e^{\eps'_1})^{\gamma_0(N-l)} e^{-\eps l}
$ with $\beta_1 e^{\eps'_1}<1$ and $e^{-\eps'_0}<1$  and either $N-l\ge N/2$ or $l\ge N/2$.
Now 
$\eps'+\eps'_1\gamma\le \eps_0$ ensures that 
$\Vert e^{(\eps'+\eps'_1\gamma)\ell}\Vert_{L^1(\bar\nu)}$.

We conclude by combining this with \eqref{equa1}, \eqref{equa2},
\eqref{equa3}, \eqref{BBBB1bis} (for the first term of  the right hand side of \eqref{BBBB1}).
\end{itemize}\end{proof}
 
\begin{lem}[Compact inclusion]\label{compactinclusion}
Let $\gamma,\eps>0$ such that $\eps+\gamma\eps'_1\le\eps_0$. Then the unit ball of $\mathcal B_{\beta_1^\gamma,\eps,\eps':=\eps+\gamma\eps'_1}$
is relatively compact in $L^{\frac{\eps_0}{\eps_0-\gamma\eps'_1}}(\bar\nu)$.
 \end{lem}
\begin{proof}
Let $\eta>0$. Let us prove that the unit ball $ B(0,1)$ of $\cB_{\beta_1^\gamma,\eps,\eps'}$ admits a finite $\eta$-cover for $\Vert\cdot\Vert_{L^{p_0}(\bar\nu)}$ with $p_0:=\frac{\eps_0}{\eps_0-\gamma\eps'_1}$. 

We consider an increasing sequence $(\Sigma_k)_{k\ge 0}$ of subsets of $\bar\Delta$ such that $\Sigma_k$ consists of a finite union of distinct elements $(Q_j)_{j=1,...,N_k}$ of the family $\{\bigcap_{m\le k}\bar F^{-m}(\bar\Delta_{l_m,j_m}),\,  l_0<k\}$ and such that
\[
\lim_{k\rightarrow+\infty}\mathbb E_{\bar\nu}\left( e^{\ell\eps_0}\mathbf 1_{\bar\Delta\setminus\Sigma_k}\right)=0\, .
\]
We consider the projector $P_k$ given by $P_kh(x)=\mathbb E_{\bar\nu}[h|Q_j]$ if $x\in Q_j$ and $P_kh(x)=0$ if $x\not\in\Sigma_k$. Set $p_0:={\frac{\eps_0}{\eps_0-\gamma\eps'_1}}$.
Observe that
\begin{align*}
\Vert h-P_kh\Vert_{L^{p_0}(\bar\nu)}^{p_0}
&\leq   |h|^{p_0}_{\eps,\infty} \mathbb E_{\bar\nu}\left(e^{p_0\ell\eps}\mathbf 1_{\bar\Delta\setminus\Sigma_k}\right)+
         |h|^{p_0}_{\beta_1^\gamma,\eps',\text{Lip}} \mathbb E_{\bar\nu}\left(e^{p_0\eps'\ell}\beta_1^{p_0\gamma k}\right)\\
&\leq     \Vert h\Vert^{p_0}_{\beta_1^\gamma,\eps,\eps'}
                \mathbb E_{\bar\nu}\left(e^{\ell\eps_0}\left(\mathbf 1_{\bar\Delta\setminus\Sigma_k}+
      \beta_1^{p_0\gamma k}\right)\right)\, ,
\end{align*}
which converges to $0$ as $k\rightarrow +\infty$.
Let $\eta>0$ and let us take $k$ large enough so that 
$\Vert \text{Id}-P_k\Vert_{\cL(\cB_{\beta_1^\gamma,\eps,\eps'},L^{p_0}(\bar\nu))}<\eta/3$.

\begin{itemize}[leftmargin=*]
\item Observe first that, for every $h\in B(0,1)$, $P_kh$ has the form $P_kh=\sum_{j=1}^{N_k} a_j e^{\eps\ell(\cdot)}\mathbf 1_{Q_j}$ with $|a_j|\le 1$.
Let us consider a finite covering  $\mathcal A$ of the closed unit complex disk $\mathbb D$ made of balls of diameter
$\eta/(3\Vert   e^{\eps_0\ell}\Vert_{L^{1}(\bar\nu)}^{1/p_0})$. We observe that for any $h\in B(0,1)$ there exists
$A_{i_1},...,A_{i_{N_k}}\subset \mathcal A$ such that $P_kh\in\sum_{j=1}^{N_k} A_{i_j} e^{\eps\ell(\cdot)}\mathbf 1_{Q_j}$.
\item Now if $h_1,h_2\in B(0,1)$ are such that $P_k(h_1),P_k(h_2)\in\sum_{j=1}^{N_k} A_{i_j} e^{\eps\ell(\cdot)}\mathbf 1_{Q_j}$, then
\[
\Vert P_k(h_1)-P_k(h_2)\Vert^{p_0}_{L^{p_0}(\bar\nu)}\le  (\eta/ (3\Vert   e^{\eps_0\ell}\Vert_{L^{1}(\bar\nu)}^{1/p_0}))^{p_0}\Vert   e^{\eps\ell}\Vert^{p_0}_{L^{p_0}(\bar\nu)}\le 
 (\eta/ 3)^{p_0}\, ,
\]
since $\eps p_0\le \eps_0$.
\item Moreover for all $h_1,h_2$ in the unit ball $B(0,1)$ of $\cB_{\beta_1^\gamma,\eps,\eps'}$ such that
$\Vert P_k(h_1)-P_k(h_2)\Vert_{L^{p_0}(\bar\nu)}<\eta/3$, we also have
\[
\Vert h_1-h_2\Vert_{L^{p_0}(\bar\nu)}\le \Vert P_k(h_1)-P_k(h_2)\Vert_{L^{p_0}(\bar\nu)}+\sum_{i=1}^2 \Vert h_i-P_k(h_i)\Vert_{L^{p_0}(\bar\nu)}<\eta\, .
\]
\item Thus the sets $P_k^{-1}\left(A_{i_j} e^{\eps\ell(\cdot)}\mathbf 1_{Q_j}\right)$ realize a finite covering of $B(0,1)$ in sets of diameter less than $\eta$ for
$\Vert\cdot\Vert_{L^{p_0}(\bar\nu)}$.
\end{itemize}
\end{proof}

\begin{proof}[Proof of Proposition \ref{propCond3'}]
Recall that $\cX_j=\cV_{a_j}^{(\eps)}$ and $\cX_j^{(+)}=\cV_{b_j}^{(\eps)}$
with
$\mathcal V^{(\eps)}_\theta:=\mathcal B_{\beta_1^\gamma,\eps+\theta\eps_1,\eps+\theta\eps_1+\gamma\eps'_1}$ for any $\theta\in[0,r+2+\frac{\eps''}{\eps_1}]$, $a_k=k\left(1+\frac{\eps''}{(r+3)\eps_1}\right)$ and $b_k=a_k+\frac{\eps''}{2(r+3)\eps_1}c$. Observe that for every
$0\le\theta<\theta'\le r+2+\frac{\eps''}{\eps_1}$, the following sequence of continuous inclusions hold true
\[
\mathcal V_{0}\hookrightarrow \mathcal V_{\theta}\hookrightarrow \mathcal V_{\theta'}\hookrightarrow \mathcal V_{r+2+\frac{\eps''}{\eps_1}}\, ,
\]
and for every $j\in\{0,...,r+2\}$, the map
$s\mapsto \mathcal L_{is}$ is $C^{j}$ from $\mathbb R$ to $\mathcal L(\mathcal V_\theta,\mathcal V_{\theta'})$ as soon as $\theta'-\theta>j$. This ensures Assumption $(A)(3)$. The fact that
$\mathcal V_\theta\hookrightarrow L^{\frac{\eps_0}{\eps_0-\gamma\eps'_1}}(\bar\nu)$ comes from $e^{\eps_0\ell}\in L^1(\bar \nu)$ and $\eps+\theta\eps_1\le\eps_0$. Moreover, the Doeblin Fortet inequalities of Assumption $(A)(4)$ follow from Lemma~\ref{DFHP}.

Since $\cL\mathbf 1=\mathbf 1$, the quasicompacity of $\cL$ on $\cB_{\eps,\eps+\gamma\eps'_1}$ as soon as $\eps+2\gamma\eps'_1\le\eps_0$ follows from Proposition~\ref{DF}
combined with Lemmas~\ref{DFHP} (ensuring the Doeblin Fortet inequality) and~\ref{compactinclusion} (ensuring the compact inclusion of $\cB_{\eps,\eps+\gamma\eps'_1}$ in $L^{\frac{\eps_0}{\eps_0-\gamma\eps'_1}}(\bar\nu)$).
The fact that 1 is the unique eigenvalue of modulus 1 of $\cL$ on these spaces and is simple follows from the assumption that $\text{gcd}(r_i)$ and  \cite[Lemma 5]{Y}. This ends the proof of Assumption $(B)(1)$.

Let us prove that $\sum_{n\ge 0}\Vert \cL^n(\bar\phi)\Vert_2<\infty$. We have already noticed that $\bar\phi\in L^2(\bar\nu)$. Moreover, 
$$\mathcal L(\bar\phi)=-i\mathcal L_0^{(1)}(\mathbf 1)\in \mathcal V^{(\eps/2)}_{1+\frac{\eps''}{3\eps_1}}.$$
Since $\mathcal L$ is quasicompact on $\mathcal V^{(\eps/2)}_{1+\frac{\eps''}{3\eps_1}}$ with single dominating eigenvalue 1 which has multiplicity 1, we conclude that $\sum_{n\ge 1}\Vert \cL^n(\bar\phi)\Vert_{L^2(\bar\nu)}\le c\sum_{n\ge 1}\Vert \cL^n(\bar\phi)\Vert_{\mathcal V^{(\eps/2)}_{1+\frac{\eps''}{3\eps_1}}}<\infty$. Since 
\[
\frac{\eps}{2}+\eps_1(1+\frac{\eps''}{3\eps_1})<\frac12(\eps+(r+2)\eps_1+\gamma\eps'_1+\eps'')\le \frac{\eps_0}2\, ,
\]
we conclude that $$\mathcal V^{(\eps/2)}_{1+\frac{\eps''}{3\eps_1}}\hookrightarrow  L^{\frac{\eps_0}{\frac\eps 2+\eps_1+\frac{\eps''}3}}(\bar\nu)\hookrightarrow L^2(\bar\nu),$$ and hence, $\sum_{n\ge 1}\Vert \cL^n(\bar\phi)\Vert_2<\infty$ which ends the proof of Assumption $(B)(3)$.

The Doeblin Fortet inequality coming from Lemma~\ref{DFHP} combined with the compact inclusion property
stated in Lemma~\ref{compactinclusion} ensures, by Proposition~\ref{DF} that the spectral radius of $\cL_{is}$ is strictly smaller than 1 and that the spectral radius of $\cL_{is}$ is smaller than or equal to 1. Hence Assumption $(B)(2)$ follows from Lemma~\ref{periphericalspect}. 
\end{proof}

\subsection{Proofs for hyperbolic Young towers and unbounded observables}
Assume $\mathfrak p=\text{Id}$, i.e., $(f,\mathcal M,\mu)=(F,\Delta,\nu)$.
For any $\beta\in (0,1)$ and $\eps\ge 0$, we recall that $\mathcal V^{(0)}_{\beta,\eps}$ is the space of functions $h:\Delta\rightarrow\mathbb C$ such that $h e^{-\eps\ell}$ belongs to the space $\widetilde\cB_\beta$ defined in \eqref{tildeB},
where $\ell(x)$ is the level of the tower $\Delta$ to which $x$ belongs.
\begin{lem}
Let $\beta\in(0,1)$ and $\eps>0$.
If  $\phi\in\mathcal V_{\beta,\eps}^{(0)}$, then there exist $\bar\phi\in\mathcal B_{\sqrt{\beta},\eps}$ and $\chi\in\mathcal V^{(0)}_{\sqrt{\beta},\eps}$
such that
\[
\phi =\bar\phi\circ\mathfrak p+\chi-\chi\circ F\, .
\]
\end{lem}
\begin{proof}
Applying Lemma~\ref{LEM0}, with the notations of that section, we know that, for any $\phi\in\mathcal V_{\beta,\eps}^{(0)}$, there exists $\bar\phi_0\in\mathcal B_{\sqrt{\beta}}=\mathcal B_{\sqrt{\beta},0}$ and $\chi_0\in\widetilde{\mathcal B}_{\sqrt{\beta}}$
such that $\phi e^{-\ell\eps}=\bar\phi_0\circ\mathfrak p+\chi_0-\chi_0\circ F$. We end the proof of the lemma by setting
$\bar\phi:=\bar\phi_0 e^{\ell\eps}$ and $\chi:=\chi_0e^{\ell\eps}$.
\end{proof}

The next lemma will be the key step to prove Assumption $(A)(2)$.
Recall $\bar S_n=\sum_{k=0}^{n-1}\bar\phi\circ F^k$. Given $H:\Delta\rightarrow\mathbb C$, as in Assumption $(A)(2)$ and as in Lemma~\ref{LEM0}, we set
\[
h_{k,s}^{(j)}:=H\circ F^k\left(i(\chi \circ  F^k-\bar S_k\circ\mathfrak p)\right)^je^{is\chi \circ F^k }  e^{-is\bar{S}_k\circ \bar{\mathfrak p} }\, ,
\]
and
\[
\bar{h}^{(j)}_{k,s}(x)= e^{-is\bar S_k(x)}\EXP_{\nu}\big[H\circ F^k\, \big(i(\chi \circ  F^k-\bar S_k(x))\big)^je^{is\chi \circ F^k }\, \big|\,\hat s(\cdot , x) >2k\big]\, .
\]
Recall that in Proposition~\ref{propCond3'} and in its proof, we have set $\beta_1=\sqrt{\beta}$ and taken
$\eps,\eps''>0$ and $\gamma\in(0,1)$ such that $\eps+(r+2+2\gamma)\eps_1+\eps''\le\eps_0$ and $\beta_0\leq \beta_1^\gamma< e^{-\eps_1}$. 
Moreover, we have set
$p_0:=\frac{\eps_0}{\eps_0-\gamma\eps_1}$ and considered
the family of Banach spaces
$(\mathcal V^{(\eps)}_\theta:=\mathcal B_{\beta_1^\gamma,\eps+\theta\eps_1,\eps+(\theta+\gamma)\eps_1})_{\theta\in[0,r+2+\frac{\eps''}{\eps_1}]}$
with $$a_k=k\left(1+\frac{\eps''}{(r+3)\eps_1}\right)\,\,\,\text{and}\,\,\,b_k=a_k+\frac{\eps''}{2(r+3)\eps_1}.$$
Furthermore, in view of Assumption $(A)[r](2)$, we have set $\cX_j=\cV^{(\eps)}_{a_j}$ and
$\cX_j^{(+)}=\cV^{(\eps)}_{b_j}$.
Observe that
 $\mathcal V^{(\eps)}_\theta\hookrightarrow L^{\frac{\eps_0}{\eps+\theta\eps_1}}(\bar\nu)$ and so that
$L^{\frac{\eps_0}{\eps_0-\eps-\theta\eps_1}}(\bar\nu)\hookrightarrow \left(\mathcal V^{(\eps)}_\theta\right)'$.

\begin{lem}\label{CondA2Young}
Let $\eps>0, \eps''>0, \gamma \in(0,1), \beta \in(0,1)$ such that $\eps+(r+2+2\gamma)\eps_1+\eps''<\eps_0$.  
Let $\phi\in\mathcal V_{\beta,\eps_1}^{(0)}$ and $H\in
\mathcal V^{(0)}_{\beta_1^\gamma,\eps_2}$.
\begin{itemize}[leftmargin=*]
\item If $\eps_2\le \eps_0-(r+2+\gamma)\eps_1$ and if $\beta_1^\gamma e^{\eps_2+(j+\gamma)\eps_1}\le 1$, then
\begin{equation}\label{hk1young}
\Vert h^{(j)}_{k,s}\circ\mathfrak p-\bar h^{(j)}_{k,s}\circ\bar{\mathfrak p}\Vert_
{L^{\frac{r_0+2}{j
+\gamma+\frac{\eps_2}{\eps_1}}}(\nu)}
\le C_0 \vartheta^k (1+|s|^\gamma)k^j\quad\mbox{with}\quad r_0=\frac{\eps_0}{\eps_1}-2\, .
\end{equation}
\item If 
$\eps_2\leq \eps_0-\eps-(b_{r+2-j}+j)\eps_1
$
$($this holds true if $\eps_2\leq \eps_0-\eps-(r+2)\eps_1-\eps''$ and so if $\eps_2\leq
2\gamma\eps_1),$ then for all $j=0,...,r+2,$
\begin{equation}\label{hk2young2}
\left\Vert {\bar h^{(j)}_{k,s}}\right\Vert_{(\cV^{(\eps)}_{b_{r+2-j}})'}\leq\left\Vert {\bar h^{(j)}_{k,s}}\right\Vert_{L^{\frac{\eps_0}{\eps_0-\eps-b_{r+2-j}\eps_1}}(\bar\nu)}\leq \left\Vert {\bar h^{(j)}_{k,s}}\right\Vert_{L^{\frac{\eps_0}{\eps_2+j\eps_1}}(\bar\nu)}
\leq C_0 (1+k)^j\, .
\end{equation}
\item If $\eps_2+j\eps_1\le \eps_0$, then
\begin{equation}\label{hk2young1}
\left\Vert (\mathcal L_{is}^{2k}{\bar h_{k,s}})^{(j)}\right\Vert_{\mathcal B_{\beta_1^\gamma,\eps_2+j\eps_1,\eps_2+j\eps_1}}\le C_0(1+|s|) k^j\, .
\end{equation}
In particular$,$ if $\eps_2\le\eps+\eps''$, then $\left\Vert (\mathcal L_{is}^{2k}{\bar h_{k,s}})^{(j)}\right\Vert_{\cV_{a_j}^{(\eps)}}
\le C_0(1+|s|) k^j$.

\end{itemize}
\end{lem}
\begin{proof}
Observe first that, for any $p\ge 1$, $\Vert  \bar{h}^{(j)}_{k,s}\Vert_{L^p(\bar\nu)}\le\Vert  h^{(j)}_{k,s}\Vert_{L^p(\nu)}$.
For $x\in\Delta_{l,j}$,
\begin{align*}
\vert\bar{h}^{(j)}_{k,s}\circ\bar{\mathfrak p}(x)-&h^{(j)}_{k,s}(x)\vert \\
&\leq  
\sum_{m=0}^{j}{j\choose m}\vert \bar S_k\circ\mathfrak p\vert^{j-m}
 \sup_{y\in\Delta_{l,j},\ \hat s(x,y)>2k}\left|(H\chi^{m}e^{is\chi})(F^k(x))-(H\chi^{m}e^{is\chi})(F^k(y))\right| \, .
\end{align*}
But, for all $x,y\in\gamma^u\subset\Delta_{l,j}$ or $x,y\in\gamma^u\subset\Delta_{l,j}$ such that $\hat s(x,y)\ge 2k$,
using the fact that $\chi\in\cV^{(0)}_{\beta_1,\eps_1}$, $|\bar S_k(x)|=\cO\left(\sum_{r=0}^{k-1}e^{\eps_1 r}\right)=\cO\left(e^{\eps_1 k}\right)$ and
\[
\left|(H\chi^{m}e^{is\chi})(F^k(x))-(H\chi^{m}e^{is\chi})(F^k(y))\right| 
=\mathcal O\left((1+|s|^\gamma)e^{(\eps_2+(m+\gamma)\eps_1)(k+l)}\beta_1^{\gamma k}\right)
\]
and so, in view of \eqref{hk1young},
\[
\left\Vert h^{(j)}_{k,s}-\bar h^{(j)}_{k,s}\circ\bar{\mathfrak p}\right\Vert^q_{L^q(\nu)} \leq \mathcal O\left((1+|s|^\gamma )^q(\beta_1^\gamma e^{\eps_2+(j+\gamma)\eps_1})^{kq}\mathbb E_{\nu}\left( e^{q(\eps_2+(j+\gamma)\eps_1)\ell}\right)\right)\, ,
\]
is in $\mathcal O(1+|s|^\gamma)^q$  since $\beta_1^\gamma e^{\eps_2+(j+\gamma)\eps_1}\le 1$ and $q(\eps_2+(j+\gamma)\eps_1)\le\eps_0$, so that
\[
q\le\frac{\eps_0}{\eps_2+(j+\gamma)\eps_1}
=\frac{r_0+2}{j+\gamma+\frac{\eps_2}{\eps_1}}
\]

In view of \eqref{hk2young2}, set 
$q=\frac{\eps_0}{\eps_2+j\eps_1}$, $r_m=\frac{\eps_2+j\eps_1}{m\eps_1}\ge 1$ and $s_m=\frac{\eps_2+j\eps_1}{\eps_2+(j-m)\eps_1}\ge $ such that $\frac 1{r_j}+\frac 1{s_j}=1$, we have
\begin{align*}\label{h_Boundbis}
\Vert\bar h^{(j)}_{k,s}\Vert_{L^q(\bar\nu)}
&\leq  j!\sum_{m=0}^j \Vert H |\chi|^{j-m}\Vert_{L^{qs_m}(\bar\nu)}\Vert |\bar S_k|^{m}\Vert_{L^{qr_m}(\bar\nu)}\\
&\leq  j!\sum_{m=0}^j \Vert H\Vert^{(0)}_{\beta_1^\gamma,\eps_2}(\Vert\chi\Vert^{(0)}_{\beta_1,\eps_1})^{j-m}\Vert e^{\ell(\eps_2+(j-m)\eps_1)}\Vert_{L^{qs_m}(\bar\nu)}\Vert \bar S_k\Vert_{L^{qmr_m}(\bar\nu)}^{m} \\
&\leq  j!\sum_{m=0}^j \Vert H\Vert^{(0)}_{\beta_1^\gamma,\eps_2}(\Vert\chi\Vert^{(0)}_{\beta_1,\eps_1})^{j-m} \mathbb E_{\bar\nu}\left( e^{\eps_0\ell}\right)^{\frac 1{s_mq}} k^{m}\Vert\bar\phi\Vert_{L^{mqr_m}(\bar\nu)}^{m} \\
&\leq  j!\sum_{m=0}^j \Vert H\Vert^{(0)}_{\beta_1^\gamma,\eps_2}(\Vert\chi\Vert^{(0)}_{\beta_1,\eps_1})^{j-m}k^{m}\Vert e^{\eps_1\ell}\Vert_{L^{mqr_m}(\bar\nu)}^{m}  \mathbb E_{\bar\nu}\left( e^{\eps_0\ell}\right)^{\frac 1{s_mq}}\, ,
\end{align*}
which is in $\mathcal O(1+|k|^j)$ since $\eps_1mqr_m=\eps_0$.

To investigate \eqref{hk2young1}, we observe that, for every $x\in\bar\Delta$,
\begin{align*}
(\cL^{2k}_{is}&\bar{h}_{k,s})^{(j)}(x)=\cL^{2k}(\tilde{H}^{(j)}_{k,s})(x),\quad\mbox{with}\quad
\tilde{H}_{k,s}(x):= e^{is\bar S_k\circ \bar F^k(x)}E_{2k}[H\circ F^k e^{is\chi \circ F^k }](x)\, .
\end{align*} 
where we used the notation $E_{2k}[G](x):=\EXP_{\nu}[G |\,\hat s( \cdot , x) >2k]$.
Observe that
\begin{align*}
&\cL^{2k}(\tilde{H}^{(j)}_{k,s})\\ &=\sum_{m=1}^j\frac{j!\, i^j}{m!(j-m)!}\sum_{k_1,...,k_m=0}^{k-1}\mathcal L^{2k}\left(e^{is\bar S_k\circ \bar F_k}
    \prod_{u=1}^m(\bar\phi\circ\bar F^{k_u+k})E_{2k}[H\circ F^k(\chi\circ F^k)^{j-m}e^{is\chi\circ F^k}]\right)\\
&=\mathcal O\left(\sum_{m=1}^j\sum_{k_1,...,k_m=0}^{k-1}\mathcal L^{2k}\left(e^{(\eps_2+(j-m)\eps_1)\ell\circ\bar F^k+\eps_1\sum_{u=1}^m \ell\circ\bar F^{k+k_u})}\right)\right)\\
&=\mathcal O\left(\sum_{m=1}^j\sum_{0\le k_1\le ...\le k_m\le k-1}\mathcal L^{k-k_m}\left(e^{\eps_1\ell}\mathcal L^{k_m-k_{m-1}}\left(e^{\eps_1\ell}\cdots\mathcal L^{k_2-k_1}\left(e^{\eps_1\ell}\mathcal L^{k_1}\left(e^{(\eps_2+(j-m)\eps_1)\ell}\mathcal L^{k}(\mathbf 1)\right)\right)\right)\right)\right)\\
&=\mathcal O\left(\sum_{m=1}^j\sum_{0\le k_1\le ...\le k_m\le k-1}\mathcal L^{k-k_m}\left(e^{\eps_1\ell}\mathcal L^{k_m-k_{m-1}}\left(e^{\eps_1\ell}\cdots\mathcal L^{k_2-k_1}\left(e^{\eps_1\ell}\mathcal L^{k_1}\left(e^{(\eps_2+(j-m)\eps_1)\ell}\right)\right)\right)\right)\right)\, .
\end{align*}
Set $\mathcal W_\theta:=\mathcal B_{\beta_1^\gamma,\theta,\theta+\gamma\eps_1}$. Recall the following facts:
\begin{itemize}
\item the function $e^{(\eps_2+(j-m)\eps_1)\ell}$ is in $\mathcal W _{\eps_2+(j-m)\eps_1}$,
\item Lemma~\ref{lem:product} ensures that the multiplication by $e^{\eps_1\ell}$ is a continuous linear map from $\mathcal W_{\eps_2+(j-m+u)\eps_1}$ to $\mathcal W_{\eps_2+(j-m+u+1)\eps_1}$ for every $u=0,...,j-1$ since $\eps_2+j\eps_1\le\eps_0$,
\item Lemma~\ref{continuousaction} ensures that $\mathcal L$ is a continuous linear operator on $\mathcal W_\theta$ for any $\theta\ge0$ such that $\theta+\gamma\eps_1\le\eps_0$.
\end{itemize}
From which we conclude that
$|\cL^{2k}(\tilde{H}^{(j)}_{k,s})|_{\eps_2+j\eps_1,\infty}=\mathcal O (k^j)$.


Recall that \eqref{bijpreimages} holds true and so, for any $x,y\in\bar\Delta$ such that
$\hat s(x,y)>1$, there exists a bijection $L_{x,y,r} : \bar F^{-r}(\{x\}) \to  \bar F^{-r}(\{y\}) $ such that for all $z\in\bar F^{-
r}(\{x\})$, $\hat s(z,L_{x,y,r}(z))> r$,
with $L_{x,y,r}$ being defined inductively on $r\ge 1$ by $L_{x,y,1}=W_{x,y}$ and $L_{x,y,r+1}(z)= W_{\bar F(z),L_{x,y,r}(\bar F(z))}(z)$.

Recall also that, for all $z \in  \bar F^{-r}(\{x\})$,
$\gamma_{k,s}(z):=E_{2k}[H\circ F^k \cdot e^{is\chi \circ F^k }]$ 
is invariant by composition by $L_{x,y,2k}$. Let us write
$\psi_{k,s}(z)=e^{is\bar{S}_{k}(z)},$
we have 
\begin{align*}
&(\cL^{2k}_{is}\bar{h}_{k,s})^{(j)}(x)-(\cL^{2k}_{is}\bar{h}_{k,s})^{(j)}(y)= \cL^{2k}(\tilde{H}^{(j)}_{k,s})(x)-\cL^{2k}(\tilde{H}^{(j)}_{k,s})(y) \\
&\le j! \sum_{m=0}^j \sum_{0\le k_1\le\cdots\le k_m\le k}
\sum_{
z\in \bar F^{-2k}(\{x\})} \left\vert A_{k,k_1,...,k_m}(z)-A_{k,k_1,...,k_m}(L_{x,y,2k}(z))
\right\vert \, \left\vert\gamma_{k,s}^{(j-m)}(z)\right\vert\, .
\end{align*}
with $A_{k,k_1,...,k_m}:=e^{S_{2k}^{(\bar{g})}(z)}\left(\psi_{k,s}(z)\prod_{u=1}^{m}\bar\phi\circ \bar F^{k_u}\right)\circ \bar F^k$.
Observe that
$$|\gamma_{k,s}^{(j-m)}|\le \mathcal O\left(e^{(\eps_2+(j-m)\eps_1)\ell\circ\bar F^k}\right) .$$
Also, as seen in \eqref{LEM0} and using \eqref{regulsumg}
\begin{align*}
&\Big|A_{k,k_1,...,k_m}-A_{k,k_1,...,k_m} (\bar F^k(L_{x,y,2k}(z)))\Big|  \\ 
&\leq e^{S_{2k}^{(\bar{g})}(z)}  e^{\bar C_{\bar g}\frac{\beta_0^{\hat s(x,y)}}{1-\beta_0}}\frac{\beta_0^{\hat s(x,y)}}{1-\beta_0}
\left|\left(\psi_{k,s}\prod_{u=0}^{m}\bar\phi\circ \bar F^{k_u}\right)\circ \bar F^k(z) \right|
\\
&\,\,\,\,\,\,\,+e^{S_{2k}^{(\bar{g})}(L_{x,y,2k}(z))}\left|\left(\psi_{k,s}\prod_{u=1}^{m}\bar\phi\circ \bar F^{k_u}\right)\circ \bar F^k(z)-\left(\psi_{k,s}\prod_{u=1}^{m}\bar\phi\circ \bar F^{k_u}\right)\circ \bar F^k(L_{x,y,2k}(z))\right|\, .
\end{align*}
Moreover
\begin{align*}
&\left|\left(\psi_{k,s}\prod_{u=1}^{m}\bar\phi\circ \bar F^{k_u}\right)\circ \bar F^k(z)-\left(\psi_{k,s}\prod_{u=1}^{m}\bar\phi\circ \bar F^{k_u}\right)\circ \bar F^k(L_{x,y,2k}(z))\right|\\
&\le \left|\prod_{u=1}^{m}\bar\phi\circ \bar F^{k_u+k}(z)-\prod_{u=1}^{m}\bar\phi\circ \bar F^{k_u+k}(L_{x,y,2k}(z))\right|\\ 
& \phantom{aaaaaaaaaa}+\left|\psi_{k,s}(\bar F^k(z))-\psi_{k,s}(\bar F^k(L_{x,y,2k}(z)))\right|\, \prod_{u=1}^{m}\left|\bar\phi\circ \bar F^{k_u+k}(z)\right|\\
&\le \Vert\bar\phi\Vert^m_{\mathcal B_{\beta_1,\eps_1,\eps_1}} e^{\eps_1\sum_{v=1}^m\ell\circ\bar F^{k_v+k}}\left(\sum_{u=1}^m \beta_1^{\hat s(x,y)+k-k_u}+
 \sum_{w=0}^{k-1}\min\left(2, s\Vert \bar\phi\Vert_{\mathcal B_{\beta_1,\eps_1,\eps_1}} \beta_1^{\hat s(x,y)+k-w} \right)\right)\\
&\le(1+ \Vert\bar\phi\Vert)^{m+1}_{\mathcal B_{\beta_1,\eps_1,\eps_1}}(1+|s|) e^{\eps_1\sum_{v=1}^m\ell\circ\bar F^{k_v+k}} \frac{\beta_1^{\hat s(x,y)}}{1-\beta_1}\, .
\end{align*}
Hence, for every $x,y\in\bar\Delta_{l,j}$
\begin{align*}
&(\cL^{2k}_{is}\bar{h}_{k,s})^{(j)}(x)-(\cL^{2k}_{is}\bar{h}_{k,s})^{(j)}(y)\\
&=\mathcal O\left(\beta_1^{\gamma\hat s(x,y)}(1+|s|)\max_{z\in\{x,y\}}\mathcal L^{2k}\left(e^{(\eps_2+(j-m)\eps_1)\ell\circ\bar F^k+\eps_1\sum_{v=1}^m\ell\circ\bar F^{k_v+k}}\right)(z)\right)\, .
\end{align*}
To conclude, we prove that
\[
\left| \mathcal L^{2k}\left(e^{(\eps_2+(j-m)\eps_1)\ell\circ\bar F^k+\eps_1\sum_{v=1}^m\ell\circ\bar F^{k_v+k}}\right)\right|_{\eps_2+j\eps_1,\infty}=\mathcal O(1)\, ,
\]
as we did in the proof of $|\cL^{2k}(\tilde{H}^{(j)}_{k,s})|_{\eps_2+j\eps_1,\infty}=\mathcal O (k^j)$.
\end{proof}
\end{appendix}

\end{document}

%% file: abbrev.tex
\def\eps{{\varepsilon}}

\def\Prob{{\mathbb{P}}}

\def\EXP{{\mathbb{E}}}

\def\complex{\mathbb{C}}

\def\bbH{\mathbb{H}}

\def\bbT{\mathbb{T}}

\def\bbX{\mathbb{X}}

\def\naturals{\mathbb{N}}

\def\reals{\mathbb{R}}

\def\integers{\mathbb{Z}}

\def\cA{\mathcal{A}}

\def\cB{\mathcal{B}}

\def\cE{\mathcal{E}}

\def\cK{\mathcal{K}}

\def\cL{\mathcal{L}}

\def\cM{\mathcal{M}}

\def\cO{\mathcal{O}}

\def\cP{\mathcal{P}}

\def\cV{\mathcal{V}}

\def\cX{\mathcal{X}}

\def\cZ{\mathcal{Z}}

\def\fF{\mathfrak{F}}

\def\fN{\mathfrak{N}}

\def\fn{\mathfrak{n}}

\makeatother

\def\beq{\begin{equation}}
\def\eeq{\end{equation}}